\documentclass[a4paper,11pt,british]{article}
\usepackage{setspace,graphicx,
epstopdf,amsmath,amsfonts,amsgen, mathtools,
amstext,amsthm,amsbsy,amsopn,amssymb,
tikz,
parskip,verbatim,mathrsfs,enumerate,
xcolor,comment}  
\usepackage[utf8]{inputenc}   
\usepackage[top=2.0cm, bottom=2.15cm, left=2.0cm, right=2.0cm]{geometry}
\usepackage[square,numbers]{natbib} 
\usepackage[pdfborder={0 0 0}]{hyperref}

\usetikzlibrary{arrows.meta}

\def\b0{\boldsymbol{0}}

 \newcommand{\C}     {\mathbb{C}} 
\newcommand{\R}     {\mathbb{R}} 
\newcommand{\Z}     {\mathbb{Z}} 
\newcommand{\N}     {\mathbb{N}}

\newcommand{\T}     {\mathbb{T}}

\newcommand{\Exp}{\mathscr{E}\kern-0.2mm{\operatorname{xp}}}
\newcommand{\Log}{\mathscr{L}\kern-0.2mm{\operatorname{og}}}

\def\1{{\mathchoice {1\mskip-4mu\mathrm l}      
{1\mskip-4mu\mathrm l} 
{1\mskip-4.5mu\mathrm l} {1\mskip-5mu\mathrm l}}}

\numberwithin{equation}{section}
\numberwithin{figure}{section}
\newtheoremstyle{plain}
  {6pt}
  {4pt}
  {\slshape}
  {}
  {\bfseries}
  {.}
  {0.5em}
  {}%
\newtheorem{thm}{\protect\theoremname}[section]
\newtheorem{defn}{\protect\definitionname}[section]
  
  \newtheorem{prop}[thm]{\protect\propositionname}
  \newtheorem{rem}[thm]{\protect\remarkname}
  \newtheorem{cor}[thm]{\protect\corollaryname}
  \newtheorem{lem}[thm]{\protect\lemmaname}
  \numberwithin{thm}{section}

\usepackage[british]{babel}

\providecommand{\theoremname}{Theorem}
\providecommand{\definitionname}{Definition}
\providecommand{\factname}{Fact}
\providecommand{\propositionname}{Proposition}
\providecommand{\remarkname}{Remark}
\providecommand{\corollaryname}{Corollary}
\providecommand{\lemmaname}{Lemma}

\usepackage{authblk}
\title{Dimer model and random lattice permutations \\ with general weights}

\author[1]{Andreas  Klippel}
\author[2]{Lorenzo Taggi}
\author[3]{Wei Wu}
\affil[1]{\small{Johannes Gutenberg-Universität Mainz, Mathematik Fachbereich, Mainz, Germany}}
\affil[2]{\small{Sapienza Universit\`a di Roma, Dipartimento di Matematica, Roma, Italy}}
\affil[3]{\small{NYU Shanghai,  Mathematics Department and NYU-ECNU Math Institute, Shanghai, China}}
\date{\today}

\begin{document}

\maketitle

\begin{abstract}
The dimer model and random lattice permutations are two fundamental
objects at the interface of probability, combinatorics, and mathematical
physics.
We study these models  on finite periodic boxes in $\mathbb Z^d$
within a common framework.

For the dimer model,  edges connecting arbitrary vertices
carry a weight which depends on their {relative displacement}
 and dimer configurations are weighted through their occupied edges. 
 Superimposing two independent
perfect matchings gives the double-dimer model, whose configurations are
collections of disjoint loops.
Permutations, instead,  are weighted through the spatial displacement 
of their jumps.

{For broad classes of weights of finite or infinite range we prove long-range order 
and the occurrence of macroscopic loops. 
This extends nearest-neighbour results to
arbitrary-range edges and jumps. 
In particular,  long-range  weights yield long-range order and macroscopic loops already in dimensions $d=1,2$.}
In dimension two, this behaviour is
qualitatively different from that of the nearest-neighbour model.

We complement these results with sharp absence criteria.
\end{abstract}

\section{Introduction}
Perfect matchings and permutations are fundamental discrete structures.
Probability measures on perfect matchings give rise to dimer-type models,
while weighting permutation jumps according to their spatial displacement
leads to spatial random permutations. 
Superimposing two independent
perfect matchings gives the double-dimer model, whose configurations are
collections of disjoint loops.
These  families connect combinatorics and
probability with statistical mechanics and mathematical physics. Planar
dimer models, in particular, have deep connections with discrete complex
analysis, random surfaces, and conformally invariant probability, whereas
spatial random permutations are related to the interacting Bose gas.

This paper studies long-range order and the presence of macroscopic loops
in dimer-type models and random lattice permutations with spatial weights
of bounded or unbounded range. 
Previous work has focused mainly on
nearest-neighbour or planar settings. Allowing dimer edges and permutation
jumps of arbitrary length changes the picture qualitatively. For broad
classes of translation-invariant weights, we prove long-range order and,
in the bipartite setting, the occurrence of macroscopic loops and
permutation cycles.
Our results hold in arbitrary dimension and include long-range
order and macroscopic loops already in dimensions \(d=1,2\). 

We also
identify complementary mechanisms preventing long-range order, governed
by the weight assigned to monomers or fixed points and, in dimension one,
by the tail of the spatial weights.

\paragraph{Dimer and double-dimer models.}
The dimer model is a fundamental object in probability theory,
statistical mechanics, and combinatorics. On planar graphs it is by now
well understood, mainly thanks to Pfaffian techniques introduced
independently by Kasteleyn and by Temperley and Fisher
\cite{Kasteleyn,Temperley}. These techniques allow exact computations and
have led to a detailed description of correlations and scaling limits;
see, for example,
\cite{Kenyon1997,Kenyon2001,KenyonOkounkovSheffield2006}. Beyond
planarity, however, much less is known, and basic questions about
correlations, loop structure, and scaling limits remain largely open.
Recent developments include a large deviation principle for the flow
function of three-dimensional dimer configurations \cite{Chandgotia},
higher-dimensional analogues of the honeycomb dimer model
\cite{Lammers21}, macroscopic double-dimer loops
\cite{QuitmannTaggi2}, positive monomer correlations
\cite{T}, the limiting shape of the multinomial dimer model
\cite{KenyonWolfram}, and local dimer dynamics in higher dimensions
\cite{HartarskyLichevToninelli}; see also
\cite{GiulianiToninelli,TaggiWu} for results in low-dimensional
non-planar settings.

We consider dimer and double-dimer models on \(\mathbb Z^d\)-like graphs
with edges joining arbitrary pairs of vertices and with weights depending
on their displacement. Our main observables are the monomer--monomer
correlation and the loop structure of the associated double-dimer model,
obtained by superimposing two independent dimer configurations. On
\(\mathbb Z^2\), decay of monomer correlations was conjectured by Fisher
and Stephenson \cite{FisherStephenson1963} and later proved rigorously by
Dub\'edat \cite{Dub}; polynomial decay of double-dimer loop connections
was established in \cite{Dubedat}, while conformal invariance of planar
double-dimer loops was proved in \cite{Kenyondoubledimer}. By contrast,
in dimensions \(d\geq3\), uniformly positive monomer correlations and
macroscopic double-dimer loops were established in
\cite{T,QuitmannTaggi2}. These results were restricted to
nearest-neighbour models.

Our results show that these phenomena persist for much broader classes of
weights. For bipartite OS-positive weights satisfying an explicit
Green-function condition, we prove uniform positivity of monomer
correlations and the occurrence of macroscopic loops. Our criteria apply
to natural families of polynomially and exponentially decaying weights.
More importantly, sufficiently spread-out long-range weights yield
long-range order and macroscopic loops even in dimensions \(d=1,2\). In
dimension two, this behaviour is qualitatively different from that of
the nearest-neighbour model.

\paragraph{Monomer double-dimer model and random lattice permutations.}
The monomer double-dimer model (MDD) extends the double-dimer model by
allowing monomers, whose density is controlled by the monomer activity.
It admits a natural interpretation through a duplicated graph. Consider
two copies of \(G\), with a vertical edge joining the two copies of each
vertex. A dimer configuration on the duplicated graph projects onto a
pair of dimer configurations on \(G\) having the same set of monomers,
with a vertical dimer interpreted as a monomer at the corresponding
vertex. Thus, the projected configuration is precisely an MDD
configuration. In this representation the vertical edges generally
destroy planarity even when \(G\) is planar. When their weight is zero,
the two dimer configurations are independent and the ordinary
double-dimer model is recovered. Classical Pfaffian techniques therefore
do not directly apply to the general MDD model; see also \cite{Jerrum}.

Random permutations and their cycle structures are classical objects in
probability and combinatorics, studied under the uniform law, the Ewens
family, and more general weighted measures; see, for example,
\cite{ArratiaTavare1992,BetzUeltschiVelenik2011}. 
Spatial random
permutations form a natural geometric class of non-uniform random
permutations, in which the weight assigned to a jump depends on the
locations of its endpoints. Here we consider a lattice  version of this construction, which we call the random
lattice-permutation (RLP) model. 
Given the vertex set \(V\) of a finite
periodic box, a bijection $
\pi:V\longrightarrow V
$
is sampled with probability proportional to
$
\prod_{x\in V} w_{x,\pi(x)},
$
where \(w_{x,y}\) are some weights depending only on the 
 displacement from the vertex \(x\) to the vertex \(y\). 
 Every permutation
decomposes into disjoint oriented cycles, with fixed points corresponding
to cycles of length one. In particular, \(w_{x,x}=w(0)\) is the weight
assigned to a fixed point.

Spatial random permutations are motivated in part by their connection
with the Bose gas, where infinite permutation cycles are closely related
to Bose--Einstein condensation; see, among others,
\cite{BetzUeltschiPD2011,ElboimPeled2019,
BetzUeltschi2008,BetzUeltschi2011}. In the lattice model considered here,
the underlying point configuration is fixed, and bijectivity imposes a
rigid hard-core constraint: the cycles are mutually vertex-disjoint and
cover all lattice sites. Related random lattice-permutation models have
been studied in \cite{Betz2014,BiskupRichthammer2015}, while the
formation of macroscopic cycles was investigated numerically in
\cite{Grosskinsky2013}. They also arise as the hard-core case of the
interacting random-walk loop models considered in \cite{QuitmannTaggi};
the macroscopic-loop result proved there requires a sufficiently large
local-time cutoff and therefore does not cover RLP.

When the underlying graph is bipartite and all nonzero weights between
distinct vertices connect opposite bipartition classes, the MDD and RLP
models are equivalent, with monomers corresponding to fixed points of
the permutation. Our
results also cover non-bipartite weights, for which this correspondence
no longer holds.

\paragraph{Main results for MDD and RLP.}
Theorems~\ref{thm:longrangeorder} and
\ref{thm:longrangeorder2} give two general sufficient conditions for
long-range order. The first applies to bipartite weights and additionally
yields macroscopic loops, whereas the second applies to general weights
under an analytic condition involving their Fourier transform. Both
conditions are formulated quantitatively in terms of the Green function
at the origin of the random walk associated with the weights. They apply
to broad families of long-range interactions and yield ordered regimes
in dimensions where nearest-neighbour models do not exhibit long-range
order.

Our second main result concerns the absence of long-range order and holds
in considerable generality.
Theorem~\ref{thm:decay-large-monomer} shows that if
\[
w(0)>\frac12,
\]
then long-range order is impossible.
Here \(w(0)\) is the monomer activity in the MDD model and the weight
associated with fixed points in the RLP model.
The threshold \(1/2\) is optimal in a universal sense for random lattice
permutations: for every \(w(0)<1/2\), we construct examples exhibiting
long-range order.

Our third and fourth main results concern dimension one, where together
they provide a particularly precise picture.
Theorem~\ref{thm:positive-activity-one-dimensional} proves absence of
long-range order for lazy-bipartite weights with finite first moment and
$w(0)>0$.

Our fourth main result, Theorem~\ref{thm:zero-activity-one-dimensional},
shows that the positive-activity assumption cannot be dropped: at zero
activity, long-range order and macroscopic loops occur even under a
finite-first-moment assumption.

Our third main result concerns dimension one, where two complementary
theorems yield a particularly precise picture.
Theorem~\ref{thm:positive-activity-one-dimensional} proves absence of
long-range order for lazy-bipartite weights with finite first moment and
$w(0)>0$.
Theorem~\ref{thm:zero-activity-one-dimensional} shows that the
positive-activity assumption is essential: at zero activity, long-range
order and macroscopic loops occur even under a finite-first-moment
assumption.
Moreover, when the first moment is infinite, our long-range examples show
that long-range order may occur even at positive activity.
Thus, in dimension one, the presence or absence of long-range order is
governed rather precisely by the interplay between the activity and the
tail of the weights.

\paragraph{Proof techniques.}
A crucial ingredient in our proofs of long-range order is the complex
spin representation introduced in \cite{TaggiWu} for  nearest neighbour weights. 
One of the main technical contributions of
the present work is to extend this framework to weights of arbitrary
range.
We show that OS-positivity of the weights provides a natural
condition under which the associated spin measure remains reflection
positive, also in the presence of long-range edges. We can therefore
apply Gaussian domination and derive an infrared bound directly in the
spin framework, following the classical reflection-positivity approach
of \cite{FrohlichSpencer,FrohlichLiebSimon}.

In the nearest-neighbour setting, reflection positivity had already been
used in \cite{QuitmannTaggi,T}, without relying on the spin
representation, starting from a reflection-positivity property
formulated directly in path space and first observed in
\cite{ChayesPryadkoShtengel}. The argument requires that, conditional on
every realization of the connections crossing the reflection plane, the
configurations induced on the two half-tori be independent and, after
reflecting one half-torus onto the other, have the same conditional law.

No analogous nontrivial decomposition is generally available for
unbounded-range edges. If one conditions only on some of the crossing
connections, additional unconditioned long edges may still couple the
two half-tori, so conditional independence fails. 
A further obstruction arises already for finite-range weights extending
beyond nearest neighbours: even when the two halves can be separated,
their conditional laws need not coincide after reflection. The spin
representation bypasses these obstructions, substantially enlarges the
class of models to which the infrared-bound method applies, and provides
a direct derivation of the infrared bound entirely within the spin
framework.

Once the infrared bound has been established, however, deducing
long-range order from it differs from the classical spin-system
argument. In an \(O(N)\) model, the identity
\(\langle (S_x^1)^2\rangle=1/N\) provides a fixed local normalization
which, together with the infrared bound, yields long-range order. No
analogous useful normalization is available here. 
We instead use the normalization provided
by edge and jump probabilities, which leads to a weighted spatial
average of the two-point function. In Fourier space, this introduces the
Fourier transform of the weights as a multiplier, and this multiplier
may change sign. We overcome the resulting obstruction by pairing
Fourier modes in the bipartite case and by a diagonal-shift argument for
general weights; see
Section~\ref{sect:prooftheorems} for the precise Fourier identities and
estimates.

Our proofs of absence of long-range order are of a different nature.
The criterion \(w(0)>1/2\) follows from a simple energy--entropy argument,
in which the threshold \(1/2\) emerges naturally. The one-dimensional result follows from a decomposition
of pairs of interval permutations at their common free cuts.
Under the finite-first-moment assumption on the lazy-bipartite weights, long-range order
is ruled out independently of the value of the positive monomer
activity.

Finally, our approach to long-range order also applies to models in which
the loop configuration may visit a vertex more than once, as in the
broader framework considered in \cite{QuitmannTaggi}. We do not pursue
this additional generality, in order to keep the proofs transparent and
focus on the three principal models considered above.

\section{Definitions and main results}
We now introduce the  models to which our results apply: the dimer model, the double-dimer model,   the monomer double-dimer model,  and random lattice permutations. 
We introduce these models on {a finite undirected simple graph $G=(V,E)$, equipped with symmetric non-negative weights}
$w=(w_{x,y})_{x,y\in V}$. {Throughout, $E$ contains every unordered pair $\{x,y\}$ with $x\neq y$ and $w_{x,y}>0$.}

\subsection{Dimer and double-dimer model}
A \textit{perfect matching} in $G$ is a subset $D \subset E$ such that each vertex of the graph $(V,D)$ has degree precisely one. 
We let $\mathcal{D}_G$ be the set of  perfect matchings in $G$.
Suppose that $G$ admits a perfect matching of positive weight. 
Given a set $A \subset V$, we let $G_A$ be the {subgraph} of $G$ with vertex set
$V \setminus A$ and with edge set consisting of all the edges  in $E$ which 
do not touch any vertex in $A$.
We define 
$$
Z^{dim}_{G,w} (A) := \sum\limits_{D  \in \mathcal{D}_{G_A}}   \prod_{  \{x,y\} \in D    } w_{x,y}
$$
and set $Z^{dim}_{G,w} = Z^{\mathrm{dim}}_{G,w} (\emptyset)$,
$Z_{G,w}^{dim}(x,y) = Z_{G,w}^{dim}(\{x,y\})$,
$Z_{G,w}^{dim}(x,x) = Z_{G,w}^{dim}(\{x\})$.

\begin{defn}[Dimer model]
The \textit{dimer model} is the probability measure ${\mathbb P_{G,w}^{\mathrm{dim}}}$ on $\mathcal{D}_G$
assigning to each perfect matching  $D \in \mathcal{D}_G$
the weight
$$
{\mathbb P_{G,w}^{\mathrm{dim}}}(D)  =   \frac{\prod_{  \{x,y\} \in D    } w_{x,y}}{ Z^{\mathrm{dim}}_{G,w} }
$$
\end{defn}
An important quantity in the analysis of the model is the monomer correlation function,
corresponding to the ratio of the weight of configurations in $G$ {after} the removal of two vertices 
and the weight of all configurations in $G$.
\begin{defn}[Monomer correlation function]
We define for each pair of vertices $x, y \in V$ the monomer correlation function
\begin{equation}\label{eq:monomerdimer}
 \mathcal{C}_{G,w}(x,y) := \frac{  {Z}^{dim}_{G,w}(x,y) }{  Z^{\mathrm{dim}}_{G,w}}.
\end{equation}
\end{defn}

The \emph{double-dimer model} is defined as the superposition of two independent dimer
configurations. The resulting configuration decomposes into a collection of mutually
disjoint, vertex-self-avoiding loops. {We denote this product measure by
$\mathbb P^{\mathrm{d.d.}}_{G,w}$ and its partition function by
$Z^{\mathrm{d.d.}}_{G,w}:=(Z^{\mathrm{dim}}_{G,w})^2$.}

\subsection{Monomer double-dimer model}
The configuration space is
\[
\Omega := \bigl\{ (m^1,m^2) \in \{0,1\}^E \times \{0,1\}^E \; : \;
n_x(m^1) = n_x(m^2) \in \{0,1\} \ \text{for all } x \in V \bigr\},
\]
where the vector $m^1$ represents the occupation of the \emph{blue dimers},
the vector $m^2$ represents the occupation of edges by \emph{red dimers},
and for each $x \in V$ and $m \in (\mathbb{N}_0)^E$ we define the \emph{local time}
\[
n_x(m) := \sum_{\substack{y\in V:\\ \{x,y\}\in E}}m_{x,y}{.}
\]
Thus, admissible configurations are such that each vertex is incident to the same number of blue and red dimers, and this number is either one or zero. In the latter case, we say that the vertex carries a \emph{monomer}.
We now introduce a probability measure on $\Omega$.

\begin{defn}[Monomer double-dimer model]
\label{def:monomerdoubledimer}
For each $m=(m^1,m^2)\in\Omega$, we define
\[
{\mathbb P^{\mathrm{mdd}}_{G,w}(m)}
:=\frac{1}{{Z^{\mathrm{mdd}}_{G,w}}}
\prod_{x\in V} w_{x,x}^{\,1-n_x(m^1)}
\prod_{\{x,y\}\in E} w_{x,y}^{\,m_{x,y}^1+m_{x,y}^2},
\]
where ${Z^{\mathrm{mdd}}_{G,w}}$ is a normalizing constant.
\end{defn}

In other words,  the weight of a configuration factorizes as follows: each monomer at $x$ contributes a factor $w_{x,x}$, while each blue or red dimer contributes a factor given by the weight of the edge on which it is located.

\begin{defn}[Loops and loop length]
For a configuration $m=(m^1,m^2)\in\Omega$, consider the multigraph on the vertex set $V$ in which each edge $\{u,v\}\in E$ appears with multiplicity
$
m^1_{u,v}+m^2_{u,v}\in\{0,1,2\}.
$
The multigraph decomposes into a disjoint union of cycles (loops) and isolated vertices (monomers).
For $x\in V$, we define ${\mathcal L_x(m)}$ {to be} \textit{the loop touching $x$}. If $x$ is a monomer, we set
$
{\mathcal L_x(m)}=\{x\}.
$
We define the length of ${\mathcal L_x(m)}$, ${|\mathcal L_x|(m)}$ as \textit{the total number of dimers} (blue and red) in the loop. In particular,
$
{|\mathcal L_x|(m)}=0$ if and only if $x$ is a monomer. 
\end{defn}

{We write $x\leftrightarrow y$ when $x$ and $y$ belong to the same loop.}

In order to define the two-point function we introduce the set $\Omega_{x,y}$
as follows.  For each $x, y \in V$ {with $x\neq y$}, we set  
\begin{multline}
\Omega_{x,y}
:=
\Bigl\{(m^1,m^2)\in\{0,1\}^E\times\{0,1\}^E:\;
n_z(m^1)=n_z(m^2)\in\{0,1\}
\ \text{for all } z\in V\setminus\{x,y\},\\
n_x(m^1)=n_y(m^1)=0,
\qquad
n_x(m^2)=n_y(m^2)=1
\Bigr\}.
\end{multline}
In other words, configurations in $\Omega_{x,y}$ satisfy the usual monomer
double--dimer constraint at all vertices except at $x$ and $y$, where we prescribe
additional {\emph{defect constraints}}.
More precisely,  the vertices $x$ and $y$ are treated as
monomers for the first colour and are covered by dimers of the second colour.
As a consequence, the superposition of the two colours consists of a collection
of mutually disjoint loops together with a single self-avoiding walk connecting
$x$ and $y$.
See also Figure \ref{fig:twopoint-longrange-odd-8x8-threepanel}
for an example.
\begin{figure}[htbp]
\centering
\begin{tikzpicture}[
    scale=0.55, 
    dot/.style={circle, fill=black, inner sep=1.1pt},
    monomer/.style={circle, draw=black, fill=white, inner sep=1.8pt, line width=0.8pt},
    bluebond/.style={-, blue, very thick, shorten >=3.5pt, shorten <=3.5pt},
    redbond/.style={-, red, very thick, densely dotted, shorten >=3.5pt, shorten <=3.5pt},
    bluejump/.style={-, blue, very thick, shorten >=3.5pt, shorten <=3.5pt},
    redjump/.style={-, red, very thick, densely dotted, shorten >=3.5pt, shorten <=3.5pt},
    lab/.style={font=\scriptsize}
]

\def\DrawGrid{
    \foreach \x in {0,...,7}{
        \foreach \y in {0,...,7}{
            \node[dot] at (\x,\y) {};
        }
    }
}

\def\DefinePoints{
    \coordinate (xpt) at (0,1);
    \coordinate (ypt) at (7,1);
    \coordinate (mone) at (2,6);
    \coordinate (mtwo) at (5,6);
}

\def\DrawLabels{
    \node[lab, below left] at (xpt) {$x$};
    \node[lab, below right] at (ypt) {$y$};
    \node[lab, above] at (mone) {$m_1$};
    \node[lab, above] at (mtwo) {$m_2$};
}

\def\DrawCommonMonomers{
    \node[monomer] at (mone) {};
    \node[monomer] at (mtwo) {};
}

\def\DrawDoubledHoriz{
    \draw (0,0)--(1,0); \draw (2,0)--(3,0); \draw (4,0)--(5,0); \draw (6,0)--(7,0);
    \draw (2,1)--(3,1); \draw (4,1)--(5,1);
    \draw (2,2)--(3,2); \draw (4,2)--(5,2); \draw (6,2)--(7,2);
    \draw (0,3)--(1,3); \draw (2,3)--(3,3); \draw (4,3)--(5,3); \draw (6,3)--(7,3);
    \draw (1,4)--(2,4); \draw (3,4)--(4,4); \draw (5,4)--(6,4);
    \draw (2,5)--(3,5); \draw (4,5)--(5,5);
    \draw (3,6)--(4,6);
    \draw (2,7)--(3,7); \draw (4,7)--(5,7);
}

\def\DrawDoubledVert{
    \draw (0,4)--(0,5); \draw (7,4)--(7,5);
    \draw (1,5)--(1,6); \draw (6,5)--(6,6);
}

\def\DrawBlueAlone{
    \draw[bluebond] (0,2)--(1,2);
    \draw[bluejump] (1,1) to[bend right=15] (6,1); 
    \draw[bluebond] (0,7)--(1,7);
    \draw[bluebond] (6,7)--(7,7);
    \draw[bluejump] (7,6) to[bend left=15] (0,6); 
}

\def\DrawRedAlone{
    \draw[redbond] (0,1)--(0,2);
    \draw[redbond] (1,2)--(1,1);
    \draw[redbond] (6,1)--(7,1);
    \draw[redbond] (0,6)--(0,7);
    \draw[redbond] (7,7)--(7,6);
    \draw[redjump] (1,7) to[bend left=15] (6,7); 
}

\begin{scope}[xshift=0cm]
    \DefinePoints
    \begin{scope}[bluebond] \DrawDoubledHoriz \DrawDoubledVert \end{scope}
    \DrawBlueAlone
    \DrawGrid
    \DrawCommonMonomers
    \node[monomer] at (xpt) {};
    \node[monomer] at (ypt) {};
    \DrawLabels
    \node[below, font=\scriptsize, align=center] at (3.5,-0.7)
    {(a) Blue dimer configuration\\ monomers \(M\cup\{x,y\}\)};
\end{scope}

\begin{scope}[xshift=8.5cm]
    \DefinePoints
    \begin{scope}[redbond] \DrawDoubledHoriz \DrawDoubledVert \end{scope}
    \DrawRedAlone
    \DrawGrid
    \DrawCommonMonomers
    \DrawLabels
    \node[below, font=\scriptsize, align=center] at (3.5,-0.7)
    {(b) Red dimer configuration\\ monomers \(M\)};
\end{scope}

\begin{scope}[xshift=17.0cm]
    \DefinePoints
    
    \begin{scope}[yshift=2pt, bluebond] \DrawDoubledHoriz \end{scope}
    \begin{scope}[yshift=-2pt, redbond] \DrawDoubledHoriz \end{scope}
    
    \begin{scope}[xshift=-2pt, bluebond] \DrawDoubledVert \end{scope}
    \begin{scope}[xshift=2pt, redbond] \DrawDoubledVert \end{scope}
    
    \DrawBlueAlone
    \DrawRedAlone
    
    \DrawGrid
    \DrawCommonMonomers
    \node[monomer] at (xpt) {};
    \node[monomer] at (ypt) {};
    \DrawLabels
    
    \node[below, font=\scriptsize, align=center] at (3.5,-0.7)
    {(c) Superposition in \(\Omega_{x,y}\):\\ loops and one path from \(x\) to \(y\)};
\end{scope}

\end{tikzpicture}
\caption{Representation of a configuration in $\Omega_{x,y}$ on an $8 \times 8$ grid {when} the distance between $x$ and $y$ is odd.
(a, b) {dimer configurations} with {monomer sets $M\cup\{x,y\}$ and $M$, respectively}. (c) Their superposition yields an element of $\Omega_{x,y}$, consisting of doubled dimers, an alternating loop, and exactly one alternating self-avoiding path connecting $x$ and $y$. }
\label{fig:twopoint-longrange-odd-8x8-threepanel}
\end{figure}
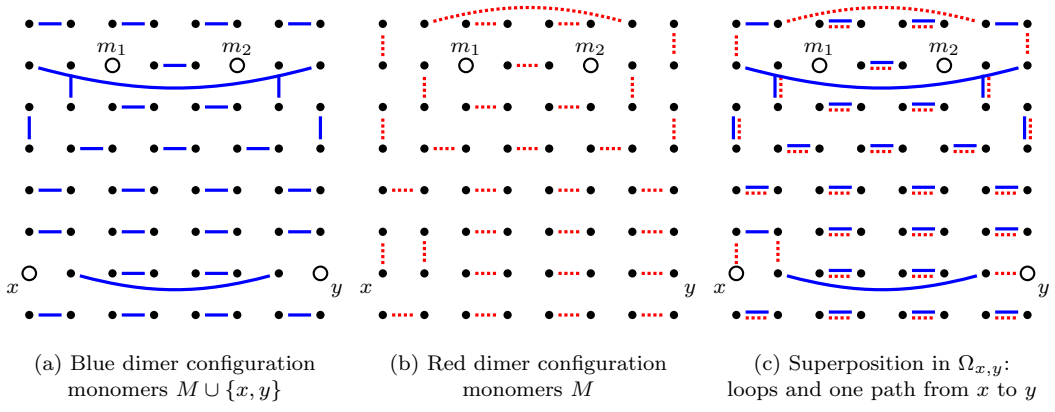

\begin{defn}[Two-point function]
\label{def:twopoint-mdd}
For each pair $x, y \in V$  with $x \neq y$,
we define the two-point function as the ratio
\[
G^{ \mathrm{mdd}}_{G,w}(x,y)
\;:=\;
\frac{Z^{ \mathrm{mdd}}_{G ,w}(x,y)}{Z^{ \mathrm{mdd}}_{G,w}},
\]
where
\[
Z^{ \mathrm{mdd}  }_{G ,w}(x,y)
\;:=\;
\sum_{(m^1,m^2)\in\Omega_{x,y}}
\prod_{z\in V \setminus \{x,y\}} w_{z,z}^{\,1-n_z(m^1)}
\prod_{\{u,v\}\in E} w_{u,v}^{\,m^1_{u,v}+m^2_{u,v}}
\]
and $Z^{ \mathrm{mdd}  }_{G,w}$ is the partition function of the monomer double-dimer model defined
previously.
If $x = y$, we {set} $G^{ \mathrm{mdd}  }_{G,w}(x,x) = 0$. 
\end{defn}

\paragraph{The fully packed regime.}
The monomer double-dimer model reduces to the {double-dimer model}
in the \emph{fully packed regime}, namely in the special case 
\[
w_{x,x}=0 \qquad \text{for every } x\in V.
\]
In this regime, every vertex is covered
by a dimer of each colour almost surely
and the {law of the model}
factorises as the product of two independent
dimer measures. 
This gives for $x \neq y$ the identities
$$
Z^{ \mathrm{mdd}  }_{G,w}
=
Z^{\mathrm{d.d.}}_{G,w}
=
\bigl(Z^{\mathrm{dim}}_{G,w}\bigr)^2,
\quad \quad \quad Z^{ \mathrm{mdd}  }_{G,w}(x,y)
=
Z^{\mathrm{dim}}_{G,w}(x,y) \, Z^{\mathrm{dim}}_{G,w},
$$
which imply
\begin{equation}\label{eq:twopointandmonomer}
G^{ \mathrm{mdd}  }_{G,w}(x,y) = \mathcal{C}_{G,w}(x,y).
\end{equation}

\subsection{Random lattice permutations}  
Let $\Omega^{\mathrm{per}}$ be the set of permutations $\pi$ of the vertices  $V$.

\begin{defn}[{Random lattice-permutation model}]
For  each $\pi \in \Omega^{\mathrm{per}}$,  we assign the weight
\begin{equation}
\label{eq:measurepermutation}
{\mathbb P_{G,w}^{\mathrm{per}}}(\pi) =    
\frac{ \prod_{x \in V} w_{x,  \pi(x)}}{Z^{\mathrm{per}}_{G, w}},
\end{equation}  
   {where} $Z^{\mathrm{per}}_{G, w}$ is a normalising constant.
   Moreover, we define the loop visiting $x$  as the set
   $$
   {\mathcal L_x(\pi)} := \{ y \in V \, : \,    y = \pi^n(x) \mbox{ for some $n \in \mathbb{N}_0$} \}
   $$
   \end{defn}
   The measure 
{$\mathbb P^{\mathrm{per}}_{G,w}$} that we defined is such that,
for each permutation $\pi \in \Omega^{\mathrm{per}}$,
each vertex $x \in V$ which is mapped to $y$  (possibly with $x = y$)
gets a multiplicative weight $w_{x,y}$.
For example,  in the special case ${w_{x,y}} = 1$ for each $x, y \in V$,  we obtain random permutations with uniform weights, a
well studied and relatively simple object. 
In order to define the two-point function we define the set of bijections from $V \setminus \{y\} $
to $V \setminus \{x\}$,
\[
\Omega^{\mathrm{per}}_{x,y}
:=
\left\{
\pi:V\setminus\{y\}\to V\setminus\{x\}
\,:\,
\pi \text{ is a bijection}
\right\}.
\]
{In other words, if $x\neq y$, then $x$ and $y$ respectively play the roles of a source and a sink of a self-avoiding walk from $x$ to $y$.}
Finally we define for each pair $x, y \in V$ (with possibly $x = y$) the two-point function
\begin{equation}\label{eq:twopointrlp}
G^{\mathrm{per}}_{G,w}(x,y) := \frac{1}{Z^{\mathrm{per}}_{G, w}} \sum\limits_{\pi \in \Omega^{\mathrm{per}}_{x,y}}
\prod_{z \in V \setminus \{y\}} w_{z, \pi(z)}.
\end{equation}

\paragraph{Bipartite and lazy-bipartite weights.}
We say that the weights are \emph{{lazy-bipartite}} if there exists a
partition \(V=V^{\mathrm e}\cup V^{\mathrm o}\) such that
\begin{equation}\label{eq:lazybipartitexcondition}
w_{x,y}=0
\qquad
\text{whenever }x\neq y
\text{ and either }
\{x,y\}\subset V^{\mathrm e}
\text{ or }
\{x,y\}\subset V^{\mathrm o}.
\end{equation}
{On $\mathbb Z^d$ and on even tori, all lattice results below use the canonical parity classes, determined by the parity of $|x|_1$.}
In this case, the spatial random-permutation model and the monomer
double-dimer model are equivalent: fixed points correspond to monomers,
while non-trivial permutation cycles correspond to loops formed by the
superposition of blue and red dimers.

If, in addition,
\(w_{x,x}=0\) for every \(x\in V\), the weights are said to be
\emph{bipartite}, and both models are equivalent to the double-dimer
model. For non-bipartite weights this correspondence generally breaks
down; nevertheless, our results apply to both models independently and
do not require bipartiteness.

\section{Main results}
\label{sect:definitions}
We now introduce the class of weights considered throughout the paper
and state our main results.
For symmetric, translation-invariant, normalized weights, write
$w(x):=w_{0,x}$ and define
\[
\widehat w(k):=\sum_{x\in\mathbb Z^d}w(x)e^{\mathrm i k\cdot x},
\qquad k\in(-\pi,\pi]^d.
\]
Let $\T_L:=(\mathbb Z/L\mathbb Z)^d$ and
$\T_L^*:=(2\pi/L)\T_L$, represented in $(-\pi,\pi]^d$.
The Green's function of the random walk with transition kernel $w$ is
\[
G^{\mathrm{RW}}_{d,w}(x):=\sum_{n\ge0}w^{*n}(x),
\qquad
g_{d,w}:=G^{\mathrm{RW}}_{d,w}(0)
=\int_{(-\pi,\pi]^d}\frac1{1-\widehat w(k)}
\frac{\mathrm dk}{(2\pi)^d}\in[1,\infty].
\]
\begin{defn}[Admissible weights]
\label{def:admissible}
A function $w:\mathbb Z^d\times \mathbb Z^d \to [0,\infty)$ is called
\emph{admissible} if the following conditions are fulfilled:
\begin{enumerate}[(i)]
\item (symmetry) $w_{x,y}=w_{y,x}$ for all $x,y\in\mathbb Z^d$,
\item (translation invariance) $w_{x,y}=w_{x+z,y+z}$ for all $x,y,z\in\mathbb Z^d$,
\item ({irreducibility}) The additive subgroup generated by
\(\{x\in\mathbb Z^d:{w_{0,x}}>0\}\) is \(\mathbb Z^d\).
\item (normalization) $\sum_{y\in\mathbb Z^d} w_{0,y}=1$,
\item (OS-positivity) $w$ satisfies Definition~\ref{def:OSpos-quadratic} below.
\item (torus Green-function convergence) Whenever $g_{d,w}<\infty$,
\begin{equation}\tag{H}\label{eq:torus-green-convergence}
\lim_{\substack{L\to\infty\\L\in2\mathbb N}}
\frac1{L^d}\sum_{k\in\T_L^*\setminus\{0\}}
\frac1{1-\widehat w(k)}=g_{d,w}.
\end{equation}
\end{enumerate}
\end{defn}
The assumption 
$
\sum_{z\in\mathbb{Z}^d} w_{0,z}=1
$
is without loss of generality.  
We refer to the value $w(0)$ as \textit{diagonal activity}.
This corresponds to the monomer activity for MDD and to the weight assigned 
to the fixed points of the permutations in RLP.
We refer to Definition~\ref{def:OSpos-quadratic} for the precise definition of
OS-positivity{. We} list below several examples of admissible weights{,
followed by two closure properties of OS-positivity. The bipartite
projections and finite convex combinations of the displayed examples
are admissible as well; see Section~\ref{sect:proofOSpositivity}.}
\begin{enumerate}

\item[(A)] \textit{Nearest neighbours:}
\[
w(x)=
\begin{cases}
\frac1{2d}, & |x|_1=1,\\
0, & \text{otherwise}.
\end{cases}
\]

\item[(B)] \textit{Positive diagonal activity, first and second nearest neighbours:} {Assume $d\geq2$.}
\[
w(x)=
\begin{cases}
\rho  & x=0 ,\\
\beta  & |x|_2=1,\\
\gamma &   |x|_2= \sqrt{2}, \\
0  &   \text{otherwise}.
\end{cases}
\]
with $\rho +  2 d \beta + 2 d (d-1) \gamma = 1$
and 
\begin{equation}\label{eq:secondnearestneighbourvalues}
{0 \leq \rho < 1}  \quad \quad 0 \le \gamma \le \frac{\beta}{2(d-1)}.
\end{equation}

\item[(C)] \textit{Exponential decay:}
for $\gamma>0$ and $\rho\in{[0,1)}$, let
\[
w_{\rho,\gamma}(x)=
\begin{cases}
\rho, & x=0,\\[1mm]
c_d(\rho,\gamma)e^{-\gamma |x|_1}, & x\neq0,
\end{cases}
\]
where
$
c_d(\rho,\gamma)
$ is a normalisation constant.

\item[(D)] \textit{Polynomial decay:}
Let \(q\in\{1,2\}\), \(s>d\), and \(\rho\in{[0,1)}\). Define
\[
w_{\rho,s,q}(x)
:=
\begin{cases}
\rho, & x=0,\\[2mm]
\dfrac{1-\rho}{Z_{s,q}}\dfrac1{|x|_q^s},
& x\neq0,
\end{cases}
\qquad \qquad 
Z_{s,q}
:=
\sum_{z\in\mathbb Z^d\setminus\{0\}}
\frac1{|z|_q^s}.
\]

\item[(E)]
\textit{Regularised polynomial decay:}
Let \(s>d\), \(\gamma>0\), and \(\rho\in{[0,1)}\). Define
\[
w_{\rho,s,\gamma}(x)
:=
\begin{cases}
\rho, & x=0,\\[2mm]
\dfrac{1-\rho}{Z_{s,\gamma}}
\dfrac{1}{(1+\gamma |x|_1)^s},
& x\neq0,
\end{cases}
\qquad  \qquad 
Z_{s,\gamma}
:=
\sum_{z\in\mathbb Z^d\setminus\{0\}}
\frac{1}{(1+\gamma |z|_1)^s}.
\]

\item[(F)] \textit{Bipartite projections:}
If $w$ is OS-positive and $\sum_{|y|_1\ \mathrm{odd}}w(y)>0$, then its bipartite projection
\[
w^{\mathrm{o}}(x)
:=
\frac{w(x)\mathbf 1_{\{|x|_1\ \mathrm{odd}\}}}
{\sum_{y:\ |y|_1\ \mathrm{odd}} w(y)}
\]
is {also} OS-positive.  
Hence,  any bipartite  projection of the examples above is also OS-positive. 

\item[(G)] \textit{Convex combinations:}
{Any} convex combination of weights which are OS-positive is OS-positive.
Hence, any convex combination of the examples above is OS-positive. 
\end{enumerate}

Let us finally stress that the families considered above are far from
exhaustive. For example, one may consider weights of the form
\[
w(x)=\mathbb P(S_N=x),
\]
where \(S\) is a simple random walk and \(N\) is an independent
geometric random variable, or more general OS-positive weights
supported on the unit cube \(\{-1,0,1\}^d\), allowing simultaneous
displacements in several coordinate directions. One could also relax
translation invariance and consider weights that are invariant only
under translations preserving the parity decomposition of
\(\mathbb Z^d\). Our analysis can be adapted to these settings as well.
We do not pursue the greatest possible level of generality here, but
focus instead on a representative collection of natural examples in
order to keep the statements and proofs transparent.

Our results apply to the torus $\mathbb{T}_L$ 
with periodic boundary conditions. 
The natural way of adapting the weights to the torus structure while preserving translation invariance and normalisation is to define the periodised weights as in the following definition.

\begin{defn}[Periodisation of the weights]\label{def:periodicisation}
Let $w:\mathbb Z^d \times \mathbb{Z}^d \to[0,\infty)$ be {an admissible weight}. 
We define their \emph{periodisation} $w^{(L)}:\mathbb T_L\times\mathbb T_L\to[0,\infty)$ by
\[
w^{(L)}_{x,y}
:= \sum_{z\in\mathbb Z^d} w_{x, y + L z}
{.}\]
\end{defn}

{For even $L$, we write
$\T_L^{\mathrm e}:=\{x\in\T_L:|x|_1\text{ is even}\}$ and
$\T_L^{\mathrm o}:=\{x\in\T_L:|x|_1\text{ is odd}\}$.}

When studying the dimer model,  MDD or RLP 
on the torus \(G=\T_L\) with the periodised weights \(w^{(L)}\), obtained from
the original weights \(w:\Z^d\times\Z^d\to{[0,\infty)}\),
we will use the subscript \(_{L,w}\) instead of \(_{\T_L,w^{(L)}}\) throughout,
so that for example \(\mathcal C_{\T_L,w^{(L)}}= \mathcal C_{L,w}\).
Moreover, since we assume translation--invariant weights, we write
\(\mathcal C_{L,w}(x-y)= \mathcal C_{L,w}(x,y)\),
$w^{(L)}(x-y) = w_{x,y}^{(L)}$,
 and adopt the same convention for all quantities.

\subsection{\texorpdfstring{{Long-range order}}{Long-range order} and macroscopic loops}
\label{sect:longrangeordertheorems}
We present {two general results} establishing long-range order and macroscopic loops.
They require the Green's function of the random walk at the origin to be sufficiently small, but rely on two different structural assumptions on the weights. The first assumption is combinatorial in nature and requires the weights to be {bipartite}, whereas the second is analytic and requires their Fourier transform to be not too negative. 

{For} $A \in \{\mathrm{per}, \mathrm{mdd}\}$, {define}
$$
M^{A}_{L,w} :=  \sum_{x \in \T_L} \frac{  G^A_{L,w}(x)}{L^d},
$$
{the Ces\`aro average of the two-point function for MDD or RLP.}

Our first theorem provides a sufficient condition for long-range order 
and macroscopic loops under the bipartite weights assumption. 
We remind the reader that, under this condition,
{the monomer double-dimer model and random lattice permutations are equivalent to the  double-dimer model and their two-point function equals the dimer monomer--monomer correlation.}

\begin{thm}[Long-range order for bipartite weights]
\label{thm:longrangeorder}
Let $d\in\N$, and let $w$ be admissible bipartite weights.
Suppose that
\begin{equation}
\label{eq:condition43}
g_{d,w}< \frac{4}{3}.
\end{equation}
Then, there exists a strictly positive constant  $c = c(w,d)$ such that 
\begin{align}
\label{eq:claim}
\liminf_{\substack{L\to\infty:\ L\in2\N}}
M^{\mathrm{mdd}}_{L,w}
& \geq c, \\
\label{eq:macroscopicloops}
\liminf_{\substack{L\to\infty:\ L\in2\N}}
\frac{\mathbb E^{\mathrm{mdd}}_{L,w}(|\mathcal L_0|)}
     {L^d} &
\geq c,
\end{align}
and, 
for all odd integers $n\in(0,cL)$ and all sufficiently large even $L$,
\begin{align}
\label{eq:pointwise-LRO}
G^{\mathrm{mdd}}_{L,w}({ne_1}) & \geq c, \\
\label{eq:pointwise-macroscopicloops}
\mathbb P^{\mathrm{mdd}}_{L,w}
({ne_1}\in\mathcal L_0) &  \geq c.
\end{align}
\end{thm}
Condition (\ref{eq:condition43}) is only for convenience,
we  refer to  Remark \ref{rem:sharper-LRO-condition}
for a sharper condition.

As we show in the next theorem, our results on long-range order are not limited to lazy-bipartite weights.
Our next theorem establishes long-range order under an assumption which involves the Fourier transform of the weights. 
Set 
\[
a_w:=\max\Bigl\{0,-\min_{k\in[-\pi,\pi]^d}\widehat w(k)\Bigr\}.
\]
For random lattice permutations the sufficient condition for {long-range order} is 
$$
g_{d,w} <
 \frac{2}
     {1+a_w}.
$$
    For the monomer double-dimer model, the condition is slightly {more restrictive},
    since we need to subtract from the right hand side an additional factor which depends on the monomer activity and the total weight at odd sites and which vanishes when the monomer activity is zero
    $$
g_{d,w} <
\frac{2 - q_w}{1+a_w},
    $$  
where
$
q_w:=
\min\left\{1,\frac{e\,w(0)}{b(w)}\right\}
$
and  $b(w) := \sum_{x \in \mathbb{Z}^d : |x|_1 \in 2 \mathbb{N}_0+1} w(x)$.
\begin{thm}[Long-range order beyond {lazy-bipartite} weights]
\label{thm:longrangeorder2}
Let \(w\) be {an admissible weight}.
Then
\begin{align}
\liminf_{\substack{L\to\infty\\L\in2\mathbb N}}
M^{\mathrm{per}}_{L,w}
&\geq
\frac{2}{1+a_w}-g_{d,w},
\label{eq:longrangequantity-per}
\\
\liminf_{\substack{L\to\infty\\L\in2\mathbb N}}
M^{\mathrm{mdd}}_{L,w}
&\geq
\frac{2- q_w}{1+a_w}
-g_{d,w},
\label{eq:longrangequantity-mdd}
\end{align}
where the second statement requires \(b(w)>0\).

Moreover, for either model,  if the corresponding right-hand side above
is positive, then there exists a strictly positive constant $c= c(w, d)$ such that 
\[
G^A_{L,w}({ne_1})\geq  c
\]
for every odd integer \(n\in(0,cL)\) and every sufficiently large even
\(L\), where \(A=\mathrm{per}\) or \(A=\mathrm{mdd}\), respectively.
\end{thm}
If  the Fourier transform is non-negative, then the sufficient condition for {long-range order in random lattice permutations} is
$$
g_{d,w}<2.
$$
In this case the numerical condition on the Green's function is weaker than the one required in the bipartite setting. On the other hand, the conclusion is correspondingly weaker: the theorem yields uniform positivity of the two-point function, but does not establish the occurrence of macroscopic loops.

\subsection{Examples}\label{sect:examples}
We now illustrate our general results through several natural choices of
weights. Depending on the example, our criteria yield long-range order,
macroscopic loops, or both. A common feature of these examples is that
our sufficient conditions are satisfied when the weights are sufficiently
spread out, either because the dimension is large enough or by tuning
parameters that control their spatial range.

For several families of
weights, in the RLP case our long-range-order criterion applies for every
fixed-point weight $\rho<1/2$, arbitrarily close to the sharp threshold
$1/2$ of Theorem~\ref{thm:decay-large-monomer}.

The nearest-neighbour case has been addressed in 
\cite{QuitmannTaggi2,T} in dimensions $d\geq3$.
We first consider weights supported on first and second neighbours.
These weights are non-bipartite, so
Theorem~\ref{thm:longrangeorder2} provides the appropriate
long-range-order criterion. 

\begin{cor}[First- and second-nearest neighbours]
\label{cor:firstsecondneighbours}
Let $d\geq3$ and consider the weights in Example~\emph{(B)} with
\begin{equation}\label{eq:choiceB}
\beta=\frac{1-\rho}{3d},
\qquad
\gamma=\frac{1-\rho}{6d(d-1)},
\end{equation}
and $\rho\in[0,1]$. There exists $\rho_0>0$ such that, for every
$\rho\in[0,\rho_0)$, these weights are admissible and the right-hand
sides of \eqref{eq:longrangequantity-per} and
\eqref{eq:longrangequantity-mdd} are both strictly positive.
Consequently, Theorem~\ref{thm:longrangeorder2} implies long-range
order for the monomer double-dimer model and random lattice
permutations with these weights.
\end{cor}

This family of parameters includes the case of zero diagonal activity.
In that case, the two-point function of the monomer double-dimer model
coincides with the monomer correlation of the dimer model. Hence the
corollary proves that monomer correlations are uniformly positive on
$\mathbb Z^d$ with edges connecting first- and second-nearest-neighbour
vertices in every dimension $d\geq3$.

We next turn to weights of infinite range, starting with exponentially
decaying weights. Theorem~\ref{thm:longrangeorder2} yields long-range
order for the original weights, while their bipartite projection falls
within the scope of Theorem~\ref{thm:longrangeorder}, which additionally
yields macroscopic loops.

\begin{cor}[Exponentially decaying weights]
Suppose that $d\geq3$ and consider the weights $w$ in Example~\emph{(C)}.
There exists $\rho_0>0$ such that, for every
$\rho\in[0,\rho_0)$, there exists $\gamma_\rho>0$ such that, if
$\gamma\in(0,\gamma_\rho)$, then the right-hand sides of
\eqref{eq:longrangequantity-per} and
\eqref{eq:longrangequantity-mdd} are both strictly positive.
Consequently, Theorem~\ref{thm:longrangeorder2} implies long-range
order for RLP and MDD.

Moreover, in the RLP case one can take $\rho_0=1/2$; hence the
conclusion holds for every $\rho<1/2$, arbitrarily close to the
threshold $1/2$.

Finally, there exists $\gamma'>0$ such that, if
$\gamma\in(0,\gamma')$, then the bipartite projection of the weights
satisfies \eqref{eq:condition43}. Hence
Theorem~\ref{thm:longrangeorder} implies long-range order and
macroscopic loops for these models with bipartite exponentially
decaying weights.
\end{cor}

The preceding examples concern weights with finite range or
exponentially decaying tails, for which our general Green-function criteria apply only in
dimensions $d\geq3$. The fully packed one-dimensional case is exceptional by Theorem~\ref{thm:zero-activity-one-dimensional}. Slowly decaying tails lead to a different picture:
sufficiently spread-out weights may produce long-range order and
macroscopic loops even in dimensions one and two.

{For brevity, we state the corollary only for Example~\emph{(E)}; analogous conclusions for Example~\emph{(D)} follow from the estimates in Lemma~\ref{lem:examples-C-E-fourier}.}

\begin{cor}[Regularised polynomially decaying weights]
\label{cor:regularised-polynomial-LRO}
Suppose that $d\geq1$ and consider the weights $w$ in
Example~\emph{(E)}. There exists $\rho_0\in(0,1/2]$ such that the
following holds. Suppose that
{
\begin{equation}\label{eq:svaluesoptimal}
\begin{cases}
d<s<2d, & d\in\{1,2\},\\
s>d,    & d\geq3.
\end{cases}
\end{equation}}
Then, for every $\rho\in[0,\rho_0)$, there exist
$\gamma_0=\gamma_0(d,s,\rho)>0$ and $c>0$ such that, for every
$0<\gamma<\gamma_0$, the right-hand sides of
\eqref{eq:longrangequantity-per} and
\eqref{eq:longrangequantity-mdd} are both strictly positive.
Consequently, Theorem~\ref{thm:longrangeorder2} implies long-range
order for MDD and RLP.

Moreover, in the RLP case one can take $\rho_0=1/2$; hence the
conclusion holds for every $\rho<1/2$, arbitrarily close to the
threshold $1/2$.

Finally, for every $s$ satisfying \eqref{eq:svaluesoptimal}, one can
choose $\gamma>0$ sufficiently small so that the bipartite projection
of the weights satisfies \eqref{eq:condition43}. Consequently,
Theorem~\ref{thm:longrangeorder} implies long-range order and the
existence of macroscopic loops for both models.
\end{cor}

\subsection{Diagonal activity greater than \texorpdfstring{$1/2$}{1/2}}
\label{sect:large-diagonal-activity}
{Our next theorem gives a complementary obstruction based on the diagonal activity.}
More precisely,  it establishes that, if $w(0) > \frac{1}{2}$ (and  hence $G^{\mathrm{RW}}_w(0) > 2$) then both the two-point function and the probability that two vertices belong to the same loop decay exponentially or polynomially with their distance, depending on the tail of the weights. 
The proof is elementary and does not rely on reflection positivity.  In fact, the theorem applies to non-negative weights in any dimension; we assume translation invariance and torus structure only for simplicity.   
\begin{thm}[Absence of long-range order  for large diagonal activity]
\label{thm:decay-large-monomer}
Let {\(w:\mathbb Z^d\to[0,\infty)\)} be symmetric, translation invariant and normalised to one. 
Suppose that 
$$ 
w(0)>\frac12.
$$
Write $|x|_1$ for the $\ell^1$ torus distance from $x$ to the origin.
{For every $A\in\{\mathrm{mdd},\mathrm{per}\}$ and either}
\[
{F^A_{L,w}(x)=G^A_{L,w}(x)}
\qquad\text{or}\qquad
{F^A_{L,w}(x)=\mathbb P^A_{L,w}(x\in\mathcal L_0)},
\]
we have
\begin{equation}\label{eq:qualitative-decay}
\lim_{R\to\infty}
\sup_L
\sum_{\substack{x\in\mathbb T_L\\ |x|_1\geq R}}
{F^A_{L,w}(x)}
=0.
\end{equation}
Moreover, if for some \(p>0\),
$
\sum_{x\in\mathbb Z^d}|x|_1^p w(x)<\infty,
$
then there exists \(C_p<\infty\) such that, uniformly in \(L\),
\[
{F^A_{L,w}(x)}
\leq
\frac{C_p}{(1+|x|_1)^p}.
\]
Furthermore, if for some \(\eta>0\),
$
\sum_{x\in\mathbb Z^d}e^{\eta|x|_1}w(x)<\infty,
$
then there exist \(C,c>0\) such that, uniformly in \(L\),
\[
{F^A_{L,w}(x)} \leq Ce^{-c|x|_1}.
\]
\end{thm}
For  RLP,  the value \(1/2\) of the diagonal activity is sharp in the following sense. 
By
Theorem~\ref{thm:decay-large-monomer}, every weight function covered by
the theorem with diagonal activity strictly greater than \(1/2\) gives
rise to spatial random-permutation and monomer double-dimer models
\textit{without long-range order.} Conversely, for every prescribed fixed-point weight
strictly smaller than \(1/2\), the examples in
Section~\ref{sect:examples} provide weights for which \textit{long-range order
does occur} in random lattice permutations. Thus, above \(1/2\), long-range order is ruled out for every
choice of weights, whereas below \(1/2\) it can be obtained by choosing
the weights sufficiently spread out.

\subsection{One dimension}
The one-dimensional case is particularly interesting in the long-range
setting, since long-range order may occur even in dimension one, as
shown in Section~\ref{sect:examples}. We first prove that positive
diagonal activity rules out long-range order {for lazy-bipartite} weights
{with} finite first moment.

Let
\[
L_{\max}:=\max_{x\in\mathbb T_L}|\mathcal L_x|.
\]

\begin{thm}[Absence of long-range order at positive activity]
\label{thm:onedimensiondecay}
\label{thm:positive-activity-one-dimensional}
Let $d=1$, and let $w:\mathbb Z\to[0,\infty)$ be symmetric and lazy
bipartite. Assume that
\[
w(0)>0,
\qquad
\sum_{r\in\mathbb Z}|r|w(r)<\infty.
\]
Then, for every $\varepsilon>0$, as $L\to\infty$ through even integers,
\[
\mathbb P^{\mathrm{mdd}}_{L,w}
   \bigl(L_{\max}\ge\varepsilon L\bigr)\longrightarrow0,
\qquad
\frac1L\mathbb E^{\mathrm{mdd}}_{L,w}L_{\max}\longrightarrow0,
\]
\[
\frac1L\mathbb E^{\mathrm{mdd}}_{L,w}
   |\mathcal L_0|\longrightarrow0,
\qquad
M^{\mathrm{mdd}}_{L,w}\longrightarrow0.
\]
\end{thm}

The restriction to lazy-bipartite weights is relevant here.
Corollary~\ref{cor:regularised-polynomial-LRO} gives non-bipartite
one-dimensional examples with positive diagonal activity and infinite
first moment for which long-range order occurs. These examples do not
establish sharpness of the finite-first-moment assumption within the
class covered by Theorem~\ref{thm:onedimensiondecay}.

The assumption $w(0)>0$ is essential. At zero monomer
activity, finite first moment does not prevent long-range order. More
precisely, {every symmetric, normalised, bipartite weight with
finite first moment exhibits long-range order and macroscopic loops.}

\begin{thm}[Long-range order and macroscopic loops at zero activity]
\label{thm:zero-activity-one-dimensional}
Let $d=1$, and let $w:\mathbb Z\to[0,\infty)$ be symmetric, normalised,
and bipartite. Assume that
\[
{M_1:=\sum_{r\in\mathbb Z}|r|w(r)<\infty.}
\]
Then there exist constants $\varepsilon,c>0$, depending only on $w$,
such that, for every even $L$,
\[
\mathbb P^{\mathrm{mdd}}_{L,w}
   \bigl(L_{\max}\ge\varepsilon L\bigr)\ge c,
\]
\[
\mathbb P^{\mathrm{mdd}}_{L,w}
   \bigl(|\mathcal L_0|\ge\varepsilon L\bigr)\ge\varepsilon c,
\qquad
M^{\mathrm{mdd}}_{L,w}\ge c.
\]
\end{thm}

\subsection{Paper organisation and notation}

This paper is {organised} as follows.
In Section~\ref{sect:spinrepresentation} we introduce the complex spin representation in the long-range setting, prove reflection positivity, and derive its main consequences: site monotonicity, Gaussian domination, and the infrared bound.
In Section~\ref{sect:fromspinstodimers} we discuss the connection between the spin representation, dimers and permutations. 
In Section~\ref{sect:sufficientcond} we establish general sufficient conditions for long-range order and the occurrence of macroscopic loops.
In Section~\ref{sect:onehalf} we present the proof of Theorem~\ref{thm:decay-large-monomer}.
In Section~\ref{sect:decayonedimension} we present the proof of Theorem~\ref{thm:onedimensiondecay}.
{In Section~\ref{sect:proof-zero-activity-one-dimensional} we prove Theorem~\ref{thm:zero-activity-one-dimensional}
and derive Corollary~\ref{cor:bi-infinite-one-dimensional}; the remaining auxiliary arguments are collected in the appendix.}

\paragraph{Notation.}
We  use the  notation $\mathbb{N}_0 = \{0, 1, \ldots\}$,
$\mathbb{N}=\{1,2,\ldots\}$, $\mathbb{R}^+_0 = \{x \in \mathbb{R}: x \geq 0\}$. 
We use the abbreviations `LHS' for  `left-hand side' and `RHS' for `right-hand side'.

\section{The complex spin representation}
\label{sect:spinrepresentation}
In this section we introduce the complex spin representations of the monomer
double-dimer and random lattice-permutation models and show that the associated complex measures are reflection
positive. We then discuss the main consequences of reflection positivity.

\subsection{The spin representation}
\label{sect:spinsystem} 
Let \(G=(V,E)\) be a finite connected undirected graph with a prescribed
vertex \(o\in V\), {and let $V=V^{\mathrm e}\mathbin{\dot\cup}V^{\mathrm o}$ be a prescribed decomposition with $o\in V^{\mathrm e}$}. Let
$
w_{i,j}=w_{j,i}\geq0,
$
be symmetric weights for each pair of vertices $i, j \in V$. 

We denote by \(V^{\mathrm e}\) and \(V^{\mathrm o}\), respectively, the
{two prescribed parity classes; the edge set need not respect this decomposition}.
Let
\[
\Xi=[0,2\pi)^2,
\qquad
\Omega_s:=\Xi^V.
\]
For each configuration
\[
\boldsymbol{s}=(s_z)_{z\in V}\in\Omega_s,
\qquad
s_z=(s_z^1,s_z^2),
\]
we define the spin \(S_z=(S_z^1,S_z^2)\) at \(z\), where
\(S_z^k:\Omega_s\to\mathbb C\), \(k\in\{1,2\}\), is given by
\begin{equation}\label{eq:spinvariable}
S_z^k(\boldsymbol{s})
:=
\begin{cases}
e^{\mathrm i s_z^k},
& (k=1\text{ and }z\in V^{\mathrm e})
  \text{ or }(k=2\text{ and }z\in V^{\mathrm o}),\\[1mm]
e^{-\mathrm i s_z^k},
& (k=2\text{ and }z\in V^{\mathrm e})
  \text{ or }(k=1\text{ and }z\in V^{\mathrm o}).
\end{cases}
\end{equation}
Thus, \(s_z^1\) and \(s_z^2\) are the angles associated with the two
components of \(S_z\).
For \(A\in\{\mathrm{per},\mathrm{mdd}\}\), define
\begin{align}
\gamma_z^A(\boldsymbol{s})
&:=
\frac{1}{(2\pi)^2}
\begin{cases}
w_{z,z}
+\overline{S_z^1(\boldsymbol{s})}\,\overline{S_z^2(\boldsymbol{s})},
& A=\mathrm{mdd},\\[1mm]
w_{z,z}
+\bigl(\overline{S_z^1(\boldsymbol{s})}\bigr)^2
+\bigl(\overline{S_z^2(\boldsymbol{s})}\bigr)^2,
& A=\mathrm{per},
\end{cases}
\label{eq:gfunction}
\\
\boldsymbol{\gamma}^A(\boldsymbol{s})
&:=
\prod_{z\in V}\gamma_z^A(\boldsymbol{s}).
\label{eq:complealternationgamma}
\end{align}

We define the Hamiltonian
\begin{equation}\label{eq:hamiltoninan}
H(\boldsymbol{s})
:=
\sum_{\{i,j\}\in E}w_{i,j}
\left(
S_i^1(\boldsymbol{s})S_j^1(\boldsymbol{s})
+
S_i^2(\boldsymbol{s})S_j^2(\boldsymbol{s})
\right).
\end{equation}

For every measurable function \(f:\Omega_s\to\mathbb C\), set
\begin{align}
\langle f\rangle_{G,w, A}
&:=
\frac{1}{Z^{\mathrm{spin}}_{G,w,A}}
\int_{\Omega_s}
\boldsymbol{d s}\,
\boldsymbol{\gamma}^A(\boldsymbol{s})
e^{H(\boldsymbol{s})}f(\boldsymbol{s}),
\label{eq:measure}
\\
Z^{\mathrm{spin}}_{G,w, A}
&:=
\int_{\Omega_s}
\boldsymbol{d s}\,
\boldsymbol{\gamma}^A(\boldsymbol{s})
e^{H(\boldsymbol{s})},
\end{align}
where
$
\boldsymbol{d s}:=
\prod_{z\in V}ds_z^1\,ds_z^2.
$
The choices \(A=\mathrm{per}\) and \(A=\mathrm{mdd}\) give the spin
representations of the random lattice permutation model and the monomer
double-dimer model, respectively.

\subsection{Reflection positivity}
\label{sect:RPsection}
\label{sect:RP}
Fix \(A\in\{\mathrm{per},\mathrm{mdd}\}\); all statements below hold for both choices. On \(G=\T_L\), let \(\langle\cdot\rangle_{L,w,A}\) denote the spin expectation \eqref{eq:measure} associated with the periodised weights \(w^{(L)}\) from Definition~\ref{def:periodicisation}. We show that this complex measure is reflection positive whenever \(w\) is OS-positive. Let \(\mathcal S\) be a mid-edge plane orthogonal to \(e_i\) for some \(i\in\{1,\ldots,d\}\).

\begin{defn}[Reflection of sites, configurations and functions]
Let \(\Theta:\T_L\to\T_L\) be reflection in \(\mathcal S\). We use the same symbol for
\[
\Theta(\boldsymbol s):=(s_{\Theta x})_{x\in\T_L},\qquad
\Theta f(\boldsymbol s):=f(\Theta\boldsymbol s),
\]
where \(\boldsymbol s=(s_x)_{x\in\T_L}\in\Omega_s\) and \(f:\Omega_s\to\mathbb C\).
\end{defn}
\begin{defn}[Function domain]
A function \(f:\Omega_s\to\mathbb C\) has domain \(D\subset\T_L\) if \(f(\boldsymbol s)=f(\boldsymbol s')\) whenever \(s_x=s_x'\) for every \(x\in D\). In particular, \(S_x^k\), \(k=1,2\), has domain \(\{x\}\).
\end{defn}
\begin{defn}[Torus halves and \(\mathcal A^\pm\)-functions]
Let \(\T_L=\T_L^+\mathbin{\dot\cup}\T_L^-\), where the two halves are connected and \(\Theta(\T_L^\pm)=\T_L^\mp\). Denote by \(\mathcal A^\pm\) the family of integrable, measurable functions with domain \(\T_L^\pm\).
\end{defn}
For \(x\in\T_L\), \eqref{eq:spinvariable} and \eqref{eq:gfunction} give
\begin{align}
\Theta S_x&=\overline{S_{\Theta x}}
=\bigl(\overline{S_{\Theta x}^1},\overline{S_{\Theta x}^2}\bigr),
\label{eq:complexalternation}\\
\Theta\gamma_x^A&=\overline{\gamma_{\Theta x}^A}.
\label{eq:gamma-reflection}
\end{align}

\begin{defn}[Osterwalder--Schrader positivity]
\label{def:OSpos-quadratic}
Let \(w:\mathbb Z^d\times\mathbb Z^d\to[0,\infty)\) be symmetric. For \(i\in\{1,\ldots,d\}\), set
\[
\theta_i x:=x-2(x\cdot e_i)e_i+e_i,\qquad
\mathbb Z^d_{i,+}:=\{x:x\cdot e_i>0\}.
\]
We call \(w\) \emph{OS-positive} if, for every finitely supported \(f:\mathbb Z^d_{i,+}\to\mathbb C\),
\begin{equation}\label{eq:OSpositivity}
\sum_{x,y\in\mathbb Z^d_{i,+}}\overline{f(x)}\,w_{x,\theta_i y}f(y)\geq0
\qquad\text{for every }i\in\{1,\ldots,d\}.
\end{equation}
\end{defn}

\begin{lem}[OS-positivity under periodisation]
\label{lem:OS-periodisation}
If \(w\) is translation invariant, summable, and OS-positive and \(L\in2\mathbb N\), then, for every mid-edge reflection \(\theta_i\) on \(\T_L\) and every \(f:\T_{L,i,+}\to\mathbb C\),
\begin{equation}\label{eq:OS-periodised}
\sum_{x,y\in\T_{L,i,+}}\overline{f(x)}\,w^{(L)}_{x,\theta_i y}f(y)\geq0.
\end{equation}
\end{lem}
The proof is postponed to the appendix.

\begin{prop}[Reflection positivity]
\label{prop:reflectionpos}
Let \(w\) be admissible. For every \(f,g\in\mathcal A^+\),
\[
\langle f\,\overline{\Theta g}\rangle_{L,w,A}
=
\overline{
\langle g\,\overline{\Theta f}\rangle_{L,w,A}
},
\qquad
\langle f\,\overline{\Theta f}\rangle_{L,w,A}\geq0.
\]
Consequently,
\begin{equation}\label{eq:RPstatement}
\left|
\langle f\,\overline{\Theta g}\rangle_{L,w,A}
\right|^2
\leq
\langle f\,\overline{\Theta f}\rangle_{L,w,A}
\langle g\,\overline{\Theta g}\rangle_{L,w,A}.
\end{equation}
\end{prop}
\begin{proof}
Set
\[
\langle f\rangle_0:=\int_{\Omega_s}f(\boldsymbol s)\,d\boldsymbol s,\qquad
\langle f\rangle_{1,A}:=\int_{\Omega_s}f(\boldsymbol s)
\boldsymbol\gamma^A(\boldsymbol s)\,d\boldsymbol s.
\]
The product measure \(\langle\cdot\rangle_0\) is reflection positive
\cite[Chapter~10]{Velenik}. Moreover, with
\[
\boldsymbol\gamma^{A,\pm}:=\prod_{x\in\T_L^\pm}\gamma_x^A,\qquad
\Theta\boldsymbol\gamma^{A,+}=\overline{\boldsymbol\gamma^{A,-}},
\]
we have, for \(f,g\in\mathcal A^+\),
\[
\langle f\,\overline{\Theta g}\rangle_{1,A}
=\bigl\langle f\boldsymbol\gamma^{A,+}\,
\overline{\Theta(g\boldsymbol\gamma^{A,+})}\bigr\rangle_0.
\]
Thus \(\langle\cdot\rangle_{1,A}\) is reflection positive.
Using ordered pairs, write
\[
H=\frac12\sum_{\substack{x,y\in\T_L\\x\neq y}}w^{(L)}_{x,y}
\sum_{i=1}^2S_x^iS_y^i=H^++H^-+H^R,
\]
where
\[
H^\pm:=\frac12\sum_{\substack{x,y\in\T_L^\pm\\x\neq y}}w^{(L)}_{x,y}
\sum_{i=1}^2S_x^iS_y^i,\qquad
H^R:=\sum_{x,z\in\T_L^+}w^{(L)}_{x,\Theta z}
\sum_{i=1}^2S_x^i\,\overline{\Theta S_z^i}.
\]
Here symmetry of \(w^{(L)}\) was used to combine the two ordered cross terms, while reflection invariance gives \(H^-=\overline{\Theta H^+}\). The kernel \(K(x,z):=w^{(L)}_{x,\Theta z}\) is positive semidefinite by Lemma~\ref{lem:OS-periodisation}; hence, with \(N:=|\T_L^+|\),
\[
K(x,z)=\sum_{\alpha=1}^N\lambda_\alpha
\varphi_\alpha(x)\overline{\varphi_\alpha(z)},\qquad
\lambda_\alpha\geq0,
\]
for an orthonormal basis \((\varphi_\alpha)_{\alpha=1}^N\) of
\(\mathbb C^{\T_L^+}\). Therefore
\[
H^R=\sum_{\alpha=1}^N\sum_{i=1}^2
\lambda_\alpha B_\alpha^i\,\overline{\Theta B_\alpha^i},
\qquad
B_\alpha^i:=\sum_{x\in\T_L^+}\varphi_\alpha(x)S_x^i.
\]
For \(\boldsymbol m=(m_\alpha^i)\in\mathbb N_0^{2N}\), set
\[
c_{\boldsymbol m}:=\prod_{\alpha=1}^N\prod_{i=1}^2
\frac{\lambda_\alpha^{m_\alpha^i}}{m_\alpha^i!},\qquad
C_{\boldsymbol m}:=e^{H^+}\prod_{\alpha=1}^N\prod_{i=1}^2
(B_\alpha^i)^{m_\alpha^i}.
\]
Expanding \(e^{H^R}\) gives
\begin{equation}\label{eq:RP-general}
\langle f\,\overline{\Theta g}\rangle_{L,w,A}
=\frac1{Z^{\mathrm{spin}}_{L,w,A}}
\sum_{\boldsymbol m\in\mathbb N_0^{2N}}c_{\boldsymbol m}
\bigl\langle fC_{\boldsymbol m}\,
\overline{\Theta(gC_{\boldsymbol m})}\bigr\rangle_{1,A}.
\end{equation}
Since \(c_{\boldsymbol m}\geq0\),
\(C_{\boldsymbol m}\in\mathcal A^+\), and
\(Z^{\mathrm{spin}}_{L,w,A}>0\), equation
\eqref{eq:RP-general} expresses
\[
(f,g)\longmapsto
\langle f\,\overline{\Theta g}\rangle_{L,w,A}
\]
as a positive linear combination of positive-semidefinite Hermitian
forms. Hence it is itself a positive-semidefinite Hermitian form.
The claim now follows from the Cauchy--Schwarz inequality for such
forms.
\end{proof}

\subsection{Monotonicity of two-point functions}
We first extend the site-monotonicity result of \cite{LeesTaggiCMP2020, LeesTaggiJSP2021} to long-range interactions.
\begin{defn}[Reflection invariant vector function]
\label{def:reflectioninvariant}
A family \(F=(F_x)_{x\in\T_L}\), \(F_x:{\Omega_s}\to\mathbb C\), is called \emph{reflection invariant} if, for every reflection \(\Theta\) through edges and all \({x\in\T_L}\),
\[
\Theta F_x=\overline{F_{\Theta x}}
\]
\end{defn}
\begin{prop}[Monotonicity]
\label{prop:monotonicity}
Let {\(w\) be admissible, and let} \(F=(F_x)_{x\in\T_L}\) be
reflection invariant, with each
$F_x$ having domain $\{x\}$, and suppose that its joint moments are
real and translation invariant. Define
\begin{equation}\label{eq:monotonicity}
\mathcal G_L(x):=\langle F_oF_x\rangle_{L,w,A}.
\end{equation}
Suppose that there exists $M < \infty$ such that,
for each  $L \in 2 \mathbb{N}$ and each odd $n$, we have 
\begin{equation}\label{eq:boundedness}
\mathcal G_L \, (n \, e_i \,  ) \leq M.
\end{equation}
For \(x_n=ne_i\), \(i\in\{1,\ldots,d\}\)
and  \(x\in\T_L\) such that \(x\cdot e_i\) is odd, one has, for every odd
$n \in (0, L/2)$
\[
\mathcal G_L(x_n)\geq\mathcal G_L(x_{n+2}),\qquad
\mathcal G_L((x\cdot e_i)e_i)\geq\mathcal G_L(x).
\]
\end{prop}
\begin{proof}
Once reflection positivity is established, the proof is the same as in \cite{LeesTaggiJSP2021}.
\end{proof}

\subsection{Chessboard estimate}
Let \(F\) have domain \(\{o\}\), and let \(e_1,\ldots,e_k\) be a self-avoiding nearest-neighbour path from \(o\) to \(t\in\T_L\). If \(\Theta_i\) denotes reflection through the mid-edge hyperplane orthogonal to \(e_i\), define
\[
F^{[t]}:=
\begin{cases}
\overline{\Theta_k\circ\cdots\circ\Theta_1(F)},&k\ \text{even},\\
\Theta_k\circ\cdots\circ\Theta_1(F),&k\ \text{odd}.
\end{cases}
\]
This function is path-independent and has domain \(\{t\}\).
\begin{prop}[Chessboard estimate]
\label{prop:chessboard}
Let \(F=(F_t)_{t\in\T_L}\) be a family of complex-valued functions on \(\Omega_s\), each with domain \(\{o\}\). Then
\begin{equation}\label{eq:chessboard}
\Bigg | \left\langle\prod_{t\in\T_L}F_t^{[t]}\right\rangle_{L,w,A} \Bigg | \, \, 
\leq\prod_{t\in\T_L}
\left\langle\prod_{s\in\T_L}F_t^{[s]}\right\rangle_{L,w,A}^{1/|\T_L|}.
\end{equation}
\end{prop}
\begin{proof}
This is the classical chessboard argument; see \cite[Chapter~10]{Velenik}.
\end{proof}

\subsection{Gaussian domination}
Let \(w:\T_L\times\T_L\to[0,\infty)\) be the periodisation of an admissible weight. For \(h\in\mathbb R^{\T_L}\), set
\[
(\Delta h)_x:=\sum_{y\in\T_L}w_{x,y}(h_y-h_x)
\]
and
\begin{equation}\label{eq:centralquantity}
Z_A(h):=\int_{\Omega_s}\boldsymbol{d s}\,\boldsymbol\gamma^A(\boldsymbol s)\exp\Bigg\{-\frac12\sum_{\{i,j\}\in E}\sum_{k=1}^2w_{i,j}\Big[\bigl(S_i^k+h_i\delta_{k,1}-S_j^k-h_j\delta_{k,1}\bigr)^2-(S_i^k)^2-(S_j^k)^2\Big]\Bigg\}.
\end{equation}
Thus \(Z_A(0)=Z^{\mathrm{spin}}_{L,w,A}\).
\begin{prop}[Gaussian domination]
\label{prop:gaussian domination}
For every \(h\in\mathbb R^{\T_L}\),
\begin{equation}\label{prop:GaussianDomination}
Z_A(h)\leq Z_A(0).
\end{equation}
\end{prop}
The proof follows from reflection positivity as in \cite[Section~10.5.3]{Velenik}; details are given in the appendix.

\begin{prop}
\label{prop:keyineq}
For every \(h\in\mathbb R^{\T_L}\),
\begin{equation}\label{eq:keyinequality}
\sum_{x,y\in\T_L}\langle S_x^1S_y^1\rangle_{L,w,A}
(\Delta h)_x(\Delta h)_y
\leq\sum_{\{x,y\}\in E}w_{x,y}(h_y-h_x)^2.
\end{equation}
\end{prop}
\begin{proof}
By Proposition~\ref{prop:gaussian domination}, \(\varphi\mapsto Z_A(\varphi h)\) has a maximum at zero. Hence
\begin{equation}\label{eq:Zsecond_eng}
\begin{aligned}
0\geq\frac{Z_A''(0)}{Z_A(0)}
={}&-\sum_{\{i,j\}\in E}w_{i,j}(h_i-h_j)^2\\
&+\left\langle\left(
\sum_{\{i,j\}\in E}w_{i,j}(S_i^1-S_j^1)(h_i-h_j)
\right)^2\right\rangle_{L,w,A}.
\end{aligned}
\end{equation}
Since
\begin{equation}\label{eq:sum-id_eng}
\sum_{\{x,y\}\in E}w_{x,y}(S_x^1-S_y^1)(h_x-h_y)
=-\sum_{x\in\T_L}S_x^1(\Delta h)_x,
\end{equation}
expanding the square in \eqref{eq:Zsecond_eng} gives
\eqref{eq:keyinequality}.
\end{proof}

\subsection{Infrared bound}
\begin{defn}[Fourier transform and dual torus]
For \(f:\T_L\to\mathbb C\) and
\(k\in\T_L^*\), define
\begin{equation}\label{eq:def_Ghat}
\widehat f(k):=\sum_{x\in\T_L}e^{\mathrm i k\cdot x}f(x).
\end{equation}
\end{defn}
\begin{prop}[Infrared bound]
\label{prop:infrared}
Suppose that the weights are OS-positive. 
 Define
\begin{equation}\label{eq:def_Dhat}
\varepsilon(k):=\sum_{z:\{o,z\}\in E}
w_{o,z}\bigl(1-\cos(k\cdot z)\bigr).
\end{equation}
Then, for every \(k\in\T_L^*\setminus\{0\}\),
\[
\widehat G^A_{L,w}(k)\leq\frac1{\varepsilon(k)}.
\]
\end{prop}
\begin{proof}
Fix \(k\in\T_L^*\setminus\{0\}\). Since
\[
\Delta\cos(k\cdot x)=-\varepsilon(k)\cos(k\cdot x),\qquad
\Delta\sin(k\cdot x)=-\varepsilon(k)\sin(k\cdot x),
\]
and
\[
\begin{aligned}
&\sum_{\{x,y\}\in E}w_{x,y}\Big[(\cos(k\cdot y)-\cos(k\cdot x))^2+(\sin(k\cdot y)-\sin(k\cdot x))^2\Big]
\\
&=\varepsilon(k)\sum_{x\in\T_L}\big[\cos^2(k\cdot x)+\sin^2(k\cdot x)\big]
\\
&=\varepsilon(k)|\T_L|,
\end{aligned}
\]
applying Proposition~\ref{prop:keyineq} to \(x\mapsto\cos(k\cdot x)\) and \(x\mapsto\sin(k\cdot x)\) and adding gives
\[
\begin{aligned}
&\varepsilon(k)^2\sum_{x,y\in\T_L}G^A_{L,w}(x,y)
\big[\cos(k\cdot x)\cos(k\cdot y)+\sin(k\cdot x)\sin(k\cdot y)\big]
\\
&=\varepsilon(k)^2\sum_{x,y\in\T_L}G^A_{L,w}(x,y)\cos\bigl(k\cdot(x-y)\bigr)
\leq\varepsilon(k)|\T_L|.
\end{aligned}
\]
By translation invariance and symmetry,
\[
\sum_{x,y\in\T_L}G^A_{L,w}(x,y)\cos\bigl(k\cdot(x-y)\bigr)
=|\T_L|\widehat G^A_{L,w}(k).
\]
Dividing by \(\varepsilon(k)^2|\T_L|\) proves the claim.
\end{proof}

\section{From spins to dimers}
\label{sect:fromspinstodimers}
We now identify the spin systems of Section~\ref{sect:spinsystem} with
the monomer double-dimer and random lattice-permutation models. The same
expansion applies to both models; only the local constraint is different.

Let \(G_{\mathrm{en}}=(V_{\mathrm{en}},E_{\mathrm{en}})\), where
\[
V_{\mathrm{en}}:=V\cup\{s\},
\qquad
E_{\mathrm{en}}
:=
E\cup\bigl\{\{x,s\}:x\in V\bigr\}.
\]
For
\[
m=(m^1,m^2)\in
\widehat\Omega_{\mathrm{en}}
:=
\bigl(\mathbb N_0^{E_{\mathrm{en}}}\bigr)^2,
\]
set
\[
n_x^i(m)
:=
\sum_{y:\{x,y\}\in E_{\mathrm{en}}}m^i_{x,y},
\qquad
\partial m
:=
\bigl(m^1_{x,s},m^2_{x,s}\bigr)_{x\in V}.
\]
We introduce the sets of admissible local times
\[
\mathcal Q_{\mathrm{mdd}}
:=
\{(0,0),(1,1)\},
\qquad
\mathcal Q_{\mathrm{per}}
:=
\{(0,0),(2,0),(0,2)\}.
\]
Thus, in the monomer double-dimer model, every non-monomer is incident
to one dimer of each colour, whereas in the random lattice-permutation
model it is incident to two dimers of the same colour.

For \(A\in\{\mathrm{mdd},\mathrm{per}\}\) and
\(u^1,u^2\in\mathbb N_0^V\), let
\[
\Omega_{\mathrm{en}}^A(u^1,u^2)
:=
\left\{
m\in\widehat\Omega_{\mathrm{en}}:
\partial m=(u^1,u^2),
\quad
\bigl(n_x^1(m),n_x^2(m)\bigr)\in\mathcal Q_A
\ \text{for every }x\in V
\right\}.
\]
For \(m\in\Omega_{\mathrm{en}}^A(u^1,u^2)\), set
\begin{equation}\label{eq:enlarged-weight}
\mu_{G,w}^{A,u}(m)
:=
\prod_{x\in V}
w_{x,x}^{\mathbf 1_{\{(n_x^1(m),n_x^2(m))=(0,0)\}}}
\prod_{\{x,y\}\in E}
\frac{
w_{x,y}^{m^1_{x,y}+m^2_{x,y}}
}{
m^1_{x,y}!\,m^2_{x,y}!
}.
\end{equation}
Source-edges carry no weight. We use the same notation for the total
weight of a set of configurations.

\begin{prop}[From spin correlations to dimers]
\label{prop:conversion}
Let \(A\in\{\mathrm{mdd},\mathrm{per}\}\) and
\(u^1,u^2\in\mathbb N_0^V\). Then
\begin{equation}\label{eq:partition-spin-dimer}
Z^{\mathrm{spin}}_{G,w,A}=Z^A_{G,w},
\end{equation}
and
\begin{equation}\label{eq:correlation}
\left\langle
\prod_{x\in V}
(S_x^1)^{u_x^1}(S_x^2)^{u_x^2}
\right\rangle_{G,w,A}
=
\frac{
\mu_{G,w}^{A,u}
\bigl(\Omega_{\mathrm{en}}^A(u^1,u^2)\bigr)
}{
Z^A_{G,w}
}.
\end{equation}
\end{prop}

\begin{proof}
For \(k_1,k_2,n_1,n_2\in\mathbb N_0\), the definition of the spin
variables gives
\begin{equation}\label{eq:angular-orthogonality}
\frac{1}{(2\pi)^2}
\int_{\Xi}
(\overline{S_x^1})^{k_1}(S_x^1)^{n_1}
(\overline{S_x^2})^{k_2}(S_x^2)^{n_2}
\,ds_x^1\,ds_x^2
=
\mathbf 1_{\{k_1=n_1\}}
\mathbf 1_{\{k_2=n_2\}}.
\end{equation}
Consequently, for every \(a,b\in\mathbb N_0\),
\begin{equation}\label{eq:local-integral-mdd}
\int_{\Xi}
\gamma_x^{\mathrm{mdd}}(\boldsymbol s)
(S_x^1)^a(S_x^2)^b\,ds_x
=
\begin{cases}
w_{x,x}, & (a,b)=(0,0),\\
1,       & (a,b)=(1,1),\\
0,       & \text{otherwise},
\end{cases}
\end{equation}
whereas
\begin{equation}\label{eq:local-integral-per}
\int_{\Xi}
\gamma_x^{\mathrm{per}}(\boldsymbol s)
(S_x^1)^a(S_x^2)^b\,ds_x
=
\begin{cases}
w_{x,x}, & (a,b)=(0,0),\\
1,       & (a,b)\in\{(2,0),(0,2)\},\\
0,       & \text{otherwise}.
\end{cases}
\end{equation}
The local integral can be non-zero only when
\((a,b)\in\mathcal Q_A\); within this set, its value is \(w_{x,x}\) when
\((a,b)=(0,0)\), and \(1\) otherwise.

We expand the exponential of the Hamiltonian as
\begin{align}
e^{H(\boldsymbol s)}
&=
\prod_{i=1}^2
\prod_{\{x,y\}\in E}
\sum_{k=0}^{\infty}
\frac{
w_{x,y}^k(S_x^iS_y^i)^k
}{k!}
\nonumber\\
&=
\sum_{q^1,q^2\in\mathbb N_0^E}
\left(
\prod_{\{x,y\}\in E}
\frac{
w_{x,y}^{q^1_{x,y}+q^2_{x,y}}
}{
q^1_{x,y}!\,q^2_{x,y}!
}
\right)
\prod_{x\in V}\prod_{i=1}^2
(S_x^i)^{n_x^i(q)},
\label{eq:H-expansion}
\end{align}
where
\[
n_x^i(q)
:=
\sum_{y:\{x,y\}\in E}q^i_{x,y}.
\]
Since \(G\) is finite, the series is absolutely convergent and may be
integrated term by term. Substituting~\eqref{eq:H-expansion} into the
unnormalized spin correlation gives
\begin{align}
&\int_{\Omega_s}
\boldsymbol{ds}\,
\boldsymbol\gamma^A(\boldsymbol s)
e^{H(\boldsymbol s)}
\prod_{x\in V}
(S_x^1)^{u_x^1}(S_x^2)^{u_x^2}
\nonumber\\
&\quad =
\sum_{q^1,q^2\in\mathbb N_0^E}
\left(
\prod_{\{x,y\}\in E}
\frac{
w_{x,y}^{q^1_{x,y}+q^2_{x,y}}
}{
q^1_{x,y}!\,q^2_{x,y}!
}
\right)
\prod_{x\in V}
\int_{\Xi}
\gamma_x^A(\boldsymbol s)
(S_x^1)^{n_x^1(q)+u_x^1}
(S_x^2)^{n_x^2(q)+u_x^2}
\,ds_x.
\label{eq:correlation-expansion}
\end{align}

By~\eqref{eq:local-integral-mdd}--\eqref{eq:local-integral-per}, the
term indexed by \(q=(q^1,q^2)\) can be non-zero only if
\begin{equation}\label{eq:local-survival}
\bigl(
n_x^1(q)+u_x^1,
n_x^2(q)+u_x^2
\bigr)
\in\mathcal Q_A
\qquad
\text{for every }x\in V.
\end{equation}
To each such \(q\), associate the unique configuration
\(m\in\widehat\Omega_{\mathrm{en}}\) defined by
\[
m^i_{x,y}:=q^i_{x,y}
\quad\text{for }\{x,y\}\in E,
\qquad
m^i_{x,s}:=u_x^i
\quad\text{for }x\in V.
\]
Then
\[
\partial m=(u^1,u^2)
\]
and
\[
n_x^i(m)
=
n_x^i(q)+u_x^i.
\]
Thus~\eqref{eq:local-survival} is equivalent to
\[
\bigl(n_x^1(m),n_x^2(m)\bigr)\in\mathcal Q_A
\qquad\text{for every }x\in V,
\]
that is, to \(m\in\Omega_{\mathrm{en}}^A(u^1,u^2)\).

Conversely, every configuration in
\(\Omega_{\mathrm{en}}^A(u^1,u^2)\) is uniquely determined by its
restriction to \(E\), since its source-edge multiplicities are fixed by
\(\partial m=(u^1,u^2)\). Hence the preceding construction is a
bijection between the surviving terms in~\eqref{eq:correlation-expansion}
and the configurations in
\(\Omega_{\mathrm{en}}^A(u^1,u^2)\).

For a surviving term, the local integral at \(x\) equals \(w_{x,x}\)
if \((n_x^1(m),n_x^2(m))=(0,0)\), and equals \(1\) otherwise. Moreover,
the Taylor coefficient associated with an original edge \(\{x,y\}\)
is
\[
\frac{
w_{x,y}^{m^1_{x,y}+m^2_{x,y}}
}{
m^1_{x,y}!\,m^2_{x,y}!
}.
\]
Therefore its total contribution is exactly
\(\mu_{G,w}^{A,u}(m)\), and~\eqref{eq:correlation-expansion} becomes
\begin{equation}\label{eq:unnormalized-conversion}
\int_{\Omega_s}
\boldsymbol{ds}\,
\boldsymbol\gamma^A(\boldsymbol s)
e^{H(\boldsymbol s)}
\prod_{x\in V}
(S_x^1)^{u_x^1}(S_x^2)^{u_x^2}
=
\mu_{G,w}^{A,u}
\bigl(\Omega_{\mathrm{en}}^A(u^1,u^2)\bigr).
\end{equation}

It remains to identify the partition function. Taking
\(u^1=u^2=0\) in~\eqref{eq:unnormalized-conversion} gives
\begin{equation}\label{eq:spin-partition-expanded}
Z^{\mathrm{spin}}_{G,w,A}
=
\mu_{G,w}^{A,0}
\bigl(\Omega_{\mathrm{en}}^A(0,0)\bigr).
\end{equation}

If \(A=\mathrm{mdd}\), every vertex in a source-free configuration has
local time either \((0,0)\) or \((1,1)\). Hence it is either a monomer
or is incident to exactly one dimer of each colour. In particular, all
edge multiplicities are in \(\{0,1\}\), so the factorials in
\eqref{eq:enlarged-weight} equal one. Thus the configurations and their
weights are exactly those of the monomer double-dimer model, and
\[
\mu_{G,w}^{\mathrm{mdd},0}
\bigl(\Omega_{\mathrm{en}}^{\mathrm{mdd}}(0,0)\bigr)
=
Z^{\mathrm{mdd}}_{G,w}.
\]

If \(A=\mathrm{per}\), every non-isolated vertex is incident to two
dimers of one colour. Consequently, every non-trivial connected
component is a monochromatic cycle. A cycle containing at least three
vertices has two possible colours, corresponding to its two possible
orientations, and in either case contributes the product of the weights
of its edges. A component consisting of two vertices joined by two
dimers has again two possible colours; each colour contributes
\(w_{x,y}^2/2!\), so their total contribution is
\[
2\,\frac{w_{x,y}^2}{2!}=w_{x,y}^2,
\]
which is exactly the weight of the transposition \(x\leftrightarrow y\).
Finally, an isolated vertex contributes \(w_{x,x}\), the weight of a
fixed point. Therefore the total weight of the source-free
configurations is precisely the random lattice-permutation partition
function:
\[
\mu_{G,w}^{\mathrm{per},0}
\bigl(\Omega_{\mathrm{en}}^{\mathrm{per}}(0,0)\bigr)
=
Z^{\mathrm{per}}_{G,w}.
\]
This proves~\eqref{eq:partition-spin-dimer}. Dividing
\eqref{eq:unnormalized-conversion} by the common partition function
gives~\eqref{eq:correlation}.
\end{proof}

\begin{cor}[Two-point function as a spin correlation]
\label{cor:twopoint-spin}
For $A\in\{\mathrm{mdd},\mathrm{per}\}$ and $x,y\in V$,
\begin{equation}\label{eq:twopoint-spin}
G^A_{G,w}(x,y)=\langle S_x^1S_y^1\rangle_{G,w,A}.
\end{equation}
\end{cor}

\begin{proof}
Apply Proposition~\ref{prop:conversion} with
$u^1=\delta_x+\delta_y$ and $u^2=0$. In the monomer double-dimer
case, deleting the two source-edges gives exactly a configuration in
$\Omega_{x,y}$; when $x=y$, the enlarged event is empty, consistently with
$G^{\mathrm{mdd}}_{G,w}(x,x)=0$. In the permutation case, the component
joining the two source-edges is a monochromatic path from $x$ to $y$, while
all other components are monomers or monochromatic cycles. Orienting the path
from $x$ to $y$ and using the cycle-orientation correspondence from the
previous proof gives the configurations in $\Omega^{\mathrm{per}}_{x,y}$,
with the same total weight. This also covers $x=y$, when the two source
dimers isolate $x$. Division by the corresponding partition function proves
the claim.
\end{proof}

\begin{prop}
\label{prop:loop-connection-spin}
Let \(L\in2\mathbb N\) and let \(w\) be {lazy-bipartite}. For distinct
\(x,y\in\T_L\),
\begin{equation}\label{eq:loop-connection-spin}
2\left\langle
S_x^1S_y^1\overline{S_x^2}\,\overline{S_y^2}
\right\rangle_{L,w,\mathrm{mdd}}
=
\begin{cases}
\mathbb P^{\mathrm{mdd}}_{L,w}(x\leftrightarrow y),
&x-y\in\T_L^{\mathrm o},\\
0,&x-y\in\T_L^{\mathrm e}.
\end{cases}
\end{equation}
\end{prop}

\begin{proof}
Since \(w\) is {lazy-bipartite},
\(w^{(L)}_{u,v}=0\) whenever \(u\neq v\) have the same parity, while
\(w^{(L)}_{z,z}=w(0)\) may be non-zero. Expanding the Hamiltonian as in
the proof of Proposition~\ref{prop:conversion}, orthogonality shows
that the insertion
\(S_x^1S_y^1\overline{S_x^2}\,\overline{S_y^2}\)
forces the colour-degrees to be \((0,2)\) at \(x\) and \(y\), and
either \((0,0)\) or \((1,1)\) at every other vertex. The vertices with
colour-degree \((0,0)\) are isolated monomers. After fixing their set,
the remaining part of the configuration is fully packed. Such a
configuration can exist only when \(x\) and \(y\) have opposite parity.
Its component containing \(x\) is a cycle containing \(y\), along
which the colours alternate except at \(x\) and \(y\).

Switching the colours along either of the two arcs from \(x\) to \(y\)
produces a monomer double-dimer configuration in which \(x\) and \(y\)
belong to the same loop, while leaving all monomers unchanged.
Conversely, every configuration satisfying \(x\leftrightarrow y\) is
obtained in this way with the same monomer set. If the marked loop has
length at least four, each source configuration has two distinct images
of the same weight, and the inverse switch is unique. If the marked
loop is the doubled edge \(\{x,y\}\), the two choices coincide, but the
corresponding Taylor coefficient contains the factor \(1/2!\). Since
the monomer set, and hence all diagonal factors, remain unchanged, in
both cases
\[
\mathbb P^{\mathrm{mdd}}_{L,w}(x\leftrightarrow y)
=
2\left\langle
S_x^1S_y^1\overline{S_x^2}\,\overline{S_y^2}
\right\rangle_{L,w,\mathrm{mdd}},
\]
which proves \eqref{eq:loop-connection-spin}. For equal-parity
endpoints, orthogonality and the parity constraint give zero.
\end{proof}

\begin{prop}[Edge probabilities]
\label{prop:derivationbound1}
For $x\ne y$,
\begin{equation}\label{eq:equivalence}
\mathbb P^{\mathrm{mdd}}_{G,w}(m^1_{x,y}=1)
=w_{x,y}\langle S_x^1S_y^1\rangle_{G,w,\mathrm{mdd}}.
\end{equation}
Moreover, for every $x,y\in V$,
\begin{equation}\label{eq:permutation-edge-spin}
\mathbb P^{\mathrm{per}}_{G,w}(\pi(x)=y)
=w_{x,y}\langle S_x^1S_y^1\rangle_{G,w,\mathrm{per}}.
\end{equation}
\end{prop}

\begin{proof}
For the first identity, remove the blue dimer on $\{x,y\}$ and add one blue
source-edge at each endpoint. This is a bijection onto
$\Omega_{\mathrm{en}}^{\mathrm{mdd}}(\delta_x+\delta_y,0)$, and it removes
exactly the factor $w_{x,y}$. Proposition~\ref{prop:conversion} concludes.

For permutations, removing the arrow $x\mapsto y$ leaves a bijection
$V\setminus\{x\}\to V\setminus\{y\}$; hence
\[
\mathbb P^{\mathrm{per}}_{G,w}(\pi(x)=y)
=w_{x,y}G^{\mathrm{per}}_{G,w}(y,x).
\]
Since the weights are symmetric, inversion of the partial bijection gives
$G^{\mathrm{per}}_{G,w}(y,x)=G^{\mathrm{per}}_{G,w}(x,y)$. The result now
follows from Corollary~\ref{cor:twopoint-spin}. The same argument with $x=y$
amounts to deleting the fixed point at $x$.
\end{proof}

\begin{prop}[Uniform upper bounds on correlations and probabilities]
\label{prop:probabilityestimate}
Suppose that \(w\) is {admissible and $b(w)>0$}. For every even \(L\),
\(x\neq y\in\mathbb T_L\), and
\(A\in\{\mathrm{mdd},\mathrm{per}\}\),
\begin{equation}\label{eq:uniform-correlation-bound}
G^A_{L,w}(x,y)
\leq
\frac{e}{{b(w)}}.
\end{equation}
Consequently,
\begin{align}
\mathbb P^{\mathrm{mdd}}_{L,w}(m^1_{x,y}=1)
&\leq
\frac{e}{{b(w)}}\,w^{(L)}_{x,y},
\label{eq:blue-edge-bound}\\
\mathbb P^{\mathrm{per}}_{L,w}(\pi(x)=y)
&\leq
\frac{e}{{b(w)}}\,w^{(L)}_{x,y},
\label{eq:permutation-edge-bound}\\
\mathbb P^{\mathrm{mdd}}_{L,w}\bigl(n_x(m^1)=0\bigr)
&\leq
\frac{e\,{w^{(L)}(0)}}{{b(w)}}.
\label{eq:monomer-bound}
\end{align}
\end{prop}

\begin{proof}
Put
$
N:=\frac{|\mathbb T_L|}{2}.
$
Let \(\mathcal D^{\mathrm o}_L\) be the set of perfect matchings using
only edges joining the two parity classes, and define
\[
Z^{\mathrm{dim},\mathrm o}_{L,w}
:=
\sum_{D\in\mathcal D^{\mathrm o}_L}
\prod_{\{x,y\}\in D}w^{(L)}_{x,y}.
\]
After enumerating the even and odd vertices by \(1,\ldots,N\), the matrix
\[
B_{x,y}
:=
\frac{w^{(L)}_{x,y}}{b(w)},
\qquad
x\in\mathbb T_L^{\mathrm e},
\quad
y\in\mathbb T_L^{\mathrm o},
\]
is doubly stochastic. Hence, by the Egorychev--Falikman theorem{~\cite{Egorychev1981,Falikman1981}},
\begin{equation}\label{eq:dimer-partition-lower}
Z^{\mathrm{dim},\mathrm o}_{L,w}
\geq
b(w)^N\frac{N!}{N^N},
\qquad
\bigl(Z^{\mathrm{dim},\mathrm o}_{L,w}\bigr)^{1/N}
\geq
\frac{b(w)}{e}.
\end{equation}

Let \(Z^{\mathrm{dim}}_{L,w}\) denote the unrestricted dimer partition
function. Clearly,
$
Z^{\mathrm{dim}}_{L,w}
\geq
Z^{\mathrm{dim},\mathrm o}_{L,w}.
$
Moreover, for both models,
\begin{equation}\label{eq:model-partition-lower}
Z^A_{L,w}
\geq
\bigl(Z^{\mathrm{dim}}_{L,w}\bigr)^2.
\end{equation}
For the monomer double-dimer model this follows by restricting to
fully packed configurations. For random lattice permutations, the
right-hand side is the total weight of permutations whose cycles all
have even length, and hence is bounded above by
\(Z^{\mathrm{per}}_{L,w}\).

Proposition~\ref{prop:conversion}, applied with
\(u^1\equiv1\) and \(u^2\equiv0\), gives
\[
\left\langle
\prod_{z\in\mathbb T_L}S_z^1
\right\rangle_{L,w,A}
=
\frac{Z^{\mathrm{dim}}_{L,w}}{Z^A_{L,w}}.
\]
Indeed, every vertex has one blue source-edge; the local constraint
therefore forces the original edges to form a dimer configuration of
the other colour in the monomer double-dimer model and of the same
colour in the permutation model. By \eqref{eq:model-partition-lower},
\[
\left\langle
\prod_{z\in\mathbb T_L}S_z^1
\right\rangle_{L,w,A}
\leq
\frac{1}{Z^{\mathrm{dim}}_{L,w}}
\leq
\frac{1}{Z^{\mathrm{dim},\mathrm o}_{L,w}}.
\]

The chessboard estimate now yields
\begin{align*}
G^A_{L,w}(x,y)
&=
\langle S_x^1S_y^1\rangle_{L,w,A} 
 \leq
\left\langle
\prod_{z\in\mathbb T_L}S_z^1
\right\rangle_{L,w,A}^{2/|\mathbb T_L|} 
 \leq
\bigl(Z^{\mathrm{dim},\mathrm o}_{L,w}\bigr)^{-1/N}
\leq
\frac{e}{b(w)},
\end{align*}
where the last inequality follows from
\eqref{eq:dimer-partition-lower}. This proves
\eqref{eq:uniform-correlation-bound}.
The identities of Proposition~\ref{prop:derivationbound1}
then give \eqref{eq:blue-edge-bound} and
\eqref{eq:permutation-edge-bound}.

Finally, Proposition~\ref{prop:conversion}, applied with
\(u^1=u^2=\delta_x\), gives
\begin{equation}\label{eq:monomer-spin-identity}
\mathbb P^{\mathrm{mdd}}_{L,w}\bigl(n_x(m^1)=0\bigr)
=
{w^{(L)}(0)}\,
\langle S_x^1S_x^2\rangle_{L,w,\mathrm{mdd}}.
\end{equation}
The chessboard estimate and Proposition~\ref{prop:conversion}, now
applied with \(u^1\equiv u^2\equiv1\), imply
\begin{align*}
\langle S_x^1S_x^2\rangle_{L,w,\mathrm{mdd}}
&\leq
\left\langle
\prod_{z\in\mathbb T_L}S_z^1S_z^2
\right\rangle_{L,w,\mathrm{mdd}}^{1/|\mathbb T_L|}
=
\bigl(Z^{\mathrm{mdd}}_{L,w}\bigr)^{-1/|\mathbb T_L|}
\leq
\bigl(Z^{\mathrm{dim},\mathrm o}_{L,w}\bigr)^{-1/N}
\leq
\frac{e}{b(w)}.
\end{align*}
Combining this with \eqref{eq:monomer-spin-identity} proves
\eqref{eq:monomer-bound}.
\end{proof}

\section{\texorpdfstring{{Long-range order}}{Long-range order} and macroscopic loops}
\label{sect:sufficientcond}
In this section we provide  general
sufficient conditions for the occurrence of {long-range order} and macroscopic loops.
The section is organised as follows. 
We first show that  long-range order implies the occurrence of macroscopic loops under the assumption that the edge weights are bipartite.  
Here we use an idea from \cite{Kenyonclaire},
which adapts naturally to the bipartite case. 
We then prove our main {theorems, Theorem~\ref{thm:longrangeorder} and Theorem~\ref{thm:longrangeorder2}, on long-range order.}

\subsection{\texorpdfstring{{Long-range order}}{Long-range order} implies macroscopic loops}

\begin{prop}\label{prop:monomerloop-lowerbound}
Let \(L\in2\mathbb N\), and let
\(w : \T_L \times \T_L \to [0,\infty)\) be symmetric, translation
invariant, nonzero, and bipartite with respect to
\(\T_L^{\mathrm e}\cup\T_L^{\mathrm o}\).
Fix any \(a\in\T_L^{\mathrm o}\) with \(w_{0,a}>0\). Then, for every \(x\in \T_L^{\mathrm o}\),
\[
w_{0,a}^2\,\mathcal{C}_{L,w}(0,x)^2
\leq
\mathbb P^{\mathrm{d.d.}}_{L,w}(0 \leftrightarrow x).
\]
Consequently,
\begin{equation}\label{eq:CauchySchwarzapp}
\frac{1}{|\T_L^{\mathrm o}|}
\sum_{x \in \T_L^{\mathrm o}} \mathcal{C}_{L,w}(0,x)
\leq
\frac{1}{w_{0,a}}
\bigg(
\frac{1}{|\T_L^{\mathrm o}|}
\sum_{x \in \T_L^{\mathrm o}}
\mathbb P^{\mathrm{d.d.}}_{L,w}(0 \leftrightarrow x)
\bigg)^{\frac12}.
\end{equation}

If
\(\sum_{y\in\T_L}w_{0,y}=1\), then
\begin{equation}\label{eq:averaged-switching-consequence}
\frac{\mathbb E^{\mathrm{d.d.}}_{L,w}|\mathcal{L}_0|}{|\T_L|}
\geq
2\bigl(M^{\mathrm{mdd}}_{L,w}\bigr)^2.
\end{equation}
{
Under the same normalization, for every $x\in\T_L^{\mathrm o}$,
\begin{equation}\label{eq:reverse-switching-bound}
\mathcal C_{L,w}(0,x)
\ge w_{0,a}\,\mathbb P^{\mathrm{d.d.}}_{L,w}(0\leftrightarrow x).
\end{equation}
}

\end{prop}

\begin{figure}[htbp]
    \centering
    \begin{tikzpicture}[
        scale=1.0,
        dot/.style={circle, fill=black, inner sep=1.5pt},
        bluebond/.style={-, blue, ultra thick, shorten >=3pt, shorten <=3pt},
        bluejump/.style={-, blue, ultra thick, bend left=30, shorten >=3pt, shorten <=3pt},
        redbond/.style={-, red, very thick, shorten >=3pt, shorten <=3pt},
        redjump/.style={-, red, very thick, bend left=30, shorten >=3pt, shorten <=3pt},
        labelnode/.style={font=\small\bfseries}
    ]

    \def\drawGrid{
        \foreach \x in {0,...,5} {
            \foreach \y in {0,...,3} {
                \node[dot] at (\x,\y) {};
            }
        }
    }

    \def\definePoints{
        \coordinate (o) at (1,1);
        \coordinate (e1) at (2,1);
        \coordinate (x) at (4,1);
        \coordinate (xe1) at (5,1);
    }

    \def\drawLabels{
        \node[labelnode, above left] at (o) {$0$};
        \node[labelnode, above right] at (e1) {$e_1$};
        \node[labelnode, below left] at (x) {$x$};
        \node[labelnode, below right] at (xe1) {$x+e_1$};
    }

    \begin{scope}[xshift=0cm]
        \drawGrid
        \definePoints
        \drawLabels

        \draw[redbond] (o) -- (1,0);
        \draw[bluebond] (1,0) -- (2,0);
        \draw[redbond] (2,0) -- (3,0);
        \draw[bluebond] (3,0) -- (3,1);
        \draw[redbond] (3,1) -- (x);

        \draw[bluejump] (e1) to node[auto, swap, black, font=\footnotesize] {$\Gamma$} (4,2);
        \draw[redbond] (4,2) -- (5,2);
        \draw[bluebond] (5,2) -- (xe1);

        \draw[redbond] (0,3) -- (1,3);
        \draw[bluebond] (1,3) -- (1,2);
        \draw[redbond] (1,2) -- (0,2);
        \draw[bluebond] (0,2) -- (0,3);

        \node[below, font=\bfseries] at (2.5,-0.8)
        {Left: Monomer sets \(\{0,x\}\) and \(\{e_1,x+e_1\}\)};
    \end{scope}

    \begin{scope}[xshift=6cm, yshift=1.5cm]
        \draw[->, very thick, >=latex] (0,0) -- (1,0)
        node[midway, above] {$\phi$};
        \node[align=center, below, font=\footnotesize] at (0.5,0)
        {Switch colours\\on \(\Gamma\)};
    \end{scope}

    \begin{scope}[xshift=8cm]
        \drawGrid
        \definePoints
        \drawLabels

        \draw[redbond] (o) -- (1,0);
        \draw[bluebond] (1,0) -- (2,0);
        \draw[redbond] (2,0) -- (3,0);
        \draw[bluebond] (3,0) -- (3,1);
        \draw[redbond] (3,1) -- (x);

        \draw[bluebond] (o) -- (e1);
        \draw[bluebond] (x) -- (xe1);

        \draw[redjump] (e1) to (4,2);
        \draw[bluebond] (4,2) -- (5,2);
        \draw[redbond] (5,2) -- (xe1);

        \draw[redbond] (0,3) -- (1,3);
        \draw[bluebond] (1,3) -- (1,2);
        \draw[redbond] (1,2) -- (0,2);
        \draw[bluebond] (0,2) -- (0,3);

        \node[below, font=\bfseries] at (2.5,-0.8)
        {Right: Single loop visiting \(0\) and \(x\)};
    \end{scope}

    \end{tikzpicture}
    \caption{An example of the map \(\phi\) with \(a=e_1\) when
    \(x\notin\{\pm e_1\}\). Blue edges belong to the first dimer
    configuration and red edges to the second one. The map first inserts
    the two blue edges \(\{0,e_1\}\) and \(\{x,x+e_1\}\) and then switches
    the colours along the unique arc \(\Gamma\) from \(e_1\) to
    \(x+e_1\) of the resulting cycle that avoids \(x\).}
\end{figure}

\begin{proof}
The matching
\(\{\{z,z+a\}:z\in\T_L^{\mathrm e}\}\) has positive weight, so the
dimer measure is well defined.
For the first claim, we adapt the switching injection from
\cite[proof of Theorem~2 and Figure~2]{Kenyonclaire} to the present
weighted setting. Since \(w\) is bipartite, the same parity argument
applies to every pair of configurations of positive weight.

It is enough to prove that, for every \(x\in\T_L^{\mathrm o}\),
\begin{equation}\label{eq:monomermonomerloop-again}
w_{0,a}^2\,Z_{L,w}^{\mathrm{dim}}(\{0,x\})^2
\leq
\bigl(Z_{L,w}^{\mathrm{dim}}(\emptyset)\bigr)^2
\mathbb P^{\mathrm{d.d.}}_{L,w}(0\leftrightarrow x).
\end{equation}

If \(x\in\{\pm a\}\), add the edge \(\{0,x\}\) to both configurations
in each ordered pair counted by
\(Z_{L,w}^{\mathrm{dim}}(\{0,x\})^2\).
This gives an injection into the double-dimer configurations for which
\(0\leftrightarrow x\), since the
resulting double-dimer configuration contains the doubled edge
\(\{0,x\}\), and multiplies the weight by
$
w_{0,x}^2=w_{0,a}^2,
$
where symmetry and translation invariance are used when \(x=-a\).
Thus \eqref{eq:monomermonomerloop-again} holds in this case.

Now let \(x\notin\{\pm a\}\). By translation invariance,
\[
Z_{L,w}^{\mathrm{dim}}(\{0,x\})
=
Z_{L,w}^{\mathrm{dim}}(\{a,x+a\}).
\]
It therefore suffices to construct an injection from the ordered pairs
counted by
\[
Z_{L,w}^{\mathrm{dim}}(\{0,x\})
Z_{L,w}^{\mathrm{dim}}(\{a,x+a\})
\]
into the double-dimer configurations for which
\(0\leftrightarrow x\), multiplying the weight by
\(w_{0,a}^2\).

Take an ordered pair \((D^{\mathrm b},D^{\mathrm r})\) contributing to
this product, colour the edges of \(D^{\mathrm b}\) blue and those of
\(D^{\mathrm r}\) red. Their union
consists of cycles and two alternating paths with endpoints
\(0,x,a,x+a\). The vertices \(0,x\) are incident only to red
edges, whereas \(a,x+a\) are incident only to blue edges.

An alternating path joining an endpoint incident only to a red edge to
one incident only to a blue edge has even length. Since \(0,a\) and
\(x,x+a\) belong to opposite
bipartition classes, neither path can join \(0\) to \(a\) or \(x\)
to \(x+a\). Consequently, adding the two blue edges
\(\{0,a\}\) and \(\{x,x+a\}\) joins the two paths into a single
cycle through all four marked vertices.

Let \(\Gamma\) be the unique arc of this cycle from \(a\) to
\(x+a\) that avoids \(x\), and let \(E(\Gamma)\) denote its edge
set. Here \(\triangle\) denotes symmetric difference. Define
\[
D_1'
:=
\Bigl(
D^{\mathrm b}
\cup
\bigl\{\{0,a\},\{x,x+a\}\bigr\}
\Bigr)
\triangle E(\Gamma),
\qquad
D_2'
:=
D^{\mathrm r}\triangle E(\Gamma).
\]
Along \(\Gamma\), the colours alternate and its first and last edges
are blue. Thus the colour switch leaves every internal vertex incident
to one edge of each colour and changes one of the two blue edges at
each of \(a\) and \(x+a\) into a red edge. Hence \(D_1'\) and
\(D_2'\) are perfect matchings of \(\T_L\), and their union contains
the same cycle through \(0\) and \(x\).

The map \(\phi:(D^{\mathrm b},D^{\mathrm r})\mapsto(D_1',D_2')\) is
injective. Indeed, the uncoloured cycle containing
\(\{0,a\}\) and \(\{x,x+a\}\), as well as its arc \(\Gamma\),
is unchanged. Switching the colours along \(\Gamma\) back and then
removing these two edges uniquely recovers
\((D^{\mathrm b},D^{\mathrm r})\).

The colour switch does not change the product of the weights, while the
two inserted edges contribute
$
w_{0,a}w_{x,x+a}=w_{0,a}^2.
$
Summing over all pairs proves
\eqref{eq:monomermonomerloop-again}. Dividing by
\(\bigl(Z_{L,w}^{\mathrm{dim}}(\emptyset)\bigr)^2\) proves the first
claim. Finally, taking square roots, summing over
\(x\in\T_L^{\mathrm o}\), and applying the Cauchy--Schwarz inequality
gives \eqref{eq:CauchySchwarzapp}.

We next prove \eqref{eq:averaged-switching-consequence}, for
which no assumption on \(w_{0,e_1}\) is needed. Fix
\(x\in\T_L^{\mathrm o}\). For every \(u,y\in\T_L\) such that
\(w_{0,u}w_{x,y}>0\), take a blue matching with monomers \(0,x\) and
a red matching with monomers \(u,y\). Then \(0,y\) are even and
\(x,u\) are odd, so the only possible coincidences among these four
vertices are \(u=x\) and \(y=0\).

Suppose first that the four vertices are distinct. Insert the two blue
edges \(\{0,u\}\) and \(\{x,y\}\). An alternating path from an
endpoint incident only to a red edge to one incident only to a blue
edge has even length. Since \(0,u\), as well as \(x,y\), have opposite
parity, neither inserted edge closes one of the two paths. The two
inserted edges therefore join the paths into a single cycle. Switching
the colours along the unique arc from \(u\) to \(y\) that avoids
\(x\) produces two perfect matchings whose loop through \(0\) contains
\(x\). The edge \(\{x,y\}\) is not switched and remains blue.

If two of the marked vertices coincide, no colour switch is needed.
If \(u=x\) and \(y\neq0\), insert \(\{0,x\}\) in blue and
\(\{x,y\}\) in red. If \(y=0\) and \(u\neq x\), insert
\(\{0,x\}\) in blue and \(\{0,u\}\) in red. If \(u=x\) and
\(y=0\), insert \(\{0,x\}\) in both colours. In each case the result
is a pair of perfect matchings whose loop through \(0\) contains \(x\),
and the inserted edges contribute the factor \(w_{0,u}w_{x,y}\), using
symmetry of \(w\) when \(y=0\).

This map is at most two-to-one. If its image does not contain the blue
edge \(\{0,x\}\), the four vertices were distinct. The vertex \(y\)
is then the blue partner of \(x\), and \(u\) is one of the two
coloured partners of \(0\). For either choice of \(u\), the arc to be
switched back and the original matchings are uniquely determined. If
the image contains the blue edge \(\{0,x\}\), there is at most one
preimage with \(u=x\) and at most one with \(y=0\); when
\(\{0,x\}\) is doubled, these two descriptions coincide. Hence
\[
\mathcal C_{L,w}(0,x)
\sum_{u,y\in\T_L}
w_{0,u}w_{x,y}\mathcal C_{L,w}(u,y)
\leq
2\mathbb P^{\mathrm{d.d.}}_{L,w}(0\leftrightarrow x).
\]
Since \(\mathcal C_{L,w}(0,x)=0\) for
\(x\in\T_L^{\mathrm e}\), the sum of the left-hand side over
\(x\in\T_L^{\mathrm o}\) may be extended to all of \(\T_L\).
Translation invariance and Parseval's identity give
\begin{align*}
&\sum_{x,u,y\in\T_L}
\mathcal C_{L,w}(0,x)w_{0,u}w_{x,y}\mathcal C_{L,w}(u,y)
=
\frac1{|\T_L|}
\sum_{k\in\T_L^*}
|\widehat w(k)|^2
\left|
\sum_{x\in\T_L}
e^{\mathrm i k\cdot x}\mathcal C_{L,w}(0,x)
\right|^2.
\end{align*}
Consequently,
\begin{align*}
\mathbb E^{\mathrm{d.d.}}_{L,w}|\mathcal L_0|
&=
2\sum_{x\in\T_L^{\mathrm o}}
\mathbb P^{\mathrm{d.d.}}_{L,w}(0\leftrightarrow x)
\geq
\frac1{|\T_L|}
\sum_{k\in\T_L^*}
|\widehat w(k)|^2
\left|
\sum_{x\in\T_L}
e^{\mathrm i k\cdot x}\mathcal{C}_{L,w}(0,x)
\right|^2.
\end{align*}
Here the first identity uses the fact that every bipartite loop has
equally many even and odd vertices. By
\eqref{eq:twopointandmonomer} for \(x\neq0\), while both sides vanish
at \(x=0\),
\[
\sum_{x\in\T_L}\mathcal{C}_{L,w}(0,x)
=
|\T_L|M^{\mathrm{mdd}}_{L,w}.
\]
Moreover, bipartiteness implies that
\(\mathcal{C}_{L,w}(0,x)=0\) when \(x\in\T_L^{\mathrm e}\). Hence, with
\(\boldsymbol\pi=(\pi,\ldots,\pi)\),
\[
\widehat w(0)=1,
\qquad
\widehat w(\boldsymbol\pi)=-1,
\qquad
\sum_{x\in\T_L}
e^{\mathrm i\boldsymbol\pi\cdot x}\mathcal{C}_{L,w}(0,x)
=
-|\T_L|M^{\mathrm{mdd}}_{L,w}.
\]
Retaining the modes \(0\) and \(\boldsymbol\pi\) proves
\eqref{eq:averaged-switching-consequence}.

{
We conclude with the reverse estimate \eqref{eq:reverse-switching-bound}.
Fix an even vertex \(x\) and an odd vertex \(y\). Take an ordered pair
\((B,R)\) of perfect matchings for which \(x\) and \(y\)
belong to the same alternating loop. Let \(u\) be the blue partner of \(x\)
and \(v\) the red partner of \(y\). Delete \(xu\) and \(yv\). The marked loop
becomes two alternating paths{, with paths of length zero allowed when
marked endpoints coincide}. Since one removed edge is blue and the other
red, each path joins an endpoint of one removed edge to an endpoint of the
other and has even length. The endpoint parities therefore force the pairs
\((u,y)\) and \((x,v)\). After switching the colours on the first path, the
blue matching leaves \(x,y\) unmatched and the red matching leaves \(u,v\)
unmatched.

For fixed \(u,v\), this map is injective: in every image pair the unique
possibly zero-length path from \(u\) to \(y\) identifies the switched edges;
switching its colours back and restoring the two deleted dimers recovers the
input. The input weight equals
\(w_{x,u}w_{y,v}\) times the output weight. Summing first over the
image and then over all pairs with these prescribed unmatched vertices gives
\[
 \bigl(Z^{\mathrm{dim}}_{L,w}\bigr)^2
 \mathbb P^{\mathrm{d.d.}}_{L,w}(x\leftrightarrow y)
 \le
 \sum_{u,v}w_{x,u}w_{y,v}
 Z^{\mathrm{dim}}_{L,w}(x,y)Z^{\mathrm{dim}}_{L,w}(u,v).
\]
After division by \((Z^{\mathrm{dim}}_{L,w})^2\),
\begin{equation}
 \mathbb P^{\mathrm{d.d.}}_{L,w}(x\leftrightarrow y)
 \le \mathcal C_{L,w}(x,y)
 \sum_{u,v}w_{x,u}w_{y,v}
 \mathcal C_{L,w}(u,v),
 \label{eq:reverse-double-dimer-switch}
\end{equation}
The vertices $u,v$ have opposite parity. The first claim and
translation invariance give $\mathcal C_{L,w}(u,v)\le w_{0,a}^{-1}$,
and the row sums of $w$ are one. Thus the sum on the right-hand side of
\eqref{eq:reverse-double-dimer-switch} is at most $w_{0,a}^{-1}$,
which proves \eqref{eq:reverse-switching-bound}.
}
\end{proof}

\subsection{Proof of \texorpdfstring{{Theorems}}{Theorems}~\ref{thm:longrangeorder} and~\ref{thm:longrangeorder2}}
\label{sect:prooftheorems}
Recall that \(M_{L,w}\) denotes the Cesàro average of the two-point
function. Both arguments rely on the infrared bound
\begin{equation}\label{eq:LROinfrared}
\widehat G_{L,w}(k)
\leq \frac{1}{1-\widehat w(k)},
\qquad k\neq 0,
\end{equation}
which follows from Proposition~\ref{prop:infrared}, since
\(\varepsilon(k)=1-\widehat w(k)\).

Unlike in classical spin systems \cite{FrohlichLiebSimon}, the main difficulty in
deriving long-range order in our setting is that one needs to bound from
above the quantity
\begin{equation}\label{eq:issue}
\sum_{k\neq 0}\widehat w(k)\widehat G_{L,w}(k)
\end{equation}
using the infrared bound. This cannot be done directly, since
\(\widehat w(k)\) need not be non-negative and no information is
available on the sign of \(\widehat G_{L,w}(k)\).

The proof of Theorem~\ref{thm:longrangeorder} overcomes this difficulty
by exploiting the symmetry properties of \(\widehat G_{L,w}(k)\)
induced by the bipartite structure of the weights. The proof of
Theorem~\ref{thm:longrangeorder2}, instead, relies on a diagonal-shift
argument.

\begin{proof}[Proof of Theorem~\ref{thm:longrangeorder}]
Since \(L\) is even and \(w\) is bipartite, \(w^{(L)}(0)=0\), the
model is fully packed, and
\[
\widehat w(k+\boldsymbol\pi)=-\widehat w(k),\qquad
\widehat G^{\mathrm{mdd}}_{L,w}(k+\boldsymbol\pi)
=-\widehat G^{\mathrm{mdd}}_{L,w}(k),
\qquad
\boldsymbol\pi=(\pi,\ldots,\pi).
\]
Proposition~\ref{prop:derivationbound1} and Fourier inversion therefore
give
\[
1=\sum_{x\in\T_L}w^{(L)}(x)G^{\mathrm{mdd}}_{L,w}(x)
=2M^{\mathrm{mdd}}_{L,w}
+\frac1{|\T_L|}
\sum_{k\notin\{0,\boldsymbol\pi\}}
\widehat w(k)\widehat G^{\mathrm{mdd}}_{L,w}(k).
\]

Choose \(\mathbb H^+\subset\T_L^*\) to contain exactly one point from
each pair \(\{k,k+\boldsymbol\pi\}\), choosing the representative with
\(\widehat w(k)\geq0\) and resolving ties arbitrarily. Pairing the two
modes and using \eqref{eq:LROinfrared}, we obtain
\begin{align*}
&\frac1{|\T_L|}
\sum_{k\notin\{0,\boldsymbol\pi\}}
\widehat w(k)\widehat G^{\mathrm{mdd}}_{L,w}(k)\\
&\qquad=
\frac2{|\T_L|}
\sum_{k\in\mathbb H^+\setminus\{0\}}
\widehat w(k)\widehat G^{\mathrm{mdd}}_{L,w}(k)
\leq
\frac1{|\T_L|}
\sum_{k\notin\{0,\boldsymbol\pi\}}
\frac{|\widehat w(k)|}{1-|\widehat w(k)|}.
\end{align*}
For $|t|<1$, the identity
\[
\frac{|t|}{1-|t|}
=\frac1{1-t}+\frac1{1+t}-\frac{2+|t|}{1+|t|},
\]
together with $\widehat w(k+\boldsymbol\pi)=-\widehat w(k)$ and
\eqref{eq:torus-green-convergence}, gives convergence of the last sum
to $I_w$, since the last term in the identity is bounded and continuous.
The torus Green-function/Riemann-sum convergence yields
\[
2\liminf_{\substack{L\to\infty\\L\in2\mathbb N}}
M^{\mathrm{mdd}}_{L,w}
\geq1-I_w,
\qquad
I_w:=
\int_{(-\pi,\pi]^d}
\frac{|\widehat w(k)|}{1-|\widehat w(k)|}
\frac{\mathrm dk}{(2\pi)^d}.
\]
By bipartiteness,
\[
g_{d,w}
=
\int_{(-\pi,\pi]^d}
\frac1{1-\widehat w(k)^2}\frac{\mathrm dk}{(2\pi)^d}.
\]
Using
$
\frac{|t|}{1-|t|}
=
\frac{t^2}{1-t^2}+\frac{|t|}{1-t^2}
$
and the Cauchy--Schwarz inequality gives
\[
I_w
\leq
g_{d,w}-1+
\sqrt{g_{d,w}\bigl(g_{d,w}-1\bigr)}
=
h_{d,w}-1.
\]
{Here and below,
$h_{d,w}:=g_{d,w}+\sqrt{g_{d,w}(g_{d,w}-1)}$.}
Consequently,
\begin{equation}\label{eq:LRO-bipartite-Cesàro}
\liminf_{\substack{L\to\infty\\L\in2\mathbb N}}
M^{\mathrm{mdd}}_{L,w}
\geq\frac{2-h_{d,w}}2.
\end{equation}
If \(g_{d,w}<4/3\), the right-hand side is positive, and
\eqref{eq:averaged-switching-consequence} immediately implies
\eqref{eq:macroscopicloops}, since
\[
\liminf_{\substack{L\to\infty\\L\in2\mathbb N}}
\frac{\mathbb E^{\mathrm{mdd}}_{L,w}|\mathcal L_0|}{|\T_L|}
\geq\frac{(2-h_{d,w})^2}{2}>0.
\]

It remains to prove the pointwise statements. By
\eqref{eq:complexalternation}, the families
$
F_x=S_x^1$
and 
$\tilde F_x=\sqrt2\,S_x^1\overline{S_x^2}
$
are reflection invariant. Corollary~\ref{cor:twopoint-spin} and
Proposition~\ref{prop:loop-connection-spin} identify the corresponding
two-point functions with \(G^{\mathrm{mdd}}_{L,w}\) and, respectively,
with the loop-connection probability on the odd sublattice.

The required uniform upper bounds follow from
Proposition~\ref{prop:probabilityestimate} for
\(G^{\mathrm{mdd}}_{L,w}\) and 
are trivial for the loop-connection probability.

Consequently Proposition~\ref{prop:monotonicity} applies to both functions. In
particular, their values at odd points of each coordinate axis are
symmetric and non-increasing in the torus distance.
Hence, the standard Cesàro-to-pointwise argument based on site monotonicity,
as in the proof of
\cite[Theorem~2.2(i)]{LeesTaggiJSP2021}, now yields constants \(C,c>0\)
such that, for every sufficiently large even \(L\) and every odd
\(n\in(0,cL)\),
$$
G^{\mathrm{mdd}}_{L,w}(ne_1)\geq C,
\qquad
\mathbb P^{\mathrm{mdd}}_{L,w}(ne_1\in\mathcal L_0)\geq C.
$$After decreasing the
constants if necessary, this proves both pointwise statements.
\end{proof}

\begin{rem}
\label{rem:sharper-LRO-condition}
The condition \(g_{d,w}<4/3\) is only a convenient sufficient criterion. The proof of Theorem~\ref{thm:longrangeorder} actually yields long-range order and macroscopic loops under the weaker condition
\begin{equation}
\label{eq:sharpquantity}
I_w:=
\int_{(-\pi,\pi]^d}
\frac{|\widehat w(k)|}{1-|\widehat w(k)|}
\frac{\mathrm dk}{(2\pi)^d}<1.
\end{equation}
The condition
\(g_{d,w}<4/3\) is then obtained from the Cauchy--Schwarz bound
\(I_w\leq h_{d,w}-1\).
\end{rem}

\begin{proof}[Proof of Theorem~\ref{thm:longrangeorder2}]
If $g_{d,w}=\infty$, the asserted lower bounds are vacuous; hence assume $g_{d,w}<\infty$.
{Set $N:=|\T_L|$.}
For \(A\in\{\mathrm{per},\mathrm{mdd}\}\), let
\[
p^A_{L,w}:=
\begin{cases}
0,&A=\mathrm{per},\\
\mathbb P^{\mathrm{mdd}}_{L,w}(n_0(m^1)=0),&A=\mathrm{mdd}.
\end{cases}
\]
Proposition~\ref{prop:derivationbound1} and Fourier inversion give
\[
1-p^A_{L,w}
=\sum_{x\in\T_L}w^{(L)}(x)G^A_{L,w}(x)
=M^A_{L,w}+
\frac1N\sum_{k\neq0}\widehat w(k)\widehat G^A_{L,w}(k).
\]
Since \(\widehat w(k)+a_w\geq0\), adding and subtracting \(a_w\) and
using \eqref{eq:LROinfrared} yields
\begin{align*}
1-p^A_{L,w}-M^A_{L,w}
&\leq\frac1N\sum_{k\neq0}
\frac{\widehat w(k)+a_w}{1-\widehat w(k)}
-a_w\bigl(G^A_{L,w}(0)-M^A_{L,w}\bigr)\\
&=(1+a_w)\frac1N\sum_{k\neq0}
\frac1{1-\widehat w(k)}-1+\frac1N
-a_w\bigl(G^A_{L,w}(0)-M^A_{L,w}\bigr).
\end{align*}
Thus
\begin{equation}\label{eq:LRO-finite-volume}
M^A_{L,w}\geq
\frac{2-p^A_{L,w}-N^{-1}+a_wG^A_{L,w}(0)}{1+a_w}
-\frac1N\sum_{k\neq0}\frac1{1-\widehat w(k)}.
\end{equation}
Hypothesis~\eqref{eq:torus-green-convergence} gives
\[
\frac1N\sum_{k\neq0}\frac1{1-\widehat w(k)}\longrightarrow g_{d,w}.
\]
Moreover, the proof of Proposition~\ref{prop:probabilityestimate},
applied to the periodised weights, gives
\[
p^{\mathrm{mdd}}_{L,w}
\leq\min\left\{1,\frac{e\,w^{(L)}(0)}{b(w)}\right\},
\qquad w^{(L)}(0)\longrightarrow w(0).
\]
Dropping the non-negative diagonal term in
\eqref{eq:LRO-finite-volume} and taking the limit proves
\eqref{eq:longrangequantity-per} and
\eqref{eq:longrangequantity-mdd}.

Let \(C_A>0\) denote the corresponding right-hand side. The diagonal
term does not affect the pointwise argument: for \(A=\mathrm{mdd}\),
\(G^{\mathrm{mdd}}_{L,w}(0)=0\); for \(A=\mathrm{per}\), subtracting
$G^{\mathrm{per}}_{L,w}(0)/N$ from \eqref{eq:LRO-finite-volume} leaves
the non-negative coefficient $a_w/(1+a_w)-1/N$ for all sufficiently
large $N$ when $a_w>0$, while for
\(a_w=0\) one has \(w(0)>0\) and
\[
w^{(L)}(0)G^{\mathrm{per}}_{L,w}(0)
=\mathbb P^{\mathrm{per}}_{L,w}(\pi(0)=0)\leq1.
\]
Hence
\[
\liminf_{\substack{L\to\infty\\L\in2\mathbb N}}
\frac1N\sum_{x\neq0}G^A_{L,w}(x)\geq C_A.
\]
Using Proposition~\ref{prop:probabilityestimate}, Proposition~
\ref{prop:monotonicity}, and the even-coordinate estimate of
\cite[Theorem~2.1]{LeesTaggiJSP2021}, summation over coordinate slices
gives
\[
\frac1N\sum_{x\neq0}G^A_{L,w}(x)
\leq\frac2L
\sum_{\substack{r\in\mathbb Z/L\mathbb Z\\r\ \mathrm{odd}}}
G^A_{L,w}(re_1)+\frac{e}{b(w)L}.
\]
The odd axis values are symmetric and non-increasing in the torus
distance. The claimed pointwise bound therefore follows by the standard
one-dimensional averaging argument, after choosing \(c>0\) sufficiently small.
\end{proof}

\section{Diagonal activity greater than \texorpdfstring{$1/2$}{1/2}}
\label{sect:onehalf}

\begin{proof}[Proof of Theorem~\ref{thm:decay-large-monomer}]
{Fix $A\in\{\mathrm{mdd},\mathrm{per}\}$ and suppress the superscript $A$ throughout the proof.}
For \(x\in\mathbb T_L\), throughout this proof,
\[
|x|_1:=\min_{m\in\mathbb Z^d}|\widetilde x+Lm|_1,
\]
where \(\widetilde x\) is any representative of \(x\). Write
\(w_L=w^{(L)}\) and set
\[
q_L(z):=\frac{w_L(z)}{w_L(0)}\mathbf 1_{\{z\ne0\}},\qquad
\vartheta:=\frac{1-w(0)}{w(0)}<1,\qquad
H_L(x):=\sum_{n\ge1}q_L^{*n}(x).
\]
Since \(w_L(0)\ge w(0)\) and \(\sum_z w_L(z)=1\),
\[
a_L:=\sum_{z\in\mathbb T_L}q_L(z)
=\frac{1-w_L(0)}{w_L(0)}\le\vartheta,
\qquad
\sum_{x\in\mathbb T_L}H_L(x)
=\sum_{n\ge1}a_L^n\le\frac{\vartheta}{1-\vartheta}.
\]

We first consider random lattice permutations. By symmetry, we may
regard \(G_{L,w}(x)\) as \(G_{L,w}(0,x)\). For \(x\ne0\), every
configuration contributing to this quantity contains a unique
self-avoiding directed path
\[
\gamma=(x_0,\ldots,x_n),\qquad x_0=0,\quad x_n=x.
\]
Once \(\gamma\) is fixed, its complement carries an arbitrary
permutation. Thus, with \(V(\gamma)=\{x_0,\ldots,x_n\}\),
\[
G_{L,w}(x)
=\sum_{n\ge1}
 \sum_{\substack{x_0=0,\;x_n=x\\x_0,\ldots,x_n\ {\rm distinct}}}
 \left(\prod_{i=0}^{n-1}w_L(x_{i+1}-x_i)\right)
 \frac{Z^{\rm per}_{L,w}(\mathbb T_L\setminus V(\gamma))}
      {Z^{\rm per}_{L,w}},
\]
where the numerator denotes the permutation partition function on the
induced subgraph. Replacing the vertices of \(\gamma\) by fixed points
gives
\[
w_L(0)^{n+1}
\frac{Z^{\rm per}_{L,w}(\mathbb T_L\setminus V(\gamma))}
     {Z^{\rm per}_{L,w}}\le1.
\]
Since a path has one more vertex than edges, dropping self-avoidance
yields
\begin{equation}
\label{eq:large-monomer-convolution}
G_{L,w}(x)\le\frac{1}{w_L(0)}H_L(x)
\le\frac{1}{w(0)}H_L(x).
\end{equation}

The same bound holds for the monomer double-dimer model. Indeed, every
configuration contributing to \(G_{L,w}(0,x)\) contains a unique
alternating self-avoiding path from \(0\) to \(x\). Its length is odd,
and its complement carries an arbitrary monomer double-dimer
configuration. Replacing the path vertices by monomers gives the same
partition-function bound as above; dropping the odd-length and
self-avoidance restrictions then gives
\eqref{eq:large-monomer-convolution}.

For the loop event, orient the cycle through \(0\) by its permutation
arrows in the RLP model and, in the monomer double-dimer model, by
following the blue edge out of \(0\). Splitting the cycle at \(x\),
replacing its vertices by fixed points or monomers, and dropping
simplicity and disjointness gives, for \(x\ne0\),
\[
\mathbb P_{L,w}(x\in\mathcal L_0)
\le H_L(x)H_L(-x)
\le\frac{\vartheta}{1-\vartheta}H_L(x).
\]
It therefore remains only to estimate \(H_L\).

For every \(r\ge0\), periodisation and the definition of the torus
distance give
\begin{equation}
\label{eq:periodised-tail}
\sup_L\sum_{\substack{z\in\mathbb T_L\\|z|_1\ge r}}q_L(z)
\le\frac{1}{w(0)}
\sum_{\substack{u\in\mathbb Z^d\\|u|_1\ge r}}w(u).
\end{equation}
If \(|z_1+\cdots+z_n|_1\ge R\), then
\(|z_i|_1\ge R/n\) for some \(i\). Hence, for every \(N\ge1\),
\[
\sum_{\substack{x\in\mathbb T_L\\|x|_1\ge R}}H_L(x)
\le\frac{1}{w(0)}
\sum_{n=1}^N n\vartheta^{n-1}
\sum_{\substack{u\in\mathbb Z^d\\|u|_1\ge R/n}}w(u)
+\frac{\vartheta^{N+1}}{1-\vartheta}.
\]
Letting first \(R\to\infty\) and then \(N\to\infty\) proves
\eqref{eq:qualitative-decay}.
The quantitative statements follow from the corresponding moment bounds for
$H_L$. Indeed, writing
$M_p:=\sup_L\sum_z(1+|z|_1)^p q_L(z)<\infty$, the triangle inequality gives
\[
\sum_x(1+|x|_1)^p H_L(x)
\leq M_p\sum_{n\geq1}n^{\max\{1,p\}}\vartheta^{n-1}<\infty.
\]
Under the exponential-moment assumption, choose $c>0$ so small that
\[ B(c):=w(0)^{-1}\sum_{z\ne0}e^{c|z|_1}w(z)<1. \]
Then
$\sum_x e^{c|x|_1}H_L(x)\leq B(c)/(1-B(c))$.
At $x=0$, enlarge the constants using
$G^{\mathrm{per}}_{L,w}(0)\leq w(0)^{-1}$,
$G^{\mathrm{mdd}}_{L,w}(0)=0$, and
$\mathbb P^A_{L,w}(0\in\mathcal L_0)=1$.
%
\end{proof}

\section{Absence of long-range order in dimension one:
proof of Theorem~\ref{thm:positive-activity-one-dimensional}}
\label{sect:decayonedimension}

{
We use the permutation representation described after
\eqref{eq:lazybipartitexcondition}: fixed points represent monomers, and
non-trivial cycles represent double-dimer loops of the same length.
On an interval, we show that the largest gap between common free cuts of
two independent permutations is negligible compared with the interval
length. A completion argument extends this estimate to partial bijections;
expanding the periodised weights then gives the required bounds on the
torus.
}

For later use, set
\[
 M_1:=\sum_{r\in\mathbb Z}|r|w(r)<\infty.
\]

\subsection{\texorpdfstring{Completion and interval partition functions}{Completion and interval partition functions}}

For a finite set \(V\), write \(\Omega_V^{\mathrm{per}}\) for the set of
permutations of \(V\). If \(A\) is a non-negative matrix indexed by \(V\),
define
\[
 \operatorname{wt}_A(P):=\prod_{x\in V}A(x,P(x)),\qquad
 \operatorname{per}A:=\sum_{P\in\Omega_V^{\mathrm{per}}}
 \operatorname{wt}_A(P).
\]
If \(\operatorname{per}A>0\), let
\[
 \mathbf P_A(P):=\frac{\operatorname{wt}_A(P)}{\operatorname{per}A}.
\]

For a finite interval \(I\subset\mathbb Z\), let
\[
 A_I(x,y):=w(y-x),\qquad
 Z_I:=Z^{\mathrm{per}}_{I,w}=\operatorname{per}A_I,
\]
and write \(I_m:=\{1,\ldots,m\}\), \(A_m:=A_{I_m}\),
\(Z_m:=Z_{I_m}\), with \(Z_0:=1\). We identify \(r\in I_L\) with its
residue class modulo \(L\), so that \(L\) represents the origin of
\(\mathbb T_L\). Define
\begin{equation}
 \widehat A_L(x,y):=w^{(L)}(y-x)
 =\sum_{k\in\mathbb Z}w(y-x+kL),\qquad
 \widehat Z_L:=\operatorname{per}\widehat A_L=Z^{\mathrm{per}}_{L,w}.\label{eq:pa:1}
\end{equation}
Since \(L\) is even, every {positive} off-diagonal entry of \(\widehat A_L\) joins
opposite parity classes. Moreover, \(kL\) is even for every \(k\ne0\), so
lazy bipartiteness gives \(w(kL)=0\) and hence
\(\widehat A_L(x,x)=w(0)\). The correspondence recalled above therefore
applies to \(\widehat A_L\). In particular,
\(Z^{\mathrm{mdd}}_{L,w}=\widehat Z_L\), and the permutation law with matrix
\(\widehat A_L\) is \(\mathbb P^{\mathrm{mdd}}_{L,w}\) under this
identification.

If \(A\) is indexed by \(V\) and \(I,J\subset V\) have the same cardinality,
write \(A^{I,J}\) for the matrix obtained by deleting rows \(I\) and columns
\(J\), and \(A^{x,y}:=A^{\{x\},\{y\}}\). For opposite-parity \(x,y\), the
correspondence identifies configurations contributing to the two-point
function with bijections between \(I_L\setminus\{x\}\) and
\(I_L\setminus\{y\}\). Since \(\widehat A_L\) is symmetric, the direction is
immaterial, and
\begin{equation}
 G^{\mathrm{mdd}}_{L,w}(x,y)
 =\frac{\operatorname{per}\widehat A_L^{x,y}}{\widehat Z_L}.\label{eq:pa:2}
\end{equation}
For \(I,J\subset V\), \(|I|=|J|=k\), let
\[
 \mathcal M_{I,J}:=\{M:V\setminus I\longrightarrow V\setminus J:
 M\text{ is a bijection}\},\qquad
 \operatorname{wt}_A(M):=\prod_{x\in V\setminus I}A(x,M(x)).
\]
We now introduce the positions which split an interval permutation into
smaller pieces. A point \(k\in\{0,\ldots,m\}\) is a \emph{free cut} of a
permutation \(\sigma\in\Omega_{I_m}^{\mathrm{per}}\) if
\begin{equation}
 \sigma(\{1,\ldots,k\})=\{1,\ldots,k\}.\label{eq:pa:10}
\end{equation}
Equivalently, no arrow crosses the cut between \(k\) and \(k+1\).
Consecutive free cuts divide the interval into boxes, and every permutation
cycle is contained in one box.

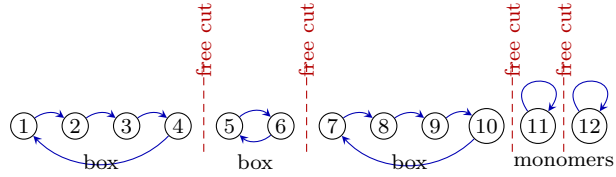
\begin{figure}[!ht]
\centering
\begin{tikzpicture}[x=0.68cm,y=0.62cm,>=stealth,
  vertex/.style={circle,draw,inner sep=1.15pt,font=\scriptsize}]
  \foreach \x in {1,...,12}{\node[vertex] (v\x) at (\x,0) {\x};}
  \draw[->,blue!70!black,bend left=36] (v1) to (v2);
  \draw[->,blue!70!black,bend left=36] (v2) to (v3);
  \draw[->,blue!70!black,bend left=36] (v3) to (v4);
  \draw[->,blue!70!black,bend left=46] (v4) to (v1);
  \draw[->,blue!70!black,bend left=42] (v5) to (v6);
  \draw[->,blue!70!black,bend left=42] (v6) to (v5);
  \draw[->,blue!70!black,bend left=36] (v7) to (v8);
  \draw[->,blue!70!black,bend left=36] (v8) to (v9);
  \draw[->,blue!70!black,bend left=36] (v9) to (v10);
  \draw[->,blue!70!black,bend left=46] (v10) to (v7);
  \foreach \x in {11,12}{
    \draw[->,blue!70!black,out=125,in=55,looseness=6] (v\x) to (v\x);}
  \foreach \c in {4.5,6.5,10.5,11.5}{
    \draw[densely dashed,red!70!black] (\c,-0.55)--(\c,1.45);
    \node[font=\scriptsize,red!70!black,rotate=90] at (\c,1.85) {free cut};}
  \node[font=\scriptsize] at (2.5,-0.75) {box};
  \node[font=\scriptsize] at (5.5,-0.75) {box};
  \node[font=\scriptsize] at (8.5,-0.75) {box};
  \node[font=\scriptsize] at (11.5,-0.75) {monomers};
\end{tikzpicture}
\caption{Free cuts are crossed by no permutation arrow and decompose the
interval into boxes. Every non-trivial displayed cycle alternates parity;
self-arrows represent monomers.}
\label{fig:free-cuts}
\end{figure}

{
The permanent minors describe partial bijections. The following completion
lemma allows us to use free cuts for these objects while preserving all
non-diagonal edges.
}

\begin{lem}[Completion of partial permutations]
\label{lem:pa-completion}
Let \(A\) be a finite symmetric non-negative matrix indexed by \(V\), with
\(A(x,x)=w(0)>0\), and let \(I,J\subset V\) satisfy
\(|I|=|J|=k\). There is an injection
\[
 \Phi_{I,J}:\mathcal M_{I,J}^2\longrightarrow
 (\Omega_V^{\mathrm{per}})^2
\]
such that, if \(\Phi_{I,J}(M,N)=(P,Q)\), then
\begin{equation}
 \operatorname{wt}_A(P)\operatorname{wt}_A(Q)
 =w(0)^{2k}\operatorname{wt}_A(M)\operatorname{wt}_A(N),\label{eq:pa:4}
\end{equation}
and every non-diagonal edge of \(M\cup N\) occurs in the undirected graph
\(P\cup Q\). Consequently,
\begin{equation}
 \operatorname{per}A^{I,J}\le w(0)^{-k}\operatorname{per}A.\label{eq:pa:5}
\end{equation}
More generally, if \(\mathcal E\subset\mathcal M_{I,J}\) and every pair
from \(\mathcal E^2\) is mapped into an event
\(\widehat{\mathcal E}\subset(\Omega_V^{\mathrm{per}})^2\), then, for the permutation
law \(\mathbf P_A\),
\begin{equation}
 \sum_{M\in\mathcal E}\operatorname{wt}_A(M)
 \le w(0)^{-k}\operatorname{per}A
 \sqrt{(\mathbf P_A\otimes\mathbf P_A)(\widehat{\mathcal E})}.\label{eq:pa:6}
\end{equation}
\end{lem}

\begin{proof}
Draw \(M\) as a red matching between left and right copies of \(V\), and draw
the transpose of \(N\) in blue. Symmetry of \(A\) ensures that transposing
\(N\) does not change its weight. For every \(v\in I\mathbin\triangle J\),
add one diagonal edge \(L_vR_v\), and for every \(v\in I\cap J\), add two
copies. The resulting uncoloured bipartite multigraph is
two-regular. On a cycle containing no added edge, retain the original
red-blue colouring. On every other cycle, colour the edges alternately,
choosing the colour of the first added edge by a fixed rule depending only
on the augmented uncoloured cycle and a fixed ordering of \(V\). The two
colour classes are full permutations \(P,Q\).

This construction is injective. Indeed, the added diagonal edges are
determined by \(I,J\): there is one at every vertex of \(I\triangle J\) and
two at every vertex of \(I\cap J\), and no original edge at such a vertex is
diagonal. After these edges are removed, each remaining path has an endpoint
at which the missing original colour is prescribed by membership in \(I\) or
\(J\); alternating the colours along the path therefore recovers its original
colouring. Cycles without added edges were left unchanged. Thus \(M\) and
the transpose of \(N\), and hence \(M,N\), are uniquely recovered from
\(P,Q\). Exactly \(2k\) diagonal edges of weight \(w(0)\) were added, which
proves \eqref{eq:pa:4}. Summation over all pairs, and then over
\(\mathcal E^2\), gives \eqref{eq:pa:5}--\eqref{eq:pa:6}.
\end{proof}

{
Write \(\mathbf P_m:=\mathbf P_{A_m}\). We first compare the partition
function of an interval with the product of those of its two parts.
}

\begin{lem}[Uniform factorisation]
\label{lem:pa-factorisation}
For \(m,n\ge0\),
\begin{equation}
 Z_mZ_n\le Z_{m+n}\le C_0Z_mZ_n,
 \qquad C_0:=\exp\!\left(\frac{M_1}{w(0)}\right).\label{eq:pa:9}
\end{equation}
\end{lem}

\begin{proof}
The lower bound follows by retaining permutations preserving both intervals.
For the upper bound, let \(B\) contain the entries of \(A_{m+n}\) whose
endpoints lie in the same one of the intervals \(I_m\) and
\(\{m+1,\ldots,m+n\}\), let \(C:=A_{m+n}-B\), and set
\(A(t):=B+tC\). Thus
\(\operatorname{per}A(0)=Z_mZ_n\) and
\(\operatorname{per}A(1)=Z_{m+n}\). Since \(A(t)\) is symmetric and has
diagonal entries \(w(0)\),
\[
 \frac{d}{dt}\operatorname{per}A(t)
 =\sum_{x,y}C(x,y)\operatorname{per}A(t)^{x,y}.
\]
Lemma~\ref{lem:pa-completion} with one deleted row and column therefore gives
\[
 \frac{d}{dt}\log\operatorname{per}A(t)
 \le\frac1{w(0)}\sum_{x,y}C(x,y).
\]
For a fixed signed displacement \(r\), there are at most \(|r|\) ordered
pairs \((x,y)\) with \(y-x=r\) which cross the cut. Hence
\(\sum_{x,y}C(x,y)\le M_1\), and integration proves the claim.
\end{proof}

\subsection{\texorpdfstring{Common free cuts}{Common free cuts}}

{
For an ordered pair of permutations of \(I_m\), call a cut \emph{common}
if it is free for both permutations. Let \(c_m\) be the total product
weight of pairs with no common cut in \(\{1,\ldots,m-1\}\).
Decomposing at the first positive common cut gives
\begin{equation}
 Z_m^2=\sum_{r=1}^m c_rZ_{m-r}^2.\label{eq:pa:11}
\end{equation}
Under \(\mathbf P_L^{\otimes2}\), let \(D_L\) be the largest gap between
consecutive common cuts, including the endpoints \(0,L\), and put
\[
 \rho_L:=\mathbf E_{\mathbf P_L^{\otimes2}}[D_L/L].
\]

\begin{lem}[Common free cuts]
\label{lem:pa-renewal}
The limit \(\lambda:=\lim_{m\to\infty}Z_m^{1/m}\) exists in
\((0,\infty)\). With
\begin{equation}
 v_m:=Z_m^2\lambda^{-2m},\qquad q_m:=c_m\lambda^{-2m},\label{eq:pa:12}
\end{equation}
we have
\begin{equation}
 C_0^{-2}\le v_m\le1,\qquad
 \sum_{m\ge1}q_m=1,\qquad
 \sum_{m\ge1}m q_m\le C_0^2.\label{eq:pa:13}
\end{equation}
Moreover,
\begin{equation}
 \rho_L\longrightarrow0.\label{eq:pa:17}
\end{equation}
\end{lem}

\begin{proof}
Lemma~\ref{lem:pa-factorisation} and Fekete's lemma give \(\lambda\),
and \(w(0)^m\le Z_m\le(\sum_rw(r))^m\) gives \(0<\lambda<\infty\).
Supermultiplicativity yields \(Z_m\le\lambda^m\). Iterating the reverse
bound in \eqref{eq:pa:9} gives
\(Z_{km}\le C_0^{k-1}Z_m^k\); taking the \(k\)-th root and letting
\(k\to\infty\) yields \(Z_m\ge C_0^{-1}\lambda^m\).
Thus \(C_0^{-2}\le v_m\le1\).

Set \(V(z):=\sum_{m\ge0}v_mz^m\) and
\(Q(z):=\sum_{m\ge1}q_mz^m\). Equation~\eqref{eq:pa:11} gives
\(V(z)=(1-Q(z))^{-1}\) for \(0<z<1\).
Since \(V(z)\ge C_0^{-2}/(1-z)\), we have \(Q(1)=1\), and
\[
 \sum_{m\ge1}m q_m
 =\lim_{z\uparrow1}\frac{1-Q(z)}{1-z}
 =\lim_{z\uparrow1}\frac1{(1-z)V(z)}\le C_0^2.
\]

Let \(N_m\) count the gaps of length \(m\) between consecutive common
cuts. Specifying such a gap leaves unrestricted pairs of permutations on
the intervals on either side, so
\begin{equation}
 \mathbf E_{\mathbf P_L^{\otimes2}}N_m
 =\frac{q_m}{v_L}
   \sum_{\substack{a,b\ge0\\a+b=L-m}}v_av_b
 \le C_0^2Lq_m.\label{eq:pa:16}
\end{equation}
For every \(\delta>0\), the deterministic bound
\(D_L\le\delta L+\sum_{m>\delta L}mN_m\) therefore gives
\[
 \rho_L\le\delta+C_0^2\sum_{m>\delta L}m q_m.
\]
Letting first \(L\to\infty\) and then \(\delta\downarrow0\) proves
\eqref{eq:pa:17}.
\end{proof}

Every component of the undirected union of the two interval permutations
lies between consecutive common cuts and thus has at most \(D_L\)
vertices. Each component obtained after adding \(k\) edges meets at most
\(k+1\) original components and thus has at most \((k+1)D_L\) vertices.
}

\subsection{Passage from the interval to the torus}

So far we have used the restriction of \(w\) to the interval \(I_L\).
For the periodised weight from Definition~\ref{def:periodicisation}, the
term with \(k=0\) is the interval weight, while the terms with \(k\ne0\)
form the boundary correction. We therefore decompose
the periodised matrix as
\begin{equation}
 \widehat A_L=A_L+T_L,\qquad
 T_L(x,y):=\sum_{k\in\mathbb Z\setminus\{0\}}w(y-x+kL).\label{eq:pa:18}
\end{equation}
Thus \(T_L\) collects exactly the additional arrows created by
periodisation. Its total mass is bounded uniformly in \(L\). Indeed,
counting a signed displacement \(r\) gives
\begin{equation}
 S_L:=\sum_{x,y\in I_L}T_L(x,y)
 =\sum_{r\in\mathbb Z}\min\{L,|r|\}w(r)\le M_1.\label{eq:pa:19}
\end{equation}
Multilinearity of the permanent yields an expansion over sets \(F\) of
arrows with pairwise distinct starting points and pairwise distinct end
points, each arrow carrying a factor \(T_L\). If \(U_F,V_F\) are their
starting-point and end-point sets and
\(\operatorname{wt}_T(F):=\prod_{(x,y)\in F}T_L(x,y)\), then
\begin{equation}
 \widehat Z_L=\sum_F\operatorname{wt}_T(F)
 \operatorname{per}A_L^{U_F,V_F}.\label{eq:pa:20}
\end{equation}
By \eqref{eq:pa:5}, \eqref{eq:pa:19}, and the fact that the full \(k\)-th power of \(S_L\) contains
every size-\(k\) partial matching in all \(k!\) orders,
\begin{equation}
 \frac1{Z_L}\sum_{|F|=k}\operatorname{wt}_T(F)
 \operatorname{per}A_L^{U_F,V_F}
 \le\frac1{k!}\left(\frac{M_1}{w(0)}\right)^k,\label{eq:pa:21}
\end{equation}
and therefore
\begin{equation}
 Z_L\le\widehat Z_L\le
 \exp\!\left(\frac{M_1}{w(0)}\right)Z_L.\label{eq:pa:22}
\end{equation}

Fix one of the sets \(F\) in \eqref{eq:pa:20}, with \(|F|=k\), and let
\(\mathcal E_F(t)\) be the partial
bijections \(M\in\mathcal M_{U_F,V_F}\) for which \(M\cup F\) contains a
non-trivial cycle of length at least \(t\). Every edge of such a cycle which
belongs to \(M\) is non-diagonal and is therefore retained by
Lemma~\ref{lem:pa-completion}. Thus, if
\(M,N\in\mathcal E_F(t)\), the graph \(P\cup Q\cup F\) obtained from that
lemma has a component of size at least \(t\). Since the \(k\) edges of \(F\)
can join at most \(k+1\) components of \(P\cup Q\), this implies
\(D_L\ge t/(k+1)\). Applying \eqref{eq:pa:6} termwise in
\eqref{eq:pa:20}, using the same counting argument as in
\eqref{eq:pa:21}, and using \(S_L\le M_1\) and
\(\widehat Z_L\ge Z_L\), we obtain
\begin{equation}
 \mathbb P^{\mathrm{mdd}}_{L,w}(L_{\max}\ge t)
 \le\sum_{k\ge0}\frac1{k!}\left(\frac{M_1}{w(0)}\right)^k
 \sqrt{\mathbf P_L^{\otimes2}
   \bigl(D_L\ge t/(k+1)\bigr)}.\label{eq:pa:23}
\end{equation}
{
For $0<t\le1$, the same bound follows from $D_L\ge1$.
Set
\[
 a:=M_1/w(0),\qquad
 C_1:=\sum_{k\ge0}\frac{a^k\sqrt{k+1}}{k!}<\infty.
\]
Integrating \eqref{eq:pa:23} with \(t=sL\), \(0<s<1\), and using
Cauchy--Schwarz gives
\begin{equation}
 \frac1L\mathbb E^{\mathrm{mdd}}_{L,w}L_{\max}
 \le C_1\sqrt{\rho_L}.\label{eq:pa:loop-mean}
\end{equation}
Indeed, the integral of the \(k\)-th square root is at most
\(\sqrt{\mathbf E_{\mathbf P_L^{\otimes2}}
 [\min\{1,(k+1)D_L/L\}]}\le\sqrt{(k+1)\rho_L}\).
}

It remains to control the averaged two-point function. If \(i\ne j\) have
the same parity, then no configuration contributing to the monomer
double-dimer two-point function exists, while
\(G^{\mathrm{mdd}}_{L,w}(i,i)=0\) by definition. Thus only
opposite-parity pairs contribute. For such a pair, if a set \(F\) occurs
when the minor in \eqref{eq:pa:2} is expanded, the deleted row and column
sets are
\[
 I_F:=U_F\cup\{i\},\qquad J_F:=V_F\cup\{j\},\qquad
 |I_F|=|J_F|=k+1.
\]
Every \(M\in\mathcal M_{I_F,J_F}\), together with \(F\), contains an
undirected \(i\)-to-\(j\) path: after the arrows of \(F\) are added, it is a
bijection from \(I_L\setminus\{i\}\) to \(I_L\setminus\{j\}\). This open
path cannot contain a diagonal arrow of \(M\), since such an arrow is a
separate one-cycle. Hence all its \(M\)-edges are retained by
Lemma~\ref{lem:pa-completion}.
Therefore \eqref{eq:pa:6} gives
\begin{equation}
 \sum_{M\in\mathcal M_{I_F,J_F}}\operatorname{wt}_{A_L}(M)
 \le w(0)^{-(k+1)}Z_L
 \sqrt{(\mathbf P_L\otimes\mathbf P_L)
       (i\leftrightarrow j\text{ in }P\cup Q\cup F)}.\label{eq:pa:24}
\end{equation}
For each realization, if \(\mathcal K\) denotes the components of
\(P\cup Q\cup F\), then
\[
 \sum_{i,j}\mathbf 1_{\{i\leftrightarrow j\}}
 =\sum_{K\in\mathcal K}|K|^2
 \le L(k+1)D_L.
\]
By translation invariance,
\[
 M^{\mathrm{mdd}}_{L,w}
 =\frac1{L^2}\sum_{i,j\in I_L}G^{\mathrm{mdd}}_{L,w}(i,j).
\]
For {each} fixed \(F\), {enlarge the sum over admissible opposite-parity
endpoints to all ordered pairs \((i,j)\in I_L^2\); all added terms are
non-negative.} Jensen's inequality and the preceding deterministic bound
give
\[
 \frac1{L^2}\sum_{i,j}
 \sqrt{(\mathbf P_L\otimes\mathbf P_L)
 (i\leftrightarrow j\text{ in }P\cup Q\cup F)}
 \le
 \sqrt{(k+1)\rho_L}.
\]
Using this inequality in the expansion over \(F\), and then using the same
counting argument as in \eqref{eq:pa:21} together with \eqref{eq:pa:19},
yields
{
\begin{equation}
 M^{\mathrm{mdd}}_{L,w}
 \le\frac{C_1}{w(0)}\sqrt{\rho_L}.\label{eq:pa:25}
\end{equation}
By \eqref{eq:pa:17}, both \eqref{eq:pa:loop-mean} and
\eqref{eq:pa:25} tend to zero. The remaining conclusions follow from
Markov's inequality and \(|\mathcal L_0|\le L_{\max}\), completing the
proof of Theorem~\ref{thm:positive-activity-one-dimensional}.
}

\section{Long-range order and macroscopic loops in dimension one:
proof of Theorem~\ref{thm:zero-activity-one-dimensional}}
\label{sect:proof-zero-activity-one-dimensional}

The nearest-neighbour case illustrates the main idea of the proof.
Since the monomer activity is zero, every configuration consists of two
perfect matchings. For even \(L\geq4\), there are only two
nearest-neighbour perfect matchings on the cycle \(\T_L\). If the two
matchings are different, their union is a single loop containing every
vertex.

{
For general weights, two independent matchings have opposite winding signs
with probability $1/2$. On this event some loop winds around the torus;
the first-moment assumption prevents its displacement from being carried
by too few long edges.

Fix an even $L$, represent $\T_L$ by $\{0,\ldots,L-1\}$, and let $B,R$
be the two independent perfect matchings.
For each dimer of $D\in\{B,R\}$, with even endpoint $x$ and odd endpoint $y$,
choose an integer $r_D(x)\equiv y-x\pmod L$, independently conditionally on
$(B,R)$, with distribution
\[
\Pr\bigl(r_D(x)=r\mid B,R\bigr)
=\frac{w(r){\mathbf1_{\{r\equiv y-x\pmod L\}}}}
       {w^{(L)}_{x,y}}{,\qquad r\in\mathbb Z}.
\]
Write $\mathbb P,\mathbb E$ for this enlarged finite-volume law.
Orient blue dimers from even to odd and red dimers from odd to even,
and let $r_x$ be the signed displacement of the arrow starting at $x$.
}

\begin{proof}
{Fix an odd integer \(a\) such that \(w(a)>0\), which exists by
normalisation and bipartiteness. For every even \(L\), the edges joining
each even vertex \(x\) to \(x+a\) form a perfect matching of positive
weight, so the finite-volume partition functions are positive.}
For \(D\in\{B,R\}\), set
\[
 \Phi_L(D):=\frac1L\sum_{x\in\T_L^{\mathrm e}}r_D(x).
\]
Since \(D\) matches every even vertex to a different odd vertex,
\[
 \sum_{x\in\T_L^{\mathrm e}}r_D(x)
 \equiv
 \sum_{y\in\T_L^{\mathrm o}}y-
 \sum_{x\in\T_L^{\mathrm e}}x
 =\frac L2\pmod L,
\]
and hence \(\Phi_L(D)\in\mathbb Z+\frac12\). Since \(w(r)=w(-r)\), reflection
at the origin preserves the law and replaces every chosen integer by its
negative. Thus \(\Phi_L(D)\) and \(-\Phi_L(D)\) have the same law.
{
Their values are nonzero, and $B,R$ are independent. Hence the event
\begin{equation}\label{eq:matching-sums-differ}
\mathcal A_L:=\{\Phi_L(B)\Phi_L(R)<0\}
\qquad\text{satisfies}\qquad
\mathbb P(\mathcal A_L)=\frac12.
\end{equation}
}

By construction,
\[
 \frac1L\sum_{x\in\T_L}r_x=\Phi_L(B)-\Phi_L(R).
\]
If \(C\) is a permutation cycle, following all its arrows returns to the
starting vertex modulo \(L\), so
\(\sum_{x\in C}r_x\in L\mathbb Z\). The sum over all cycles is
\(L(\Phi_L(B)-\Phi_L(R))\). On $\mathcal A_L$, it is non-zero, so some cycle \(C\) satisfies
\(\sum_{x\in C}r_x\ne0\), and hence
\begin{equation}
 \sum_{x\in C}|r_x|
 \ge\left|\sum_{x\in C}r_x\right|
 \ge L.
 \label{eq:cycle-distance}
\end{equation}

We next show that this distance cannot usually be produced by a few very long
arrows. At zero activity the monomer double-dimer measure is the
double-dimer measure. Hence Proposition~\ref{prop:monomerloop-lowerbound},
applied to the periodised weights, gives for opposite-parity \(x,y\)
\[
 \bigl({w^{(L)}(a)}\bigr)^2\mathcal C_{L,w}(x,y)^2
 \le \mathbb P^{\mathrm{mdd}}_{L,w}(x\leftrightarrow y)\le1.
\]
Since \({w^{(L)}(a)\ge w(a)}\),
\begin{equation}
 \mathcal C_{L,w}(x,y)\le\frac1{{w(a)}}.
 \label{eq:uniform-dimer-source}
\end{equation}
Deleting the dimer for which the chosen integer is \(r\) gives, for
\(x\in\T_L^{\mathrm e}\) and with \(x+r\) understood modulo \(L\),
\begin{equation}
 \mathbb P\bigl(r_D(x)=r\bigr)
 =w(r)\mathcal C_{L,w}(x,x+r)
 \le\frac{w(r)}{{w(a)}}.
 \label{eq:dimer-step-bound}
\end{equation}

{
Choose $K\ge1$ such that
\[
\sum_{|r|>K}|r|w(r)\le\frac{w(a)}8,
\qquad \varepsilon:=\frac1{2K},
\]
and set $T_K:=\sum_{x\in\T_L}|r_x|\mathbf1_{\{|r_x|>K\}}$.
The two matchings contain $L$ dimers in total, so
\eqref{eq:dimer-step-bound} and Markov's inequality give
\begin{equation}\label{eq:long-arrow-tail}
\mathbb ET_K\le\frac L8,
\qquad \mathbb P(T_K>L/2)\le\frac14.
\end{equation}
}
{
On $\mathcal A_L\cap\{T_K\le L/2\}$, the cycle in
\eqref{eq:cycle-distance} satisfies
\[
K|C|\ge\sum_{x\in C}|r_x|-T_K\ge L/2.
\]
Consequently, \eqref{eq:matching-sums-differ} and
\eqref{eq:long-arrow-tail} imply
}
\begin{equation}
 \mathbb P^{\mathrm{mdd}}_{L,w}
 (L_{\max}\ge\varepsilon L)\ge\frac14.
 \label{eq:macroscopic-loop-zero}
\end{equation}

By translation invariance,
\begin{align}
 \mathbb P^{\mathrm{mdd}}_{L,w}(|\mathcal L_0|\ge\varepsilon L)
 &=\mathbb E^{\mathrm{mdd}}_{L,w}
 \left[\frac1L\sum_{x\in\T_L}
 \mathbf1_{\{|\mathcal L_x|\ge\varepsilon L\}}\right] \notag\\
 &\ge\varepsilon\,
 \mathbb P^{\mathrm{mdd}}_{L,w}(L_{\max}\ge\varepsilon L)
 \ge\frac\varepsilon4,
 \label{eq:tagged-loop-zero}\\
 \frac1L\mathbb E^{\mathrm{mdd}}_{L,w}|\mathcal L_0|
 &\ge\varepsilon\,
 \mathbb P^{\mathrm{mdd}}_{L,w}
 (|\mathcal L_0|\ge\varepsilon L)
 \ge\frac{\varepsilon^2}{4}.
 \label{eq:tagged-loop-mean-zero}
\end{align}
{
It remains to prove long-range order. The reverse estimate in
Proposition~\ref{prop:monomerloop-lowerbound}, applied to the periodised
weights, gives, for opposite-parity vertices,
\begin{equation}\label{eq:source-from-connection}
\mathcal C_{L,w}(x,y)
\ge w^{(L)}(a)\mathbb P^{\mathrm{mdd}}_{L,w}(x\leftrightarrow y)
\ge w(a)\mathbb P^{\mathrm{mdd}}_{L,w}(x\leftrightarrow y).
\end{equation}
}

By \eqref{eq:twopointandmonomer},
\(G^{\mathrm{mdd}}_{L,w}=\mathcal C_{L,w}\) for opposite-parity endpoints,
whereas \(G^{\mathrm{mdd}}_{L,w}=0\) for equal-parity endpoints. Moreover,
every bipartite loop contains equally many even and odd vertices.
Consequently, \eqref{eq:source-from-connection} and
\eqref{eq:tagged-loop-mean-zero} give the third assertion:
\begin{align}
 M^{\mathrm{mdd}}_{L,w}
 &=\frac1L\sum_{y\in\T_L^{\mathrm o}}\mathcal C_{L,w}(0,y)
 \ge\frac{{w(a)}}L\sum_{y\in\T_L^{\mathrm o}}
 \mathbb P^{\mathrm{mdd}}_{L,w}(0\leftrightarrow y) \notag\\
 &=\frac{{w(a)}}{2L}
 \mathbb E^{\mathrm{mdd}}_{L,w}|\mathcal L_0|
 \ge\frac{{w(a)}\varepsilon^2}{8}.
 \label{eq:averaged-two-point-zero}
\end{align}
Thus {all} finite-volume claims hold with
\[
 c:=\min\left\{\frac14,\frac{{w(a)}\varepsilon^2}{8}\right\}.
\]

\end{proof}

{
\subsection{Bi-infinite cycles in subsequential limits}

The finite-volume argument also yields a statement on $\mathbb Z$.
Using the displacements sampled above and representatives
$x\in\{0,\ldots,L-1\}$, define
\[
\pi_L(x+kL):=x+kL+r_x,\qquad k\in\mathbb Z.
\]
This is an $L$-periodic bijection of $\mathbb Z$; denote its law by $\mu_L$.
For a bijection $\sigma$, write
$\mathcal L_x(\sigma):=\{\sigma^n(x):n\in\mathbb Z\}$.
We equip the space of bijections with the local topology: both $\sigma$
and $\sigma^{-1}$ converge pointwise, with $\mathbb Z$ given the discrete
topology.

\begin{cor}[Bi-infinite cycles]\label{cor:bi-infinite-one-dimensional}
Under the assumptions of Theorem~\ref{thm:zero-activity-one-dimensional},
the family $(\mu_L)_{L\in2\mathbb N}$ is tight.
There exists $c=c(w)>0$ such that every subsequential limit $\mu$ as
$L\to\infty$ through even integers satisfies
\[
\mu\bigl(|\mathcal L_0|=\infty\bigr)\ge c.
\]
\end{cor}

\begin{proof}
Fix an odd $a$ with $w(a)>0$, as above.
By \eqref{eq:dimer-step-bound} and symmetry, uniformly in $L$ and
$x\in\mathbb Z$,
\[
\mu_L\bigl(|\pi_L(x)-x|>K\bigr)
\le\frac1{w(a)}\sum_{|r|>K}w(r)\longrightarrow0.
\]
Interchanging the two matchings, together with their chosen displacements,
sends $\pi_L$ to $\pi_L^{-1}$, so the same estimate holds for the inverse.
Given $\eta>0$, choose integers $K_x$ with
\[
\frac2{w(a)}\sum_{x\in\mathbb Z}\sum_{|r|>K_x}w(r)<\eta.
\]
The bijections satisfying $|\sigma(x)-x|\le K_x$ and
$|\sigma^{-1}(x)-x|\le K_x$ for every $x$ form a compact set in the
local topology; the union bound gives this set $\mu_L$-probability at
least $1-\eta$. This proves tightness.

Let $\mu_{L_j}\Rightarrow\mu$, where $L_j\to\infty$, and choose
$\varepsilon>0$ as in \eqref{eq:tagged-loop-zero}.
The lift of a torus cycle has the same length if its total displacement
is zero, and is infinite otherwise. Therefore, for every fixed $m$ and
all sufficiently large $j$, \eqref{eq:tagged-loop-zero} gives
\[
\mu_{L_j}\bigl(|\mathcal L_0|\ge m\bigr)\ge\frac\varepsilon4.
\]
The event $\{|\mathcal L_0|\ge m\}$ is both open and closed in the local
topology{. Indeed, it is equivalent to the pairwise distinctness of
$0,\sigma(0),\ldots,\sigma^{m-1}(0)$, and the finite-iterate maps are
continuous in this topology.}
Passing to the limit and then letting $m\to\infty$ yields
$\mu(|\mathcal L_0|=\infty)\ge\varepsilon/4$.
Every infinite orbit of a bijection is bi-infinite.
\end{proof}
}

\section{Appendix}
\label{sect:appendix}

\subsection{OS-positivity: examples and technical lemmas}
\label{sect:proofOSpositivity}

We first record a moment criterion for $\ell^1$-radial weights. Notice that
the diagonal value does not enter the OS quadratic form: if
$x,y\in\Z^d_{i,+}$, then $x\neq\theta_i y$. Thus an arbitrary non-negative
constant diagonal weight may be added.

\begin{lem}
\label{lem:moment-OS}
Let $j:\N_{>0}\to[0,\infty)$ and suppose that there exists a finite positive
measure $\mu$ on $[-1,1]$ such that
\begin{equation}\label{eq:moment_rep_minus1_1}
j(n+1)=\int_{-1}^{1}t^n\,\mu(\mathrm dt),\qquad n\ge0.
\end{equation}
For any $d\ge1$, define
\[
w_{x,y}=j(|x-y|_1),\qquad x\neq y,
\]
and set $w_{x,x}=w(0)$ for an arbitrary constant $w(0)\ge0$. Then $w$ is
OS-positive in the sense of Definition~\ref{def:OSpos-quadratic}.
\end{lem}

\begin{proof}
Fix $i$ and, after relabelling coordinates, take $i=1$. Write
$x=(m,u)$ and $y=(n,v)$, where $m,n\ge1$ and
$u,v\in\Z^{d-1}$. Since
\[
|x-\theta_1y|_1=m+n-1+|u-v|_1,
\]
\eqref{eq:moment_rep_minus1_1} gives
\[
w_{x,\theta_1y}
=\int_{-1}^{1}t^{m-1}t^{n-1}t^{|u-v|_1}\,\mu(\mathrm dt).
\]
For every $t\in[-1,1]$, the kernel
$K_t(u,v):=t^{|u-v|_1}$ is positive semidefinite on $\Z^{d-1}$. Indeed,
for $|t|<1$ its Fourier transform is
\[
\prod_{r=1}^{d-1}\frac{1-t^2}{1-2t\cos k_r+t^2}\ge0,
\]
whereas for $t=\pm1$ the assertion follows directly, since the corresponding
kernels have rank one. Hence, with
\[
F_t(u):=\sum_{m\ge1}f(m,u)t^{m-1},
\]
we obtain
\begin{align*}
&\sum_{x,y\in\Z^d_{1,+}}
\overline{f(x)}\,w_{x,\theta_1y}\,f(y)\\
&\qquad=\int_{-1}^{1}\sum_{u,v\in\Z^{d-1}}
\overline{F_t(u)}K_t(u,v)F_t(v)\,\mu(\mathrm dt)\ge0.
\end{align*}
The same argument applies to every coordinate direction.
\end{proof}

The passage from a one-dimensional reflection-positive profile to an
$\ell^1$-radial interaction in arbitrary dimension is also given in
\cite[Proposition~3.1(i)]{AizenmanFernandez}.

We now verify OS-positivity for the cases \emph{(A)}--\emph{(G)} in the
list of admissible weights.

\begin{enumerate}
\item[(A)] \textit{Nearest neighbours:}
Fix $i$ and, after relabelling coordinates, take $i=1$. For
$x,y\in\Z^d_{1,+}$,
\[
w_{x,\theta_1y}
=\frac1{2d}\,
\mathbf1_{\{x_1=y_1=1\}}\mathbf1_{\{x_\perp=y_\perp\}}.
\]
Consequently,
\[
\sum_{x,y\in\Z^d_{1,+}}\overline{f(x)}w_{x,\theta_1y}f(y)
=\frac1{2d}\sum_{u\in\Z^{d-1}}|f(1,u)|^2\ge0.
\]
For reflection positivity of the corresponding nearest-neighbour random-path
model, see also \cite[Theorem~4.3]{T}.

\item[(B)] \textit{Positive diagonal activity, first and second nearest
neighbours:}
Again only $x_1=y_1=1$ contributes. Writing $g(u):=f(1,u)$, the OS
quadratic form equals
\[
\beta\sum_u|g(u)|^2
+\gamma\sum_{\substack{u,v\in\Z^{d-1}\\|u-v|_2=1}}
\overline{g(u)}g(v).
\]
By Fourier transformation this is
\[
\int_{[-\pi,\pi]^{d-1}}
\left(\beta+2\gamma\sum_{r=1}^{d-1}\cos k_r\right)
|\widehat g(k)|^2\,\frac{\mathrm dk}{(2\pi)^{d-1}}\ge0,
\]
because
\[
\beta+2\gamma\sum_{r=1}^{d-1}\cos k_r
\ge\beta-2(d-1)\gamma\ge0
\]
by the parameter restriction in case~\emph{(B)}. For $d=1$ the
second-neighbour shell is empty and, with case~\emph{(B)} understood as
$\rho+2\beta=1$, the assertion is immediate. The diagonal term
$\rho\delta_0$ does not contribute. The same parameter range for $d\ge2$ is
treated in \cite[Remark~4.5(1)]{BiskupChayesCrawford}; cf.
\cite[Proposition~3.4]{FrohlichLiebSimon} for the general
reflection-positive construction.

\item[(C)] \textit{Exponential decay:}
For $n\ge0$,
\[
c_d(\rho,\gamma)e^{-\gamma(n+1)}
=\int_0^1t^n\,c_d(\rho,\gamma)e^{-\gamma}
\delta_{e^{-\gamma}}(\mathrm dt).
\]
Lemma~\ref{lem:moment-OS} and the preceding observation about the diagonal
prove the claim.
This exponential $\ell^1$-interaction is also treated in
\cite[Remark~4.5(2)]{BiskupChayesCrawford}; see
\cite{FrohlichLiebSimon} for the classical reflection-positivity framework.

\item[(D)] \textit{Polynomial decay:}
For $q=1$, put $c=(1-\rho)/Z_{s,1}$. Since
\[
\frac{c}{(n+1)^s}
=\int_0^1t^n\,\frac{c}{\Gamma(s)}
(-\log t)^{s-1}\,\mathrm dt,
\]
Lemma~\ref{lem:moment-OS} applies.
For the same $\ell^1$-power-law class, see also
\cite[Remark~4.5(3)]{BiskupChayesCrawford}.

For $q=2$, the case $d=1$ is identical to $q=1$. Suppose $d\ge2$, put
\[
D:=d-1,\qquad \alpha:=\frac{s}{2},\qquad
\nu:=\alpha-\frac D2=\frac{s-d+1}{2}>\frac12,
\]
and define
\[
g_a(z):=(a^2+|z|_2^2)^{-\alpha},
\qquad a>0,\quad z\in\R^D.
\]
For $p\neq0$, the Fourier--Bessel formula and the integral representation
of the modified Bessel function give
\begin{align*}
\widehat g_a(p)
&=\frac{2\pi^{D/2}}{\Gamma(\alpha)}
\left(\frac{|p|_2}{2a}\right)^\nu K_\nu(a|p|_2)\\
&=c_{D,s}|p|_2^{2\nu}
\int_1^\infty e^{-a|p|_2r}
(r^2-1)^{\nu-\frac12}\,\mathrm dr,
\end{align*}
where $c_{D,s}>0$; the value at $p=0$ is obtained by continuity.
Fourier inversion, with $a=m+n-1$, therefore gives
\begin{align*}
&\sum_{m,n\ge1}\sum_{u,v\in\Z^D}
\overline{f(m,u)}g_{m+n-1}(u-v)f(n,v)\\
&\quad=c'_{D,s}\int_{\R^D}\int_1^\infty
|p|_2^{2\nu}(r^2-1)^{\nu-\frac12}
\left|
\sum_{m\ge1}\sum_{u\in\Z^D}
f(m,u)e^{-(m-\frac12)|p|_2r}e^{-\mathrm ip\cdot u}
\right|^2
\,\mathrm dr\,\mathrm dp\ge0,
\end{align*}
with $c'_{D,s}>0$. Thus $|x|_2^{-s}$ is OS-positive. Multiplication by
$(1-\rho)/Z_{s,2}$ and addition of $\rho\delta_0$ preserve this property.
For earlier constructions of Euclidean power-law reflection-positive
interactions, see \cite[Section~3, in particular (3.1)]{AizenmanFernandez}
and \cite{FrohlichLiebSimon}.

\item[(E)] \textit{Regularised polynomial decay:}
Put $c=(1-\rho)/Z_{s,\gamma}$. For $n\ge0$,
\[
\frac{c}{(1+\gamma(n+1))^s}
=\int_0^1t^n\,
\frac{c\gamma^{-s}}{\Gamma(s)}
t^{1/\gamma}(-\log t)^{s-1}\,\mathrm dt.
\]
The measure on the right is finite and positive, so
Lemma~\ref{lem:moment-OS} applies.
See also \cite[Section~3, Eq.~(3.2)]{AizenmanFernandez}.

\item[(F)] \textit{Bipartite projections:}
This follows from Lemma~\ref{lem:bipartite-projection} below, provided that
$b(w)>0$.

\item[(G)] \textit{Convex combinations:}
The OS quadratic form is linear in $w$. Hence every non-negative linear
combination of OS-positive weights is OS-positive; convex combinations also
preserve normalization.
This closure property is also stated in
\cite[Corollary~3.6(2)]{FrohlichLiebSimon}.
\end{enumerate}

{We now} verify \eqref{eq:torus-green-convergence} for {the weights in}
Examples~\emph{(A)}--{\emph{(E)}}, {their bipartite projections in~\emph{(F)},
and finite convex combinations of these weights in~\emph{(G)}.
For each of the individual kernels,} whenever $g_{d,w}<\infty$, the
Fourier estimates {below} give
\[
1-\widehat w(k)\ge c|k|_2^\alpha
\]
near zero{,} for some $\alpha<d${. Indeed, for $d\ge3$ the positive
nearest-neighbour weights give this bound with $\alpha=2$. For
$d\in\{1,2\}$, the transient members of the displayed families are
the power-law kernels with $d<s<2d$, and their bipartite projections;
the estimates below apply with $\alpha=s-d<d$.
Consequently, for all sufficiently small $\delta>0$,}
\[
\sup_{L\in2\mathbb N}\frac1{L^d}
\sum_{\substack{k\in\T_L^*\\0<|k|_2\le\delta}}
\frac1{1-\widehat w(k)}\le C\delta^{d-\alpha}{.}
\]
{Irreducibility makes the denominator strictly positive away from zero,
so ordinary} Riemann-sum convergence {there proves
\eqref{eq:torus-green-convergence}}.

For {a} finite convex {combination}, a transient component with positive
coefficient supplies the same {lower} bound{. If all components are
recurrent, then $d\le2$. For these particular families, the elementary
estimate $1-\cos(k\cdot x)\le\min\{2,|k|_2^2|x|_2^2/2\}$ gives,
for every recurrent component and $0<|k|_2\le1$,
\[
1-\widehat w(k)\le
\begin{cases}
C|k|_2,&d=1,\\
C|k|_2^2\log(e/|k|_2),&d=2.
\end{cases}
\]
These upper bounds are preserved by finite convex combinations and
force the Green integral to diverge. Such combinations are therefore}
recurrent and impose no further requirement in
\eqref{eq:torus-green-convergence}. {The assertions here concern the
displayed families; the OS-positivity closure properties in
\emph{(F)} and~\emph{(G)} alone are not used to infer
\eqref{eq:torus-green-convergence} for arbitrary kernels.}

\begin{lem}[Bipartite projection preserves OS-positivity]
\label{lem:bipartite-projection}
Let $w$ be OS-positive and define
\[
w^b_{x,y}:=w_{x,y}\mathbf1_{\{|x-y|_1\ {\rm odd}\}}.
\]
Then $w^b$ is OS-positive. Consequently, if $w$ is translation invariant and
summable and $b(w)>0$, then the normalized projection
\[
w^{\mathrm o}:=b(w)^{-1}w^b
\]
is OS-positive.
\end{lem}

\begin{proof}
Fix $i$ and set
\[
\chi(x):=(-1)^{|x|_1},\qquad
f_\sigma(x):=f(x)\mathbf1_{\{\chi(x)=\sigma\}},
\quad \sigma\in\{\pm1\}.
\]
Since $\theta_i$ reverses parity,
\[
\chi(\theta_i y)=-\chi(y),
\]
and hence
\[
|x-\theta_i y|_1\ {\rm odd}
\quad\Longleftrightarrow\quad
\chi(x)=\chi(y).
\]
Therefore, for every finitely supported $f:\Z^d_{i,+}\to\C$,
\begin{align*}
&\sum_{x,y\in\Z^d_{i,+}}
\overline{f(x)}\,w^b_{x,\theta_i y}\,f(y)\\
&\qquad=
\sum_{\sigma\in\{\pm1\}}
\sum_{x,y\in\Z^d_{i,+}}
\overline{f_\sigma(x)}\,w_{x,\theta_i y}\,f_\sigma(y)\ge0.
\end{align*}
Thus $w^b$ is OS-positive. Multiplication by the positive constant
$b(w)^{-1}$ proves the normalized assertion.
\end{proof}

\begin{proof}[{Proof of Lemma} \ref{lem:OS-periodisation}]
This is the standard stability of reflection positivity under periodic
boundary conditions; see \cite[Proposition~3.4]{FrohlichLiebSimon} for the
one-dimensional construction and \cite[Lemma~4.4]{BiskupChayesCrawford} for the
multidimensional torus formulation. We include a short proof since the
latter reference omits the details.

By translation invariance and relabelling the coordinates, it suffices
to consider \(i=1\) and the torus reflection represented by
\(\theta_0(x_1,x')=(1-x_1,x')\). Let
\[
R_{L,+}:=\{x\in\mathbb Z^d:1\le x_1\le L/2,\ 0\le x_j<L,\
j=2,\ldots,d\},
\]
which we identify with \(\mathbb T_{L,1,+}\). For \(a\in\mathbb Z\), set
\[
\theta_a(x_1,x'):=(2a+1-x_1,x'),\qquad
H_a^+:=\{x:x_1>a\},\quad H_a^-:=\{x:x_1\le a\}.
\]
Since \(\theta_a\) is an integer translate of \(\theta_0\), translation
invariance implies OS-positivity with respect to \(\theta_a\) on
\(H_a^+\). It also holds on \(H_a^-\). Indeed, if \(F\) is supported in
\(H_a^-\), then \(G(u):=\overline{F(\theta_a u)}\) is supported in
\(H_a^+\), and symmetry of \(w\) gives
\[
\sum_{u,v\in H_a^+}\overline{G(u)}w_{u,\theta_a v}G(v)
=
\sum_{x,y\in H_a^-}\overline{F(x)}w_{x,\theta_a y}F(y)\ge0.
\]

We first record a transverse periodisation argument. Suppose that
\(R_{L,+}\) lies in one of the two half-lattices of \(\theta_a\), and set
\[
K_a(x,y):=\sum_{z'\in\mathbb Z^{d-1}}
w_{x,\theta_a y+L(0,z')},\qquad x,y\in R_{L,+}.
\]
We claim that \(K_a\) is positive semidefinite. Let
\(B_N:=\{-N,\ldots,N\}^{d-1}\) and define
\[
F_N\bigl(x+L(0,z')\bigr):=|B_N|^{-1/2}f(x),
\qquad x\in R_{L,+},\ z'\in B_N,
\]
and \(F_N=0\) elsewhere. Its support lies in the same half-lattice as
\(R_{L,+}\), so OS-positivity yields
\[
0\le
\sum_{x,y\in R_{L,+}}\overline{f(x)}f(y)
\sum_{z'\in\mathbb Z^{d-1}}\lambda_N(z')
w_{x,\theta_a y+L(0,z')},
\]
where
\[
\lambda_N(z')
:=\frac{|B_N\cap(B_N-z')|}{|B_N|}
\in[0,1],\qquad \lambda_N(z')\longrightarrow1.
\]
Since \(w\) is summable, dominated convergence gives
\[
\sum_{x,y\in R_{L,+}}\overline{f(x)}K_a(x,y)f(y)\ge0.
\]

For \(r\in\mathbb Z\), let \(a_r:=Lr/2\in\mathbb Z\). Since
\(\theta_{a_r}y=\theta_0y+Lre_1\), periodisation gives
\[
w^{(L)}_{x,\Theta y}
=\sum_{r\in\mathbb Z}\sum_{z'\in\mathbb Z^{d-1}}
w_{x,\theta_{a_r}y+L(0,z')}
=\sum_{r\in\mathbb Z}K_{a_r}(x,y).
\]
If \(r\le0\), then \(R_{L,+}\subset H_{a_r}^+\), whereas for \(r\ge1\),
\(R_{L,+}\subset H_{a_r}^-\). Hence every \(K_{a_r}\) is positive
semidefinite. Absolute convergence therefore implies
\[
\sum_{x,y\in R_{L,+}}
\overline{f(x)}w^{(L)}_{x,\Theta y}f(y)
=
\sum_{r\in\mathbb Z}
\sum_{x,y\in R_{L,+}}
\overline{f(x)}K_{a_r}(x,y)f(y)\ge0,
\]
which is \eqref{eq:OS-periodised}.
\end{proof}

\subsection{Proofs of the long-range-order corollaries}
\label{sect:proof-LRO-corollaries}

We first treat the finite-range family. We then prove the four
Fourier estimates used for Examples~\emph{(C)}--\emph{(E)}.

\begin{proof}[Proof of Corollary~\ref{cor:firstsecondneighbours}]
{
Let \(p_d\) be the probability kernel obtained from the weights in
Example~\emph{(B)} by setting \(\rho=0\) in
\eqref{eq:choiceB}. Its Fourier transform is
$
\widehat p_d(k)
=
\frac{2}{3d}\sum_{i=1}^d\cos k_i
+
\frac{2}{3d(d-1)}
\sum_{1\leq i<j\leq d}\cos k_i\cos k_j.
$
Moreover,
$
\widehat p_d(k)+\frac13
=
\frac{2}{3d(d-1)}
\sum_{1\leq i<j\leq d}
(1+\cos k_i)(1+\cos k_j).
$
Consequently,
$
\min_{k\in[-\pi,\pi]^d}\widehat p_d(k)=-\frac13.
$
For \(d\geq4\), \(\widehat p_d\) is the average of the \(d\)
copies of \(\widehat p_{d-1}\) obtained by deleting one coordinate.
Applying Jensen's inequality to the function
$
t\longmapsto\frac{1}{1+\varepsilon-t}
$
and then letting \(\varepsilon\downarrow0\) gives
$
g_{d,p_d}\leq g_{d-1,p_{d-1}}.
$
For \(d=3\),
$
\widehat p_3(k)
=
\frac23\left(\frac13\sum_{i=1}^3\cos k_i\right)
+
\frac13\left(
\frac13\sum_{1\leq i<j\leq3}\cos k_i\cos k_j
\right).
$
Applying the same regularized Jensen argument and using Watson's
evaluations, we obtain
$$
\begin{aligned}
g_{3,p_3}
&\leq
\frac{2}{3\pi^3}
\int_{[0,\pi]^3}
\frac{\mathrm dk}
{1-\frac13\sum_{i=1}^3\cos k_i}
\\
&\quad+
\frac{1}{3\pi^3}
\int_{[0,\pi]^3}
\frac{\mathrm dk}
{1-\frac13\sum_{1\leq i<j\leq3}
\cos k_i\cos k_j}
\\
&<
\frac23\frac{1517}{1000}
+
\frac13\frac{269}{200}
<
\frac32;
\end{aligned}
$$
see \cite{Watson1939}. It follows that
$$
g_{d,p_d}<\frac32
\qquad\text{for every }d\geq3.
$$

For \(0\leq\rho<1\), the weights in the corollary can be written as
$
w=\rho\delta_0+(1-\rho)p_d.
$
They are normalized, symmetric, and translation invariant, and they
are irreducible because \(\beta>0\). Furthermore,
$
\gamma=\frac{\beta}{2(d-1)},
$
so their OS-positivity follows from
Section~\ref{sect:proofOSpositivity}. Thus, these weights are
admissible for every \(\rho\in[0,1)\).
For these weights,
$
b(w)=\frac{2(1-\rho)}{3}>0,
$
and
$
g_{d,w}=\frac{g_{d,p_d}}{1-\rho},
a_w=\max\left\{0,\frac{1-4\rho}{3}\right\},
q_w=\min\left\{1,\frac{3e\rho}{2(1-\rho)}\right\}.
$
At \(\rho=0\), the right-hand sides of
\eqref{eq:longrangequantity-per} and
\eqref{eq:longrangequantity-mdd} are both equal to
$
\frac32-g_{d,p_d}>0.
$
Both expressions are continuous functions of \(\rho\) in a
neighbourhood of zero. Therefore, there exists
\(\rho_0=\rho_0(d)\in(0,1)\) such that they are both strictly positive
for every \(\rho\in[0,\rho_0)\). The conclusion follows from
Theorem~\ref{thm:longrangeorder2}.}
\end{proof}

\begin{lem}[Fourier estimates for Examples~\emph{(C)}--\emph{(E)}]
\label{lem:examples-C-E-fourier}
For the weights \(w\) from Examples~\emph{(C)}--\emph{(E)} at
\(\rho=0\) and their bipartite projections \(w^{\mathrm{o}}\) from
Example~\emph{(F)}, the following limits hold:
\begin{equation}\label{eq:examples-C-E-limits}
 g_{d,w}\to1,\qquad \min_k\widehat w(k)\to0,\qquad
 \widehat w(\boldsymbol\pi)\to0,\qquad g_{d,w^{\mathrm{o}}}\to1,
\end{equation}
Here the minimum is over \([-\pi,\pi]^d\). These limits hold as
\(\gamma\downarrow0\) in Example~\emph{(C)}, where \(d\ge3\); as
$s\downarrow d$ in Example~\emph{(D)}, for $q\in\{1,2\}$ and every $d\ge1$;
and as \(\gamma\downarrow0\) in
Example~\emph{(E)}, where \(s\) is fixed and
\(d<s<2d\) for \(d\in\{1,2\}\), while \(s>d\) for \(d\ge3\).
\end{lem}

\begin{proof}
Coordinates in \(k+\boldsymbol\pi\) are taken modulo \(2\pi\). Put
\(r:=\min\{1,\max_{1\le j\le d}|k_j|\}\).
Positive constants \(c,C\) may change from line to line. They may
depend on \(d\), and
in Example~\emph{(E)} on \(s\), but not on the parameter tending to its
limit. For \(h\in\{w,w^{\mathrm{o}}\}\) and \(0<\lambda<1\), Fourier
inversion gives
\[
 \sum_{n\ge0}\lambda^nh^{*n}(0)=(2\pi)^{-d}
 \int_{[-\pi,\pi]^d}\frac{\mathrm dk}{1-\lambda\widehat h(k)},
 \qquad (2\pi)^{-d}\int_{[-\pi,\pi]^d}\widehat h(k)\,\mathrm dk=h(0)=0.
\]
Here \(h^{*n}\) is the \(n\)-fold convolution of \(h\), with
\(h^{*0}(0)=1\).
Since \((1-\lambda\widehat h)^{-1}\le2(1-\widehat h)^{-1}\) for
\(1/2\le\lambda<1\), every integrable bound used below permits
\(\lambda\uparrow1\) and yields
\begin{equation}\label{eq:green-Fourier-appendix}
 g_{d,h}=(2\pi)^{-d}\int_{[-\pi,\pi]^d}
 \frac{\mathrm dk}{1-\widehat h(k)}.
\end{equation}

For \(q\in\{1,2\}\) and \(t>0\), put
\[
 R_{q,t}(k)=\sum_{x\in\mathbb Z^d}e^{-t|x|_q^q}e^{\mathrm i k\cdot x},
 \qquad \Phi_{q,t}(k)=\frac{R_{q,t}(k)}{R_{q,t}(0)}.
\]
Direct summation and Poisson summation, respectively, give
\begin{align}
 R_{1,t}(k)&=\prod_{j=1}^d\frac{\sinh t}{\cosh t-\cos k_j},
 \label{eq:R1-product}\\
 R_{2,t}(k)&=\left(\frac\pi t\right)^{d/2}
 \sum_{m\in\mathbb Z^d}e^{-|k+2\pi m|_2^2/(4t)}.
 \label{eq:R2-poisson}
\end{align}
Define the one-dimensional factor by
\[
 \phi_{q,t}(u):=
 \frac{\sum_{n\in\mathbb Z}e^{-t|n|^q}e^{\mathrm i u n}}
 {\sum_{n\in\mathbb Z}e^{-t|n|^q}}.
\]
Then \(\Phi_{q,t}(k)=\prod_{j=1}^d\phi_{q,t}(k_j)\). Each factor is
even and
belongs to \([0,1]\). For \(q=1\),
\[
 \phi_{1,t}(u)=\frac{\cosh t-1}{\cosh t-\cos u},
\]
while for \(q=2\) the product identity gives
\[
 \phi_{2,t}(u)=\prod_{m\ge1}
 \frac{1+2e^{-(2m-1)t}\cos u+e^{-2(2m-1)t}}
 {(1+e^{-(2m-1)t})^2}.
\]
Both formulas show that \(\phi_{q,t}\) is nonincreasing on
\([0,\pi]\); hence \(\Phi_{q,t}(k)\le\phi_{q,t}(r)\). They also give
\begin{equation}\label{eq:parity-Phi}
 \Phi_{q,t}(k+\boldsymbol\pi)
 \ge\Phi_{q,t}(\boldsymbol\pi)\Phi_{q,t}(k).
\end{equation}
For \(q=1\), this follows by cross multiplication. For \(q=2\), the
difference for each factor, after multiplication by
\((1+e^{-(2m-1)t})^4\), is
\[
 4e^{-(2m-1)t}(1+e^{-2(2m-1)t})(1-\cos u)\ge0.
\]
Finite products and then their limit prove \eqref{eq:parity-Phi}. Since
\[
 1-\Phi_{q,t}(\boldsymbol\pi)=\frac{2}{R_{q,t}(0)}
 \sum_{\substack{x\in\mathbb Z^d\\|x|_1\text{ is odd}}}
 e^{-t|x|_q^q}>0,
\]
we may therefore put
\begin{equation}\label{eq:Q-bound}
 Q_{q,t}(k):=\frac{\Phi_{q,t}(k)-\Phi_{q,t}(k+\boldsymbol\pi)}
 {1-\Phi_{q,t}(\boldsymbol\pi)}\le\Phi_{q,t}(k).
\end{equation}

For \(0<t\le1\),
\begin{equation}\label{eq:basic-Fourier-bound}
 R_{q,t}(0)\ge ct^{-d/q},\qquad
 1-\Phi_{q,t}(k)\ge c\frac{r^2}{t^{2/q}+r^2},\qquad
 1-\Phi_{q,t}(\boldsymbol\pi)\ge c.
\end{equation}
For \(q=1\), these follow from \eqref{eq:R1-product},
\(1-\cos r\ge2r^2/\pi^2\), and \(\cosh t-1\le Ct^2\). For \(q=2\),
the first follows from \eqref{eq:R2-poisson}. If \(r\le\sqrt t\), then
\[
 1-\phi_{2,t}(r)=
 \frac{\sum_{n\in\mathbb Z}e^{-tn^2}(1-\cos(rn))}
 {\sum_{n\in\mathbb Z}e^{-tn^2}}
 \ge c\frac{r^2}{t},
\]
because \(\sum_{n\in\mathbb Z}e^{-tn^2}\le Ct^{-1/2}\), while the
terms with
\((2\sqrt t)^{-1}\le|n|\le t^{-1/2}\) have total weight at least
\(ct^{-1/2}\) and satisfy \(1-\cos(rn)\ge cr^2/t\). If
\(r\ge\sqrt t\), the factors in the product above are nonincreasing on
\([0,\pi]\), so \(1-\phi_{2,t}(r)\ge1-\phi_{2,t}(\sqrt t)\ge c\).
This proves the second estimate; the third is its value at
\(k=\boldsymbol\pi\). Also,
\begin{equation}\label{eq:large-t-bound}
 |R_{q,t}(k)-1|\le R_{q,t}(0)-1\le Ce^{-ct},\qquad t\ge1,
\end{equation}
because \(t|x|_q^q\ge t/2+|x|_q^q/2\) for \(x\ne0\).

\emph{Example~(C).}
Since \(R_{1,\gamma}(0)^{-1}=\tanh(\gamma/2)^d\),
\begin{equation}\label{eq:exp-symbols}
 \widehat w(k)=
 \frac{\Phi_{1,\gamma}(k)-\tanh(\gamma/2)^d}
 {1-\tanh(\gamma/2)^d},
 \qquad
 \widehat{w^{\mathrm{o}}}(k)=
 \frac{\Phi_{1,\gamma}(k)-\Phi_{1,\gamma}(k+\boldsymbol\pi)}
 {1-\Phi_{1,\gamma}(\boldsymbol\pi)}.
\end{equation}
Since \(\Phi_{1,\gamma}\ge0\),
\((2\pi)^{-d}\int_{[-\pi,\pi]^d}\widehat w(k)\,\mathrm dk=w(0)=0\),
and \(\Phi_{1,\gamma}(\boldsymbol\pi)=\tanh(\gamma/2)^{2d}\), we obtain
\[
 -\frac{\tanh(\gamma/2)^d}{1-\tanh(\gamma/2)^d}
 \le\min_{k\in[-\pi,\pi]^d}\widehat w(k)\le0,
 \qquad \widehat w(\boldsymbol\pi)=-\tanh(\gamma/2)^d.
\]
For fixed \(k\ne0\), \eqref{eq:R1-product} gives
\(\widehat w(k)\to0\); the odd transform tends to zero outside
\(\{0,\boldsymbol\pi\}\). The first transform in
\eqref{eq:exp-symbols} is at most \(\Phi_{1,\gamma}(k)\) because
\(\Phi_{1,\gamma}\le1\), and the second is at most this quantity by
\eqref{eq:parity-Phi}. Hence, for \(h\in\{w,w^{\mathrm{o}}\}\),
\[
 \frac1{1-\widehat h(k)}\le C(1+r^{-2})
\]
by \eqref{eq:basic-Fourier-bound}. This is integrable for \(d\ge3\), so
dominated convergence in \eqref{eq:green-Fourier-appendix} proves both
Green-function limits.

{
\emph{Example~(D).}
Fix $q\in\{1,2\}$ and $0<\delta<d$, and let $d<s\le d+\delta$. Put
\[
F_{s,q}(k)=\sum_{x\ne0}|x|_q^{-s}e^{\mathrm ik\cdot x}
=\frac1{\Gamma(s/q)}\int_0^\infty
t^{s/q-1}(R_{q,t}(k)-1)\,\mathrm dt.
\]
Equations~\eqref{eq:R1-product}, \eqref{eq:R2-poisson} and \eqref{eq:large-t-bound} give
\[
c(s-d)^{-1}\le F_{s,q}(0)\le C(s-d)^{-1},
\qquad \inf_kF_{s,q}(k)\ge-C.
\]
For fixed $k\ne0$, the same formulas show that $F_{s,q}(k)$ remains bounded as $s\downarrow d$: for $q=1$ one has $R_{1,t}(k)\le C_k t^{2-d}$ for $0<t\le1$, while for $q=2$ one has $R_{2,t}(k)\le C_k t^{-d/2}e^{-c_k/t}$. Hence
\[
\widehat w(k)=\frac{F_{s,q}(k)}{F_{s,q}(0)}\to0,
\qquad
\widehat{w^{\mathrm o}}(k)=
\frac{F_{s,q}(k)-F_{s,q}(k+\boldsymbol\pi)}
{F_{s,q}(0)-F_{s,q}(\boldsymbol\pi)}\to0
\]
outside $\{0\}$ and $\{0,\boldsymbol\pi\}$, respectively, and $\min_k\widehat w(k)\to0$.
For either $h=w$ or $h=w^{\mathrm o}$, restrict its Laplace expression for $1-\widehat h(k)$ to $0<t<r^q$. Equations~\eqref{eq:basic-Fourier-bound} and \eqref{eq:Q-bound} then give
\[
1-\widehat h(k)\ge c(s-d)
\int_0^{r^q}t^{(s-d)/q-1}\,\mathrm dt
\ge c r^{s-d}\ge c r^\delta.
\]
The integrable majorant $C r^{-\delta}$ permits dominated convergence in \eqref{eq:green-Fourier-appendix}, proving the two Green-function limits.
}

\emph{Example~(E).}
Keep \(s>d\) fixed and put
\[
 F_\gamma(k)=\sum_{x\in\mathbb Z^d}
 (1+\gamma|x|_1)^{-s}e^{\mathrm i k\cdot x}.
\]
The scalar Laplace formula and absolute convergence give
\begin{equation}\label{eq:regularized-Laplace}
 F_\gamma(k)=\frac1{\Gamma(s)}\int_0^\infty
 u^{s-1}e^{-u}R_{1,\gamma u}(k)\,\mathrm du.
\end{equation}
Thus \(F_\gamma(k)\ge0\) and
\(0<F_\gamma(\boldsymbol\pi)<1\). Since
\(\gamma^dR_{1,\gamma u}(0)\to2^du^{-d}\) and, for
\(0<\gamma\le1\), is at most \(C(1+u^{-d})\), dominated convergence
gives
\begin{equation}\label{eq:F0-asymptotic}
 \gamma^dF_\gamma(0)\to\frac{2^d\Gamma(s-d)}{\Gamma(s)}.
\end{equation}
The Fourier transforms are
\[
 \widehat w(k)=\frac{F_\gamma(k)-1}{F_\gamma(0)-1},
 \qquad
 \widehat{w^{\mathrm{o}}}(k)=
 \frac{F_\gamma(k)-F_\gamma(k+\boldsymbol\pi)}
 {F_\gamma(0)-F_\gamma(\boldsymbol\pi)}.
\]
Consequently,
\[
 -\frac1{F_\gamma(0)-1}\le
 \min_{k\in[-\pi,\pi]^d}\widehat w(k)\le0.
\]
For fixed \(k\ne0\), dominated convergence in
\[
 \gamma^dF_\gamma(k)=\frac1{\Gamma(s)}\int_0^\infty
 u^{s-1}e^{-u}[\gamma^dR_{1,\gamma u}(0)]
 \Phi_{1,\gamma u}(k)\,\mathrm du
\]
gives \(\gamma^dF_\gamma(k)\to0\), because
\(\Phi_{1,\gamma u}(k)\to0\) for every \(u>0\) and the majorant is
\(Cu^{s-1}e^{-u}(1+u^{-d})\). Hence
\(\widehat w(k)\to0\) for \(k\ne0\), including
\(k=\boldsymbol\pi\), and \(\widehat{w^{\mathrm{o}}}(k)\to0\) outside
\(\{0,\boldsymbol\pi\}\).

From \eqref{eq:regularized-Laplace}, \eqref{eq:F0-asymptotic}, and
\eqref{eq:basic-Fourier-bound},
\[
\begin{aligned}
 1-\widehat w(k)
 &=\frac{\int_0^\infty u^{s-1}e^{-u}R_{1,\gamma u}(0)
 (1-\Phi_{1,\gamma u}(k))\,\mathrm du}
 {\Gamma(s)(F_\gamma(0)-1)}\\
 &\ge c\int_0^1u^{s-d-1}
 \frac{(r/\gamma)^2}{u^2+(r/\gamma)^2}\,\mathrm du
 \ge c\min\{1,(r/\gamma)^{\min\{s-d,2\}}\}.
\end{aligned}
\]
The last inequality follows by integrating over \((0,r/\gamma)\) when
\(r\le\gamma\) and \(s-d\le2\), over \((1/2,1)\) when
\(r\le\gamma\) and \(s-d>2\), and over \((0,1)\) when
\(r\ge\gamma\). For the odd restriction, the exact identity is
\[
 1-\widehat{w^{\mathrm{o}}}(k)=
 \frac{\int_0^\infty u^{s-1}e^{-u}[R_{1,\gamma u}(0)-
 R_{1,\gamma u}(\boldsymbol\pi)](1-Q_{1,\gamma u}(k))\,\mathrm du}
 {\int_0^\infty u^{s-1}e^{-u}[R_{1,\gamma u}(0)-
 R_{1,\gamma u}(\boldsymbol\pi)]\,\mathrm du}.
\]
Here \(1-Q_{1,\gamma u}(k)\ge1-\Phi_{1,\gamma u}(k)\), while for
\(0<u<1\) and small \(\gamma\),
\(R_{1,\gamma u}(0)-R_{1,\gamma u}(\boldsymbol\pi)
\ge cR_{1,\gamma u}(0)\); the denominator is at most
\(\Gamma(s)F_\gamma(0)\). Thus the same lower bound holds with
\(w^{\mathrm{o}}\) in place of \(w\). Hence, for
\(h\in\{w,w^{\mathrm{o}}\}\),
\[
 \frac1{1-\widehat h(k)}
 \le C\left(1+r^{-\min\{s-d,2\}}\right).
\]
This is integrable because \(\min\{s-d,2\}<d\): for \(d\le2\) this
follows from \(s-d<d\), and for \(d\ge3\) from \(2<d\). Dominated
convergence proves the two Green-function limits and completes the
proof.
\end{proof}

\begin{proof}[Proofs of the exponential and polynomial corollaries]
Admissibility is proved in Section~\ref{sect:proofOSpositivity}.
In Example~\emph{(C)},
\[
\begin{aligned}
 g_{d,w_{\rho,\gamma}}
 &=\frac{g_{d,w_{0,\gamma}}}{1-\rho},\\
 a_{w_{\rho,\gamma}}
 &=\max\left\{0,-\rho-(1-\rho)
 \min_{k\in[-\pi,\pi]^d}\widehat w_{0,\gamma}(k)\right\},\\
 b(w_{\rho,\gamma})
 &=\frac{1-\rho}{2}
   (1-\widehat w_{0,\gamma}(\boldsymbol\pi)),\\
 q_{w_{\rho,\gamma}}
 &=\min\left\{1,
 \frac{2e\rho}
 {(1-\rho)(1-\widehat w_{0,\gamma}(\boldsymbol\pi))}\right\}.
\end{aligned}
\]

Fix \(\rho<1\), and in Example~\emph{(E)} also fix \(s\).
Lemma~\ref{lem:examples-C-E-fourier} shows that the RLP lower bound
tends to
$
 \frac{1-2\rho}{1-\rho},
$
which is positive exactly when \(\rho<1/2\). The MDD lower bound tends
to
$
 2-\min\left\{1,\frac{2e\rho}{1-\rho}\right\}
 -\frac1{1-\rho},
$
which is positive exactly when
\(0\le\rho<1/[2(e+1)]=\rho_*^{\mathrm{mdd}}\). By the definition of
convergence, for each fixed \(\rho<\rho_*^{\mathrm{mdd}}\) both lower
bounds are positive for all sufficiently small \(\gamma\) in
Examples~\emph{(C)} and~\emph{(E)}, and for all \(s>d\) sufficiently
close to \(d\) in Example~\emph{(D)}. The RLP lower bound alone has
the same property for each fixed \(\rho<1/2\).
Theorem~\ref{thm:longrangeorder2} proves these assertions.

The bipartite projections are \(w_{0,\gamma}^{\mathrm{o}}\),
\(w_{0,s,q}^{\mathrm{o}}\), and \(w_{0,s,\gamma}^{\mathrm{o}}\),
respectively, and do not depend on \(\rho\). The fourth estimate in
\eqref{eq:examples-C-E-limits} makes their Green functions smaller than
\(4/3\) for all sufficiently small \(\gamma\) in
Examples~\emph{(C)} and~\emph{(E)}, and for $s>d$ sufficiently close to $d$ in Example~\emph{(D)}. Here \(s\) in
Example~\emph{(E)} is fixed in the stated range{.}
Theorem~\ref{thm:longrangeorder} gives the remaining conclusions.
\end{proof}

\subsection{Proof of Gaussian domination}

\begin{lem}
\label{lem:Zh+-}
Let \(\Theta\) be a mid-edge reflection. For
\(h\in\mathbb R^{\T_L}\), define
\[
h_x^{\pm,\Theta}:=
\begin{cases}
h_x,&x\in\T_L^\pm,\\
h_{\Theta x},&x\in\T_L^\mp.
\end{cases}
\]
Then
\[
\lvert Z_A(h)\rvert^2
\leq Z_A(h^{+,\Theta})Z_A(h^{-,\Theta}).
\]
\end{lem}

\begin{proof}
Recall that the unnormalised functional
\(\langle\cdot\rangle_{1,A}\) introduced in the proof of
Proposition~\ref{prop:reflectionpos} is reflection positive. Hence
\[
(f,g)_\Theta:=
\left\langle f\,\overline{\Theta g}\right\rangle_{1,A}
\]
{defines a positive-semidefinite Hermitian form on $\mathcal A^+$.}
Let \(P,Q\in\mathcal A^+\) be bounded, and let
\((C_r)_{r=1}^m\) and \((D_r)_{r=1}^m\) be finite families of bounded functions in
\(\mathcal A^+\). Expanding the mixed exponential and applying
Cauchy--Schwarz gives
\begin{equation}
\label{eq:exp-reflection}
\begin{aligned}
&\left\lvert
\left\langle
\exp\left\{P+\overline{\Theta Q}
+\sum_{r=1}^m C_r\,\overline{\Theta D_r}\right\}
\right\rangle_{1,A}
\right\rvert^2
\\
&\quad\leq
\left\langle
\exp\left\{P+\overline{\Theta P}
+\sum_{r=1}^m C_r\,\overline{\Theta C_r}\right\}
\right\rangle_{1,A}
\left\langle
\exp\left\{Q+\overline{\Theta Q}
+\sum_{r=1}^m D_r\,\overline{\Theta D_r}\right\}
\right\rangle_{1,A}.
\end{aligned}
\end{equation}
Indeed, after expanding, one first applies Cauchy--Schwarz to the
reflection form and then to the sum of its coefficients.

We now use the spectral decomposition from the proof of
Proposition~\ref{prop:reflectionpos}. Set, for
\(\alpha=1,\ldots,N\) and \(i=1,2\),
\[
B_{\alpha,h}^i
:=B_\alpha^i+\delta_{i,1}
\sum_{x\in\T_L^+}\varphi_\alpha(x)h_x.
\]
The part of the exponent in \eqref{eq:centralquantity} coupling the two
halves is
\[
\sum_{\alpha=1}^N\sum_{i=1}^2
\lambda_\alpha B_{\alpha,h}^i\,
\overline{\Theta\!\left(B_{\alpha,h\circ\Theta}^i\right)},
\]
while all remaining terms are supported entirely in one of the two
halves. Apply \eqref{eq:exp-reflection} with the two families
\[
\left(\sqrt{\lambda_\alpha}B_{\alpha,h}^i\right)_{\alpha,i},
\qquad
\left(\sqrt{\lambda_\alpha}B_{\alpha,h\circ\Theta}^i\right)_{\alpha,i}.
\]
The two diagonal expressions are precisely
\(Z_A(h^{+,\Theta})\) and \(Z_A(h^{-,\Theta})\), respectively, and
are nonnegative by reflection positivity. This proves the claim.
\end{proof}

\begin{proof}[Proof of Proposition~\ref{prop:gaussian domination}]
The change of variables \(\boldsymbol s\mapsto-\boldsymbol s\) in
\eqref{eq:centralquantity} shows that \(Z_A(h)\in\mathbb R\) for real
\(h\). Since the spins are bounded, \eqref{eq:centralquantity} implies
\[
\lvert Z_A(h)\rvert
\leq C\exp\left\{-\tfrac12\sum_{\{x,y\}\in E}w_{x,y}(h_x-h_y)^2
+C\sqrt{\sum_{\{x,y\}\in E}w_{x,y}(h_x-h_y)^2}\right\}.
\]
Since $Z_A(h+c\mathbf1)=Z_A(h)$, restrict to the mean-zero subspace
$\{h:\sum_xh_x=0\}${. Irreducibility makes the positive-weight torus
graph connected, so the quadratic form
$\sum_{\{x,y\}\in E}w_{x,y}(h_x-h_y)^2$ is positive definite on this
subspace. The displayed upper bound therefore tends to zero as
$\lVert h\rVert\to\infty$ there. Since $Z_A(0)>0$, continuity implies
that} $\lvert Z_A\rvert$ attains a maximum. Among its maximisers choose \(h^\star\) minimising
\[
N(h):=\#\bigl\{\{x,y\}:x,y\in\T_L\text{ are nearest neighbours and }
h_x\neq h_y\bigr\}.
\]
If \(N(h^\star)>0\), reflect through the midpoint of a nearest-neighbour
edge on which \(h^\star\) is not constant. Lemma~\ref{lem:Zh+-} and
maximality give
\[
\lvert Z_A(h^\star)\rvert^2
\leq Z_A(h^{\star,+,\Theta})Z_A(h^{\star,-,\Theta})
\leq \lvert Z_A(h^\star)\rvert^2.
\]
Hence both reflected fields are again maximisers after subtracting their respective means; this changes neither $Z_A$ nor $N$. On the other hand,
at least one of them has strictly fewer discontinuous nearest-neighbour
edges than \(h^\star\), a contradiction. Therefore \(N(h^\star)=0\),
so \(h^\star=c\mathbf 1\) for some \(c\in\mathbb R\). Since
\eqref{eq:centralquantity} depends on \(h\) only through differences,
\(Z_A(c\mathbf 1)=Z_A(0)\). Consequently,
\[
Z_A(h)\leq\lvert Z_A(h)\rvert\leq Z_A(0),
\]
as claimed; see also \cite[Proposition~10.27]{Velenik}.
\end{proof}

\section*{Acknowledgements} 
AK’s research was partially supported by the Cusanuswerk. 
AK would also like to thank Volker Betz for making the {research visit to Sapienza Universit\`a di Roma} possible.
The work of W.W. was partially supported by MOST grant 2021YFA1002700.

\end{document}